\documentclass[reqno]{amsart}
\usepackage{amssymb, amsfonts, amsmath, amsthm}
\usepackage{remark}
\usepackage{enumerate}
\usepackage{xypic}
\usepackage{tikz-cd}
\usepackage{mathrsfs} 
\usepackage[normalem]{ulem}
\usepackage{hyperref}

\newtheorem{thm}{Theorem}[section]
\newtheorem{theorem}[thm]{Theorem}
\newtheorem{lemma}[thm]{Lemma}
\newtheorem{corollary}[thm]{Corollary}
\newtheorem{proposition}[thm]{Proposition}
\newremark{definition}[thm]{Definition}
\newremark{remark}[thm]{Remark}
\newremark{example}[thm]{Example}
\newremark{notation}[thm]{Notation}
\newremark{algo}[thm]{Algorithm}

\newcommand{\disc}{\mathrm{disc}}
\newcommand{\ord}{\operatorname{ord}}
\newcommand{\Gal}{\mathrm{Gal}}
\newcommand{\Img}{\mathrm{Im}\, }
\newcommand{\Spec}{\operatorname{Spec}}
\newcommand{\Spf}{\operatorname{Spf}}

\newcommand{\Pic}{\mathrm{Pic}}

\newcommand{\Res}{\mathrm{Res}}
\newcommand{\Jac}{\mathrm{Jac}}
\newcommand{\Z}{\mathbb Z}
\newcommand{\Q}{\mathbb Q}
\newcommand{\cO}{\mathcal O}
\newcommand{\FF}{\mathbb F}
\newcommand{\PP}{\mathbb P}

\newcommand{\p}{\mathfrak p}

\newcommand{\wW}{\widetilde{W}}
\newcommand{\whW}{\widehat{W}}

\newcommand{\wZ}{\widetilde{Z}}
\newcommand{\whZ}{\widehat{Z}}
\newcommand{\chara}{\mathrm{char}}
\newcommand{\cont}{\mathrm{cont}}

\newcommand{\Gl}{\mathrm{Gl}}
\newcommand{\cA}{\mathcal{A}}
\newcommand{\cE}{\mathcal{E}}
\newcommand{\cC}{\mathcal{C}}

\newcommand{\caD}{\mathcal{D}}

\newcommand{\cX}{\mathcal{X}}
\newcommand{\cZ}{\mathcal{Z}}
 
\newcommand{\cK}{\mathcal K}
\newcommand{\qi}{\langle \sigma \rangle}

\newcommand{\tp}{\tilde{p}}

\newcommand{\can}{\mathrm{can}}
\newcommand{\tI}{{\rm I}}
\newcommand{\tII}{{\rm II}}
\newcommand{\tIII}{{\rm III}}
\newcommand{\tIV}{{\rm IV}}

\begin{document}

\title[Minimal Weierstrass models]{\bf Minimal Weierstrass models and
regular models of hyperelliptic curves}
\author{Qing Liu}

\address{Univ. Bordeaux, CNRS, IMB, UMR 5251, F-33400 Talence, France}
\email{qing.liu@math.u-bordeaux.fr}

\date{} 

\dedicatory{Dedicated to the memory of Jean Fresnel,  1939-2025}

\begin{abstract} Let $C$ be a hyperelliptic curve of genus $g\ge 2$
  over a discrete valuation field $K$ with perfect residue field. We
  study the minimal Weierstrass models of $C$. When there is more
  than one such model, we find interesting properties on the minimal
  regular model and the canonical model of $C$.  
  For curves of genus 2, we characterize the existence of stable reduction using these minimal Weierstrass models. Furthermore, when $g=2$ and there is more than
  one minimal Weierstrass model, we can compute the Euler factor, Tamagawa
  number, and conductor of $\Jac(C)$, as well as a volume form for its 
  N\'eron model, by utilizing two specific minimal Weierstrass models.
  \end{abstract}

\maketitle

\tableofcontents

Let $\cO_K$ be a discrete valuation ring with perfect residue field $k$
and field of fractions $K$. Let $C$ be a hyperelliptic curve of genus $g\ge 2$
over $K$. Many local arithmetic invariants of $C$ and of the Jacobian
$\Jac(C)$ of $C$, such as the Euler factor (when $\cO_K$ is a local ring of
a ring of algebraic integers), the Tamagawa number and the volume form 
can be determined using a regular model of $C$.
But from the computational point of view, it is not easy to find
a regular model (see however \cite{Mus} and
\cite{Ld}), and the global structure of a regular 
model is rather complicated, as its closed fiber can have
many irreducible components. 
On the other hands, Weierstrass models (see \S~\ref{basics}) 
are very simple objects because they are defined by a single equation. 
Our main goal is to get as much information as possible
using only Weierstrass models, and especially the minimal
Weierstrass models (\S~\ref{basics}). It turns out this is
possible when $g$ is even and (more importantly) when $C$ has, up to isomorphisms,
more than one minimal Weierstrass model.  
When this is the case, a first natural question is how
they are related to each other and how many they can be. 

\begin{theorem}[Corollary~\ref{cor:equa_term}, see also Theorem~\ref{chain-mwm}] 
\label{thm:01}  Suppose $g$ is even and that $C$ has more than one minimal  
  Weierstrass model. Then there exists an even $n\ge 2$ and
  a Weierstrass equation
  \[ y^2+Q(x)y=P(x) \] 
of $C$ such that the minimal Weierstrass models of $C$ are given by the
equations (for even $0\le i\le n$):   
  \begin{equation} \label{eq:mini_W_i} 
y_i^2+ \pi^{-(g+1)i/2} Q(\pi^i x_i) y_i= 
\pi^{-(g+1) i} P(\pi^i x_i) 
  \end{equation} 
with $x_i=x/\pi^{i}$ and $y_i=y/\pi^{(g+1)i/2}$. 
\end{theorem}

When the conditions of Theorem~\ref{thm:01} are satisfied, we denote
by $W_i$, for all even $0\le i\le n$,  the minimal Weierstrass model of $C$ defined by Equation~\eqref{eq:mini_W_i}. 

\begin{theorem}[Theorem \ref{thm:bound_n}] Keep the hypothesis of Theorem~\ref{thm:01}. Then the number $m_C$ of the minimal Weierstrass models (up to isomorphisms) satisfies
  \[
  m_C \le 1 + \frac{\nu(\Delta_{C})}{2g(g+1)} 
  \]
  where $\nu(\Delta_{C})$ is the valuation of the
  discriminant of any minimal Weierstrass 
  model of $C$. This inequality is sharp. 
\end{theorem}

Now what can be said about the regular models of $C$, based 
on information from the minimal Weierstrass models ? Consider
the smallest model $W_0\vee W_n$ of $C$ dominating both $W_0$ and
$W_n$ (Definition~\ref{dfn:ccW}). Then the 
strict transforms of $(W_0)_k, (W_n)_k$ intersect at a single rational
point $p_{0,n}$ (\S~\ref{subsect:mrm}). 

\begin{theorem}[Theorem~\ref{regular-even}]
  Keep the hypothesis of Theorem~\ref{thm:01}. 
  \begin{enumerate}[\rm (1)] 
  \item The minimal regular model  of $C$ over $\cO_K$
    dominates $W_0\vee W_n$.  
\item  The model $W_0\vee W_n$ is semi-stable at $p_{0,n}$, and this possibly
  singular point is solved by a chain of $(n/2)-1$ projective lines over $k$, 
  each of them being the strict transform of $(W_i)_k$ for an even $2\le i\le n-2$.  
    \end{enumerate} 
\end{theorem}

We see that the Weierstrass models $W_0, W_n$ contain the most important
information. Let $p_0\in (W_0)_k$, $p_n^*\in (W_n)_k$ be the images of
$p_{0,n}$ above. In \S~\ref{subsect:rig} when $K$ is complete, we see that
as rigid analytic curve over $K$, 
$C$ is the union of the affinoid curves
$U_0:=(\widehat{W}_0 \setminus \{ p_0 \})\otimes K$,  
$U_n:=(\widehat{W}_n \setminus \{ p_n^* \})\otimes K$ and
an open annulus $\{ z\in K \ |  \ |\pi^n|< |z| < 1 \}$. Here 
$\widehat{-}$ means the formal completion with respect to the closed fiber,
and $-\otimes K$ means the rigid analytic generic fiber. One can embed $U_i$ into a
smooth projective curve $C_i$ of genus $g/2$ over $K$ (Theorem~\ref{thm:FM}).
The restriction to $U_i$ of the hyperelliptic involution  of $C$ can
be extended to a hyperelliptic involution on $C_i$ if $\chara(k)\ne 2$
(Proposition~\ref{prop:tame}). Otherwise
a necessary and sufficient condition for the existence of the extension 
is given in Proposition~\ref{prop:extend_h}. 
As an application, we have : 

\begin{corollary}[Corollary \ref{cor:gl}] Keep the hypothesis of Theorem~\ref{thm:01}. 
  Then $\cC_k$ is union of two projective curves $X_0, X_n$ over $k$, attached
  together by a chain of $(n/2)-1$ projective lines over $k$, and such that
  $X_i$ is the closed fiber of the minimal regular model of a projective
  smooth curve $C_i$ of genus $g/2$ over the completion $\hat{K}$ of $K$.
  If $\chara(k)\ne 2$ or if the conditions of Proposition~\ref{prop:extend_h} 
  for $p_0, p_{n}^*$ are satisfied, one can take $C_i$ hyperelliptic. 
\end{corollary}

The canonical model $\cC^{\can}$ of $C$ (\S~\ref{subsect:cm}) is obtained by
contracting the $\PP^1_k$'s of self-intersection $-2$ in $\cC_k$. Its closed
fiber has at most $2g-2$ irreducible components (while $\cC_k$ can
have arbitrary large number of irreducible components),  
but it still contains useful information. Indeed, under mild conditions 
(satisfied under the hypothesis of Theorem~\ref{thm:01}) 
we can compute the volume form of the Jacobian of $C$ with $\cC^{\can}$ and 
the Euler factor with $\cC^{\can}_k$, see Proposition~\ref{prop:volume_can} and
\S~\ref{subsect:EF}.   

\begin{proposition}[Proposition \ref{prop:positive}]
Keep the hypothesis of Theorem~\ref{thm:01}. Let $\cC^{\can}$ be the
canonical model of $C$. Then $\cC^{\can}$ dominates $W_0$ and $W_n$. 
\end{proposition}

When $g=2$, the above results are particularly interesting.

\begin{theorem}[See Theorem \ref{K1-K2}] \label{thm:07} Let $g=2$. Then 
  $C$ has more than one minimal Weierstrass model if and only
  if $\cC_k$ has type $[\cK_1-\cK_2-m]$ over $k$ (with $m\ge 1$), 
  where $\cK_1, \cK_2$ are Kodaira symbols for the reduction
  of elliptic curves over $K$. The integer $m$ is equal to 
  $n/2$. 
\end{theorem}

As an application, we can characterize the existence of the stable
reduction over $K$. Note that a curve $C$ has stable
reduction over $K$ if and only if $\cC^{\can}$ is stable and it is then the stable
model of $C$. 

\begin{theorem}[See Theorem~\ref{thm:stableg2}] Let $g=2$.
    \begin{enumerate}[\rm (1)] 
  \item The curve $C$ has stable reduction of type {\rm (I)-(IV)} 
    (i.e. the quotient of $\cC^{\can}_k$ by the hyperelliptic involution is
    isomorphic to $\PP^1_k$) if and only if $C$ has a unique minimal
    Weierstrass model $W$ and if the latter is semi-stable.  
  \item The curve $C$ has stable reduction equal to the union of two
    irreducible semi-stable curves of arithmetic genus $1$ if and only
    if $C$ has more than one minimal Weierstrass model and if
    $(W_0)_k, (W_n)_k$ are semi-stable away from $p_0$ and $p_{n}^*$. 
  \end{enumerate} 
\end{theorem}

Some local arithmetic invariants can be explicitly computed under the
hypothesis of Theorem~\ref{thm:07}

\begin{theorem}[Proposition \ref{prop:volume} and Theorem \ref{thm:g2-Euler}] 
  Suppose $g=2$ and that $C$ has more than one minimal
  Weierstrass models. Let $\cA$ be the N\'eron model of the Jacobian of $C$. 
\begin{enumerate}[\rm (1)] 
\item (Volume form) We
  have a canonical isomorphism
  \[ \det \omega_{\cA/\cO_K} \simeq \pi^{n/2} \det H^0(W_0, \omega_{W_0/\cO_K}) \] 
\item (Euler factors) Suppose moreover that $k$ is finite. Let $\Gamma_0$, $\Gamma_n$ be
  respectively the normalization of $(W_0)_k$ at $p_0$ and of $(W_n)_k$ at
  $p_{n}^*$. Then we have the equality 
  \[ P(\cA^0_k, T)= P(\Pic^0_{\Gamma_0/k}, T) P(\Pic^0_{\Gamma_n/k}, T) \]
  of the characteristic polynomials of the Frobenius of $k$
  acting on the Tate modules 
  of $\cA^0_k$, $\Pic^0_{\Gamma_0/k}$ and $\Pic^0_{\Gamma_n/k}$. 
\end{enumerate} 
\end{theorem}

Keep the hypothesis of Theorem~\ref{thm:07}. Then in $W_0$, there
is at most one singular point $w_0\ne p_0$.
We then have $\delta(w_0)\le 3$ and $w_0$ is rational over $k$.
For such a singular point (even if $g(C)>2$), the classical Tate's
algorithm can be used almost {\it mutatis mutandis} (\S~\ref{subsect:eTa}). Similarly for
$W_n$. This allows us to determine precisely $\cK_1, \cK_2$ over $k$. Once this
is done, we have the Tamagawa number and the conductor of $\Jac(C)$ for
free. 

\begin{theorem}[Propositions~\ref{prop:tkk}, \ref{prop:cond}] Suppose $g=2$ and
  $\cC_k$ has type $[\cK_1-\cK_2-m]$, $m\ge 0$ over $k$. Then  
  \begin{enumerate}[\rm (1)] 
  \item The group of connected components of $\cA_k$
    is isomorphic to the product $\Phi_1\times \Phi_2$,  
    where $\Phi_i$ is the group of components given by $\cK_i$.
  \item The conductor $f_C$ of $\Jac(C)$ is given by the formula
    \[
\nu(\Delta_C) = f_C + N-1+ 12m, 
\]
where $N$ is the number of irreducible components of $\cC_{\bar{k}}$. 
  \end{enumerate}  
\end{theorem}

For simplicity, in the introduction here we only present the results when $C$
has more than one minimal Weierstrass. In fact, if we have a Weierstrass
model $W$ such that $\delta(p)\le 3$ (see the discussions above
Lemma~\ref{lem:change_y} for the definition of $\delta(p)$ and $\lambda(p)$) 
for all $p\in W_k$, then the minimal 
desingularization of $W$ can be obtained by performing repeatedly the following
operations: find rational points $p\in W_k(k)$ with $\lambda(p)\ge 2$, 
replace $W$ by $W(p)$ if necessary (\S~\ref{subsect:eTa}).  The exceptional
locus of the minimal desingularization is described in terms of the Kodaira symbols
for the reduction of elliptic curves.  The arithmetic invariants such as
Euler factor, Tamagawa number, conductor and volume form
are then easily derived. See \S~\ref{sect:g2} for precise statements. 

As an illustration, we consider in \S~\ref{exp:22}
the modular curve $X_0(22)$ over $\mathbb Q$. We
give the Euler factors and the Tamagawa numbers at the primes of bad reduction
$p=2, 11$, and we determine a volume form of the N\'eron model of the Jacobian of
$X_0(22)$. 
\smallskip

From the practical point of view, given an arbitrary Weierstrass model $W$ of $C$,
there is a simple criterion, in terms of the multiplicities $\lambda(p)$ for
$p\in W_k(k)$, to decide whether $W$ is minimal and, when $W$ is minimal,
whether it is the unique one (\cite{LRN}, Propositions 4.3 and 4.6).
In \cite{LRN} we gave an algorithm to find a minimal Weierstrass
model. In the present work we provide an algorithm \ref{algo:mwm},
also based on the computation of the multiplicities $\lambda$, 
giving all the minimal Weierstrass models of $C$, when $g$ is even. It has
been implemented in PARI by Bill Allombert. 
\medskip

Most of the results in this work are for even $g$. The case when $g$ is odd
is slightly more complicated and will be treated in a future work.   
\medskip

{\bf Notation} Throughout this paper, $K$ is a discrete valuation field with
valuation ring $\cO_K$ and perfect residue field $k$.
We denote by $\pi$ a uniformizing element of $\cO_K$,
$\nu$ or $\nu_K$ the normalized valuation ($\nu(\pi)=1$) on $K$. 

Unless specified otherwise,  $C$ is a hyperelliptic curve over $K$
with hyperelliptic involution $\sigma$; $\cC$ is 
the minimal regular model of $C$; $\cC^{\can}$ is the canonical model of $C$.
Weierstrass models are usually denoted by $W$, sometimes by $U$. 
\medskip

\noindent {\bf Acknowledgments} I would like to thank 
Bill Allombert for stimulating discussions, Michel Matignon and
Mohamed Sa\"idi for helpful conversations concerning the rigid analytic
structure of $C$, and Michael Stoll to pointing out \cite{MuSt} for
works relating to reduction of genus $2$ curves and for suggesting
a correct statement of Proposition~\ref{prop:tkkc}. 

\begin{section}{Minimal Weierstrass models} \label{sect:1}

We describe in this section the relation between the minimal
Weierstrass models of a given hyperelliptic curve $C$, when there is
more than one such model. The main result is Theorem~\ref{chain-mwm}. 

\subsection{Weierstrass models}\label{basics}  We recall some basic definitions
and results on Weierstrass models of hyperelliptic curves.  In this work
we only consider \emph{normal models}. 
A \emph{Weierstrass model $W$ of $C$ over $\cO_K$} is the normalization
of a smooth model $Z$ of $C/\qi=\PP^1_K$ in $K(C)$. Note that $Z\simeq \PP^1_{\cO_K}$.  A \emph{coordinate function of $Z$}
is a rational function $x\in K(Z)$ such that $Z\setminus \mathrm{div}_\infty(x)\simeq \Spec \cO_K[x]$. The pre-image of $Z\setminus \mathrm{div}_\infty(x)$ in $W$ is an affine scheme defined by an equation as \eqref{eq:start} below. By construction,  
$\sigma$ acts on $W$ and we have $W/\qi =Z$. 

Unless specified otherwise, we always consider equations 
\begin{equation}
  \label{eq:start}
  y^2+Q(x)y=P(x), \quad P, Q\in \cO_K[x] 
\end{equation}
with $\deg Q(x) \le g+1$ and $\deg P(x) \le 2g+2$. The
discriminant $\Delta_W$ of $W$ is the discriminant of the
equation~\eqref{eq:start} (\cite{LRN}, Definition 1.2). The 
valuation $\nu(\Delta_W)\in \mathbb N$ depends  only on $W$, not
on the choice of the equation.

A Weierstrass model $W$ is said to be \emph{minimal} if $\nu(\Delta_W)$
is minimal among all Weierstrass models of $C$. A criterion is
given in \cite{LTR}, Corollaire 2 (see also \cite{LRN}, Proposition 4.3) for
$W$ to be minimal, in terms of the multiplicity $\lambda(p_0)$ of 
rational points $p_0\in W_k(k)$ (\cite{LRN}, Definition 3.2, and
\cite{LTR}, D\'efinition 10 for all closed points $p_0\in W_k$).   
There is a treatment of minimal Weierstrass models when $\chara(k)\ne 2$ and
$C$ has semi-stable reduction in \cite{DDMM}, \S 17. 

For a given curve $C$ there are only finitely many
minimal Weierstrass models over $\cO_K$ up to isomorphisms
of models (\cite{LTR}, Corollaire 4).  
A criterion is given in \cite{LTR}, Corollaire 2 (see also \cite{LRN},
Prop. 4.6\footnote{In (2.a) therein, $W$ must be assumed to be
minimal.}) for the uniqueness of the minimal Weierstrass model. 

Let us recall (\cite{LTR}) the notation $\varepsilon(W)=0$ or $1$ depending on
whether $W_k$ is reduced or not. For $p_0\in W_k(k)$, a new Weierstrass model
$W(p_0)$ is defined (\cite{LTR}, D\'efinition 12, see also
Definition~\ref{dfn:ccW}(4)  below), and 
we have 
\begin{equation} \label{eq:disc-p0}
\left\lbrace\begin{matrix}   
\nu(\Delta_{W(p_0)})& = &
\nu(\Delta_{W})+2(2g+1)(g+1-2[\lambda(p_0)/2]),\\[0.8em] 
              \varepsilon(W(p_0))& =& \lambda(p_0)-2[\lambda(p_0)/2].\hfill
\end{matrix}
\right.
\end{equation}
(\cite{LTR}, Lemme 9(a) and \cite{LRN}, Remark 4.2).

\subsection{Chains of Weierstrass models}\label{chain-w} We start by
describing the relations between the smooth models of $\PP^1_K$. 
Let $Z$ be a smooth model of $\PP^1_K$ over $\cO_K$. Let $q_0\in Z_k(k)$.
Let $x$ be a coordinate function of $Z$ (\S \ref{basics}) such that
$x(q_0)=0$, we define $Z(q_0)$ as the smooth model with coordinate function $x/\pi$.
It is immediate to check that all coordinate functions vanishing at $q_0$
give the same $Z(q_0)$. Let $q_1^* \in Z(q_0)_k(k)$ be the pole of $x/\pi$. Then
$Z=Z(q_0)(q_1^*)$. 

We call a \emph{chain of smooth models of $\PP^1_K$} 
a sequence
\[ Z_0, Z_1, \dots, Z_n\]
of pairwise non-isomorphic smooth models of $\PP^1_K$ such that
for all $i\le n-1$, $Z_{i+1}=Z_i(q_i)$ for some $q_i\in (Z_i)_k(k)$.
Note that $Z_n, Z_{n-1}, \dots, Z_0$ is then also a chain. 

\begin{lemma}\label{lem:csm} 
  Let $Z, Z'$ be two distinct ({\it i.e.} non-isomorphic) smooth models
  of $\PP^1_K$ over $\cO_K$.
  \begin{enumerate}[\rm (1)] 
  \item (\cite{LTR}, \S 4.2) There exist $n\ge 1$ and a coordinate function of 
    $Z$ such that $x/\pi^n$ is a coordinate function of $Z'$.
  \item For all $i\le n$, let $Z_i$ be the smooth model of $\PP^1_K$ with 
    $x/\pi^i$ as a coordinate function. Then
$(Z_i)_{0\le i\le n}$ is a chain of smooth models of $\PP^1_K$ with
 $q_i\in (Z_i)_k$ the zero of $x/\pi^i$.
\item The above chain is the unique one starting with $Z$ and ending with $Z'$.
\end{enumerate}
\end{lemma}

\begin{proof} (1)-(2) Let $x$, $x'$ be respective coordinate functions of $Z$ and
  $Z'$. There exist $a, b, c, d\in \cO_K$, with $\gcd(a, b, c, d)=1$ such
  $x=(ax'+b)/(cx'+d)$. The integer $n=|\nu(ad-bc)|>0$ is independent
  on the choice of $x, x'$.  Denote it by $d(Z, Z')$. 
  Inverting $x$ if necessary   we can suppose that $x$ is a rational
  function on $Z'_k$. There exists 
  $r\ge 0$ such that $x/\pi^r \ne 0$ in $k(Z'_k)$. Replacing $x$ by 
  $x-\pi^r c$ for some $c\in \cO_K$  we can suppose that $r>0$.
  Let $q_0\in Z_k$ be the zero of $x$ and let $Z_1=Z(q_0)$. Then
  $d(Z_1, Z')=n-1$. If $n=1$, then $x/\pi$ is a coordinate function of
  $Z'$ and $Z'=Z(q_0)$. Otherwise we start again with the pair $Z_1, Z'$
  and the rational function $x/\pi$ on $Z'_k$.  By induction on $n$
  we construct the $Z_i$, $q_i$ and a coordinate function $x$ of $Z$
  such that $x/\pi^{i}$ is a coordinate function of $Z_i$ for all $i$
  (translating $x$ by a suitable constant in $\cO_K$ if necessary
  at each step). 

  (3) Let $U_0=Z, \ U_1, \ \dots, \ U_m = Z'$ be a chain
  with $U_{j+1}=U_j(u_j)$ for all $j\le m-1$. Let $t$ be a coordinate function
  of $Z$ such that $t(u_0)=0$. Then $u_1$ is not the pole of $t/\pi$
  in $(U_1)_k$ because otherwise $U_2=U_0$.  Replacing $t$ by $t-\pi c$ for
  some $c\in \cO_K$, we can suppose that $(t/\pi)(u_1)=0$. Repeating this
  process we find as in (1) a coordinate function $t$ of $Z$ such that
  $t/\pi^j$ is a coordinate function of $U_j$ for all $j\le m$. 
  There exist $r\in \Z$ and 
    \[
     \begin{pmatrix}
        a_{11} & a_{12} \\ a_{21} & a_{22}
    \end{pmatrix},
 \begin{pmatrix}
        b_{11} & b_{12} \\ b_{21} & b_{22}
    \end{pmatrix} \in \Gl_2(\cO_K) 
  \] 
  such that
  $\pi^{-n}x=(a_{11} \pi^{-m}t + a_{12})/(a_{21}\pi^{-m}t + a_{22})$ and 
  \[
\begin{pmatrix}
        a_{11}\pi^n & \pi^{n+m}a_{12} \\ a_{21} & \pi^m a_{22}
    \end{pmatrix}=\pi^r
 \begin{pmatrix}
        b_{11} & b_{12} \\ b_{21} & b_{22}
    \end{pmatrix}. 
  \] 
  As $\gcd(a_{11}, a_{21})=1$, this implies easily that
  $a_{11}, a_{22}\in \cO_K^*$, $a_{21}\in \pi\cO_K$ and $n=m$. 
 Hence $x(u_0)=0$ and $u_0=p_0$. Therefore $U_1=Z_1$. Again by induction we find
  $U_i=Z_i$ for all $i\le n-1$. 
  \end{proof}

  Let $Z, Z'$ be two smooth models of $\PP^1_K$.
With the notation of Lemma~\ref{lem:csm}(1), we denote by
\[
  Z\wedge Z'=\{ x=0 \} \subset Z_k, \quad
  Z'\wedge Z:=\{ (\pi^n/x)= 0 \} \subset Z'_k.  
\] 
We should think them as the
``\emph{intersection points}'' of $Z_k$ and $Z'_k$. From the
geometric point of view, there is a smallest model $Z\vee Z'$ of $\PP^1_K$
dominating $Z$ and $Z'$. The closed fiber of $Z\vee Z'$ is union of two
projective lines $\PP^1_k$ intersecting transversely at a rational
point $\tilde{q}_0$. Then the points of $Z\wedge Z'$ and
$Z'\wedge Z$ are the respective images of $\tilde{q}_0$ in $Z$ and $Z'$. 
See Figure~\ref{inter-ZZ'}.

\begin{figure}[h] 
 \centering 
\begin{tikzpicture}[scale=1] 
  \draw (0, 1.5) node [anchor=east]  {$\Omega$} -- (1.5, 0);
  \draw (1, 0) -- (2.5, 1.5) node [anchor=west] {$\Omega'$};
  \draw[->] (1.45,0.25) -- (5, 0.25 );
  \draw (-0.5, 0.5) node {$(Z\vee Z')_k$};
  \draw  (1, 0.25)  node [left] {$\tilde{q}_0$ };
    \draw (1.25, 0.25)  node {$\bullet$}; 
  \draw (4, 1.5) node [anchor=west] {$Z_k$} -- (5.5,0); 
  \draw (5.25, 0.25) node {$\bullet$};
  \draw (5.30, 0.25) node [right] {$x=0$};
  \coordinate (P) at (1.40, 0.23);
  \coordinate (Q) at (6.5, 0.25);
  \draw[->] (P) to [bend right] (Q); 
  \draw (6.75, 0.25) node [right] {$(\pi^n/x)=0$};
  \draw (6.70, 0.25) node {$\bullet$}; 
  \draw (6.5, 0) -- (7.5, 1.5) node [anchor=west] {$Z'_k$}; 
\end{tikzpicture} 
\caption{$\Omega$, $\Omega'$ are the respective strict transforms of
  $Z_k$ and $Z'_k$.}
\label{inter-ZZ'}
\end{figure}

We call the sequence $Z, Z_1, \dots, Z_{n-1}, Z'$ of Lemma~\ref{lem:csm}
the \emph{chain of smooth models} connecting $Z$ to $Z'$, and we denote by 
$d(Z, Z')$ the length $n$ of the chain. If $Z=Z'$, $d(Z, Z')=0$. 
Note that for all $0<i<n$, we have
\[ Z\wedge Z_i=Z\wedge Z', \quad Z'\wedge Z_i=Z'\wedge Z.\] 
The model $Z\vee Z'$ is semi-stable over $\cO_K$ with at most
a singular point $\tilde{q}$. The minimal resolution
$\widetilde{Z\vee Z'}$ of this singularity dominates $Z_i$ for all $0\le i\le n$,
the strict transform of $(Z_i)_k$ is isomorphic to $(Z_i)_k$ and 
$(\widetilde{Z\vee Z'})_k$ is 
a chain of $n+1$ projective lines over $k$. See Figure~\ref{tree}.

\begin{figure}[h] 
\begin{tikzpicture}[scale=1] 
  \draw (0, 1.5) node [anchor=east]  {$Z_k$} -- (1.5, 0);
  \draw (1, 0) -- (2.5, 1.5);
  \draw (3, 1.5) node {$(Z_1)_k$};
  \draw (1.5, 1.5) -- (3,0);
    \draw (2.5, 0) node {$(Z_2)_k$};
  \draw[dashed]  (2,0.5) -- (5, 0.5);
  \draw (4,0) -- (5.5,1.5);
    \draw (6.7, 0.1) node {$(Z_{n-1})_k$};
  \draw (4.5,1.5) -- (6,0);
  \draw (5.5, 0) -- (7,1.5) node [anchor=west] {$Z'_k$}; 
\end{tikzpicture} 
\caption{The closed fiber of $(\widetilde{Z\vee Z'})_k$.}
\label{tree} 
\end{figure}

\begin{lemma} \label{sum-dist} Let $Z, Z', Z''$ be three pairwise distinct
  smooth models of $\PP^1_K$. Then
  \[ d(Z', Z'') \le d(Z', Z)+d(Z, Z'') \]
  with equality if and only if $(Z\wedge Z') \cap (Z\wedge Z'')=\emptyset$.
    In particular $d(-, -)$ is a distance on the set of the isomorphism classes
  of the smooth models of $\PP^1_K$. 
\end{lemma}

\begin{proof} Let $(Z_i)_{0\le i\le n}$, $(U_j)_{0\le j\le m}$ 
 be the chains connecting $Z$ to respectively $Z'$ and $Z''$.
    Let $Z\wedge Z'=\{ q_0 \}$ and $Z\wedge Z''=\{ u_0\}$. 
  First suppose that $q_0\ne u_0$.  Then
  $Z''=U_m, U_{m-1}, \dots, U_0=Z_0, \dots, Z_n=Z$ is a chain of length $n+m$
  and we have the desired equality. 
  
  Suppose now $q_0=u_0$. Then $U_1=Z_1$ and
  \[ d(Z', Z)+d(Z, Z'')=d(Z_1, Z')+ d(Z_1, Z'')+2.\] 
  If $Z_1=Z'$ or $Z''$ then we have the
  desired strict inequality. Otherwise we proceed by induction on
  $\min\{ n, m\}$. 
  \end{proof}

\begin{definition} \label{dfn:ccW}
Let $W, W'$ be Weierstrass models of $C$.
Let $Z=W/\qi$ and  $Z'=W'/\qi$.
\begin{enumerate}[\rm (1)] 
\item Let $(Z_i)_{0\le i \le n}$ be the chain of smooth models of $C/\qi$ connecting
$Z$ to $Z'$. Let $W_i$ be the normalization of $Z_i$ in $K(C)$. We call
the sequence $(W_i)_{0\le i\le n}$ the \emph{chain of Weierstrass models connecting
  $W$ to $W'$}. 
We define the \emph{distance} $d(W, W'):=d(Z, Z')$. 
\item If $W\ne W'$, we denote by $W\wedge W'$ the pre-image of
$Z\wedge Z'$ by $W\to Z$. Note that $W\wedge W'\ne W'\wedge W$ and
for all $0<i<n$, we have
\[ W\wedge W_i=W\wedge W', \quad W'\wedge W_i=W'\wedge W.\] 
\item We denote by $W\vee W'$ the normalization of $Z\vee Z'$ in $K(C)$.
  This is then the
smallest model of $C$ dominating $W$ and $W'$. The intersection of the
strict transforms of $W_k$ and $W'_k$ is the pre-image of $W\wedge W'$
and of $W'\wedge W$.  
\item Let $p_0\in W_k$ be a point lying over a rational point $q_0\in Z_k(k)$. We 
define $W(p_0)$ as the normalization of $Z(q_0)$ in $K(C)$. This generalizes
slightly \cite{LTR}, D\'efinition 12 (where $p_0\in W_k(k)$).
\end{enumerate} 
\end{definition}

\begin{remark} Keep the above notation. 
  \begin{enumerate}[\rm (1)] 
  \item (\emph{Caution}) The notation $Z\wedge Z'$ and $W\wedge W'$ differ
 slightly from 
  \cite{LTR}, \S 4. The $Z\wedge Z'$ there is the $(Z\wedge Z')\cup (Z'\wedge Z)$
  here. Similarly for $W\wedge W'$. The new (non commutative) notation is
  simpler. 
\item The dual graph of $\widetilde{Z\vee Z'}_k$ is a simple chain 
  (Figure~\ref{tree}). Consider the minimal desingularization of
  $W\vee W'$ at the points lying over $W\wedge W'$. Then the same property
  holds if $g$ is even and $W, W'$ are minimal
  (Theorem~\ref{regular-even}). But this is not true in general.    
  \end{enumerate}
\end{remark}

\begin{lemma} \label{lem:Wp0} Let $W, W'$ be distinct Weierstrass models of $C$.
  \begin{enumerate}[\rm (1)] 
  \item Let $p_0\in W\wedge W'$. If $\delta(p_0)\ge 1$ (see the summary above Lemma~\ref{lem:change_y}), then $W\wedge W'$ is reduced to a
    single point of $W_k$, rational over $k$.
\item  We have $d(W, W')=1$ if and only if $W'=W(p_0)$ for some
  $p_0\in W_k$ lying over a rational point of $Z_k$. This then implies that
  $p_0\in W\wedge W'$. 
  \end{enumerate}  
\end{lemma}

\begin{proof} (1) The quotient morphism 
  $W_k\to Z_k$ is a cover of degree $2$ not \'etale above the point
  $q\in Z_k(k)$ of $Z\wedge Z'$ by Lemma~\ref{lem:d0}(1) if $W_k$ is reduced, 
  and $W_k\to Z_k$ is a universal homeomorphism if $W_k$ is non-reduced.
  So   $W\wedge W'$ consists in a single rational point. 

  (2) This follows from the equivalence
  $d(Z, Z')=1$ if and only if $Z'=Z(q_0)$ for some rational point $q_0\in Z_k$
  (Lemma~\ref{lem:csm}(3)). By definition $\{q_0 \}=Z\wedge Z'$. 
  \end{proof}

\begin{remark} \label{rmk:p0_nr} 
  Let $W$ be a Weierstrass model, $Z=W/\qi$ and let
  $q_0 \in Z_k(k)$. Suppose that the (set-theoretical) pre-image
  $W\times_Z \{ q_0\}$ of $q_0$ is not a single rational point. 
  Then $W_k \to Z_k$ is \'etale over $q_0$. So $W_k$ is reduced and
  for any $p_0\in W_k$ lying over $q_0$,  the multiplicities
$\lambda(p_0), \delta(p_0)$ (see below) are equal to $0$.  The
closed fiber $W(p_0)_k\simeq \PP^1_{k(p_0)}$, hence all points
of $W(p_0)_k$ are smooth over $k$ and therefore have multiplicities
$\lambda=\delta=0$. One can check directly that Formula~\eqref{eq:disc-p0}
holds:
\[
\nu(\Delta_{W(p_0)})=\nu(\Delta_{W})+2(2g+1)(g+1). 
\] 

Let $W'$ be another Weierstrass model and suppose that $W\wedge W'$ is not
a single rational point. Repeating the above discussions we find 
\[ \nu(\Delta_{W'})=\nu(\Delta_{W})+2(2g+1)(g+1)d(W, W').\]  
\end{remark}

Let $W_0, \dots, W_n$ be a chain of Weierstrass models.
We look for an equation of $W_0$ from which we can
deduce directly an equation of $W_i$ for all
$1\le i\le n$ (Proposition~\ref{prop:same-y}). 
\medskip 

\noindent {\bf Multiplicities $\lambda$ and $\delta$.} Let $W$ be a Weierstrass model
of $C$. Let us recall the definition
of the multiplicities $\delta(p_0), \lambda(p_0)$ for closed points $p_0\in W_k$
as given in \cite{LTR}, D\'efinitions 9 and 10. 

Let $y^2+Q(x)y=P(x)$ be an equation of $W$ as in Equation~\eqref{eq:start} and 
such that $x$ is regular at $p_0$. Set 
\begin{equation}
  \label{eq:def_delta}
  \delta_{p_0}(y):=\min\{ 2 \ord_{p_0}(\bar{Q}(x)),
  \ord_{p_0}(\bar{P}(x))\}  
\end{equation}
where $\ord_{p_0}(f(x))$ stands for the vanishing order of $f(x)\in k[x]$ at $p_0$. 
When $x(p_0)=0$, we write $\ord_0$ and $\delta_0(y)$
instead of $\ord_{p_0}$ and $\delta_{p_0}(y)$. 

If $W_k$ is reduced, $\delta(p_0)$ is the maximum of the $\delta_{p_0}(y)$'s for
all possible equations as above, that is,  we can change $y$ by $uy+F(x)$ with
$u\in \cO_K^*$ and $F(x)\in \cO_K[x]$ with $\deg F(x)\le g+1$. 
If $W_k$ is non-reduced, we only consider equations with $\pi \mid Q(x), P(x)$
and we let $\delta(p_0)=\ord_{p_0} \overline{(\pi^{-1}P)}(x)$ (it does not depend
on the choice of $P(x)$).

To define $\lambda(p_0)$, we suppose for simplicity that the image of $p_0$
in $W/\qi$ is rational over $k$ and refer to \cite{LTR}, D\'efinition 10
(with $r=1$ there) for the general case. We then translate $x$ to make
$x(p_0)=0$. We denote by $\mu_0$ the Gauss valuation on $K(x)$ with respect
to the variable $x/\pi$:
\[
\mu_0\left(\sum_j a_jx^j\right) = \min_j \{ \nu(a_j)+j \}
\]
and we put 
\[ \lambda_0(y):=\min\{ 2 \mu_0(Q(x)),
  \mu_0(P(x)) \}.\]
By definition $\lambda(p_0)=\max_y \lambda_0(y)$ for all possible
equations as above (we can change $y$ by $uy+F(x)$ as for the
definition of $\delta(p_0)$).
See also \cite{Ld}, Definition 2.3 in a more general context.  
This multiplicity has the property that $\lambda(p)\le 1$ if and only if $W$ is regular
at $p$ (\cite{LTR}, Lemme 7(a)). See also Lemma~\ref{lem:d0} for more
information.

\begin{lemma} \label{lem:change_y} Keep the above notation and suppose $\chara(k)=2$. 
  Write $P(x)=\sum_{j\ge 0} a_jx^j$. Let 
  \[
H(x)=\sum_{j\ge 0, \ \nu(a_{2j})\in 2\mathbb N}  e_jx^j, \quad \text{with } e_j^2\in {a}_{2j}(1+\pi \cO_K) 
  \]
  ($H(x)=0$ if the sum is on an empty set). Let $z=y+H(x)$ and let
  \[
z^2+R(x)z=S(x)
   \]
   be the equation of $z$.
\begin{enumerate}[\rm (1)] 
\item We have $\lambda_0(z) \ge \lambda_0(y)$ and the inequality is
  strict if $\lambda_0(y)<\lambda(p_0)$.
\item If $W_k$ is non-reduced, then $\bar{R}(x)=\bar{S}(x)=0$.
\item \label{crit_delta} Suppose that $W_k$ is reduced, and
$2\ord_0(\bar{Q}(x))\le \ord_0(\bar{P}(x))$ or 
  $\ord_0(\bar{P}(x))$ is odd, then $\delta_0(y)=\delta(p_0)$.
\item If $W_k$ is reduced, then $\delta_0(z)\ge \delta_0(y)$.
  The inequality is strict if $\delta_0(y)<\delta(p_0)$.  
\end{enumerate}
\end{lemma}

\begin{proof} We have $R(x)=Q(x)-2H(x)$ and $S(x)=P(x)+Q(x)H(x)-H(x)^2$. 

(1) By construction $\mu_0(H)\ge \mu_0(P)/2$.  
As $\mu_0$ is a valuation on $K(x)$, 
we check immediately that $\lambda_0(z)\ge \lambda_0(y)$. 
Suppose that $\lambda_0(y)<\lambda(p_0)$. By \cite{LRN}, Lemma 3.3,
$\mu_0(Q)> \mu_0(P)/2$ and $\nu(a_j)+j < \mu_0(P)$ for all odd $j$.
This implies that $\mu_0(H)=\mu_0(P)/2$. It is easy to check that 
$\lambda_0(z)> \lambda_0(y)$ (see the proof of \cite{LRN}, {\it loc. cit.}).  

(2) As $W_k$ is non-reduced, we have $\bar{Q}(x)=0$,  and $\bar{P}(x)$ is a
square, hence equal to $\bar{H}(x)^2$. Thus $\bar{R}(x)=\bar{Q}(x)=0$ and
$\bar{S}(x)=0$. 

(3) The proof is similar to (but much easier than) that of \cite{LRN}, Lemma 3.3. 

(4) We have $\ord_0(\bar{H}(x))\ge \ord_0(P(x))/2$. 
If $2\ord_0(\bar{Q}(x)) \le \ord_0(\bar{P}(x))$, then 
$\ord_0(\bar{S}(x))\ge 2\ord_0(\bar{Q}(x))=2\ord_0(\bar{R}(x))$.
So $\delta_0(z)=\delta(p_0)$ by (3). Suppose that 
$2\ord_0(\bar{Q}(x)) > \ord_0(\bar{P}(x))$,
then $\ord_0(\bar{S}(x))\ge\ord_0(\bar{P}(x))$. If moreover
$\delta_0(y)<\delta(p_0)$, then by (3) $\ord_0(\bar{P}(x))$ is even.
This implies that $\ord_0(\bar{S}(x))>\ord_0(\bar{P}(x))$, hence
$\delta_0(z)>\delta_0(y)$. 
\end{proof}

\begin{remark} \label{rmk:lambda'}
  Keep the notation of Lemma~\ref{lem:change_y}. 
  \begin{enumerate}[\rm (1)]
  \item Suppose that $\chara(k)=2$. We know by \cite{LRN}, Lemma 3.3, that
    if $2\mu_0(Q)\le \mu_0(P)$ or $\mu_0(P)=\nu(a_j)+j$ for some odd $j$, 
    then $\lambda(p_0)=\lambda_0(y)$.
    Conversely, Lemma~\ref{lem:change_y}(1) shows that one can always
    translate $Q(x)$ by 
    some $H(x)$ so that $2\mu_0(Q) \le \mu_0(P)$ or, 
    $2\mu_0(Q)> \mu_0(P)$ and $\mu_0(P)=\nu(a_j)+j$ holds
    only when $j$ or $\nu(a_j)$ is odd. 
\item\label{lambda_1} Let $\cO_{K'}$ be a discrete valuation ring dominating
  $\cO_{K}$, 
  with perfect residue field, and such that $\pi$ is a uniformizing element
  of $\cO_{K'}$. Let $y^2+Q(x)y=P(x)$ be an equation of $W$ satisfying
  the condition of (1) and Lemma~\ref{lem:change_y}(3).
  It is also an equation of $W':=W_{\cO_{K'}}$ (which is a 
  Weierstrass model of $C_{K'}$ over $\cO_{K'}$ by \cite{LTR}, Lemme 2). 
 Let $p'_0\in W'_{k'}$ be a point lying over $p_0$. Then
  by (1), $\lambda(p'_0)=\lambda_0(y)=\lambda(p_0)$ (see also
\cite{LTR}, Lemme 10; \cite{Ld}, Proposition 2.13).  
  \end{enumerate} 
\end{remark} 

\begin{proposition}\label{prop:same-y}
Let $W_0, \dots, W_n$ be a chain of Weierstrass models with $n\ge 0$.  
Let $x$ be a coordinate function of $W_0/\qi$ such that for all $0\le i\le n$, 
$x_i:=x/\pi^i$ is a coordinate function of $W_i/\qi$ (Lemma~\ref{lem:csm}(1)). 
Let $p_i\in W_i\wedge W_{i+1}$ (a zero of $x_i$). Then there exists an equation
  \[ y^2+Q(x)y=P(x) \]
  of $W_0$ such that for all $0\le i\le n$, 
\begin{enumerate}[\rm (1)] 
\item an equation of $W_i$ is
  \begin{equation} \label{eq:Wi} 
     y_i^2+Q_i(x_i)=P_i(x_i),
  \end{equation}
where, denoting by $m_i=\sum_{0\le j\le i-1} [\lambda(p_i)/2]$, 
\[   y_i=y/\pi^{m_i}, \ Q_i(x_i)=Q(\pi^i x_i)/\pi^{m_i}, \ 
  P_i(x_i)=P(\pi^i x_i)/\pi^{2m_i},\] 
\item $\lambda_0(y_i)=\lambda(p_i)$ (and when $\chara(k)=2$,
  Equation~\eqref{eq:Wi}  satisfies the condition of
  Remark~\ref{rmk:lambda'}(1)); 
\item $\delta_0(y_i)=\delta(p_i)$ if $(W_i)_k$ is reduced; 
\item $\bar{Q}_i(x_i)=\bar{P}_i(x_i)=0$ if $(W_i)_k$ is non-reduced. 
  \end{enumerate}  
\end{proposition} 

\begin{proof} If $\chara(k)\ne 2$, we just take $Q(x)=0$. Suppose
  $\chara(k)=2$. We prove the proposition by induction on $n$. The case $n=0$
  is proved directly 
  by using Lemma~\ref{lem:change_y} repeatedly if necessary.

  Suppose that the proposition is proved for a chain of length
  $n-1\ge 0$. Let $y^2+Q(x)y=P(x)$ be a corresponding equation.    
 First, an equation of $W_{n}$ is
 $y_{n}^2+Q_{n}(x_{n})y_{n}=P_{n}(x_{n})$ by \cite{LTR}, Lemme 7(d).  
 Let $H(x)$ be constructed from $P(x)$ as in Lemma~\ref{lem:change_y}.
 Let
 \[ H_i(x_i):=\pi^{-m_i}H(\pi^{i}x_i), \quad 0\le i\le n.\]   
 Write $P(x)=\sum_j a_j x^j$. As $P_i(x_i)=\sum_{j} (\pi^{ij-2m_i} a_j)x_i^j$, we immediately see that
 $H_i(x_i) \in \cO_K[x_i]$ is a polynomial constructed from $P_i(x_i)$
 as in Lemma~\ref{lem:change_y}. Let $z=y+H(x)$. Let 
 $z_i=z/\pi^{m_i}=y_i+H_i(x_i)$ with equation $z_i^2+R_i(x_i)z_i=S_i(x_i)$.
 By Lemma~\ref{lem:change_y}, for all $0\le i\le n$,
 $\lambda_0(z_i)\ge \lambda_0(y_i)$;   $\delta_0(z_i)\ge \delta_0(y_i)$ if
 $(W_i)_k$ is reduced; and Property (4) is satisfied. By the induction hypothesis,
 for all $i\le n-1$, $\lambda_0(z_i)=\lambda(p_i)$,  and
 $\delta_0(z_i)=\delta_0(p_i)$ if $(W_i)_k$ is reduced. Again by
 Lemma~\ref{lem:change_y}, after repeating this operation finitely many times,
 we get an equation such that (2)-(3) are satisfied for $i=n$. 
\end{proof} 
  
\subsection{Minimal Weierstrass models and multiplicities}
Let us first recall and improve some properties of the multiplicities
$\lambda$ and $\delta$ (\cite{LTR}, Lemmes 6-7). 

\begin{lemma} \label{lem:d0} Let $W$ be a Weierstrass model of $C$. 
\begin{enumerate}[\rm (1)] 
\item Suppose $W_k$ reduced. Let $p\in W_k$. Then $\delta(p)=0$ is equivalent to 
  $W_k\to Z_k$ being \'etale at $p$, and 
  $\delta(p)\leq 1$ is equivalent to $W_k$ being smooth at $p$.
\item We have $\min\{ 1, \delta(p)\}\le \lambda(p)-\varepsilon(W)\le \delta(p)$.   
In particular, if $W_k$ is non-reduced, then $W$ is regular at $p$ if and only if
  $\delta(p)=0$.
\item If $W_k$ is non-reduced, then $\lambda(p)\ge 1$ for all $p\in W_k$ 
    and we have
    \[ \sum_{p\in W_k} (\lambda(p)-1)[k(p):k]\le
    \sum_{p\in W_k} \delta(p)[k(p):k]=2g+2. \]
\item Suppose that $W_k$ is reduced.
  Then 
    \[ \sum_{p\in W_k} \lambda(p)[k(p):k]\le
      \sum_{p\in W_k} \delta(p)[k(p):k] \le 2g+2 \]
if $W_k\to Z_k$ is separable. Otherwise, 
 $\delta(p)\ge \lambda(p)\ge 1$ for all $p\in W_k$ and we have 
    \[ \sum_{p\in W_k} (\lambda(p)-1)[k(p):k]\le
    \sum_{p\in W_k} (\delta(p)-1)[k(p):k] =2g. \]
\end{enumerate}
\end{lemma} 

\begin{proof} (1) We treat the case $\chara(k)=2$. Otherwise the
  proof is easier. Let $q\in Z_k$ be the image of $p$. 
Suppose first that $\delta(p)=0$. 
Write an equation $y^2+f(x)y=g(x)$ for $W_k$ such
  that $\min \{ 2\ord_q(f(x)), \ord_q(g(x)) \}=0$. If $\ord_q(f(x))=0$,
  then $W_k\to Z_k$ is \'etale at $p$. If $\ord_q(f(x))>0$, then
  $g(x)$ does not vanish at $q$ (otherwise $\delta(p)>0$). As 
  $W_k\times_{Z_k} \{ q \} = \Spec k(q)[T]/(T^2+g(q))$ is reduced
  ($k$ is perfect), $W_k\to Z_k$ is \'etale at $p$. 

If $\delta(p)\geq 1$, then $\ord_q(f(x)), \ord_q(g(x))>0$, so 
$W_k\times_{Z_k}\{ q \} \simeq \Spec k(q)[T]/(T^2)$,
and $W_k\to Z_k$ is ramified at $p$.  
  
The second part of (1) is just \cite{LTR}, Lemme 6 (a).

(2) is {\it op. cit.}, Lemme 7 (e) (with $r=1$).
The fact that $\lambda(p)\ge 1$ in Cases (3) and (4) inseparable is because
$\lambda(p)=0$ implies $\delta(p)=\varepsilon(W)=0$ by (2) and
$W\to Z$ \'etale at $p$ by (1). Taking (2) into account, 
(3)-(4) are \cite{LTR}, Lemme 6(c)-(e). 
Note that in the proof of Lemme 6(d), the equalities  
$\delta(p)=\delta(p,y)=2\delta(p, \tilde{Q}(x))$ must be replaced with 
$\delta(p)\le 2\delta(p, \tilde{Q}(x))$. 
\end{proof}

Let $W$ be a Weierstrass model and $p_0\in W_k(k)$.
Let $p_1^*\in W(p_0)\wedge W$. Then \cite{LTR}, Lemme 9(b) says that 
\begin{equation} \label{eq:lambdap1*} 
\lambda(p_1^*)=2g+2-2[\lambda(p_0)/2]+\varepsilon(W). 
\end{equation} 
As $W(p_0)$ depends on $\lambda(p_0)$ and in general we only have 
$\lambda(p_0)\le \delta(p_0)+\varepsilon(W) \le 2g+2 + \varepsilon(W)$, there is no hope to get
a similar relation between $\delta(p_0)$ and $\delta(p_1^*)$.  

\begin{lemma} \label{lem:delta_p0_p1*} Keep the above notation. If
  $\delta(p_0)=\lambda(p_0)-\varepsilon(W)$ and is odd, then
\[ \delta(p_1^*)=2g+2-\delta(p_0). \]   
\end{lemma}

\begin{proof} Let $y^2+Q(x)y=P(x)$ be an equation of $W$ satisfying the
  conditions of Proposition~\ref{prop:same-y}  for the chain $W, W_1=W(p_0)$.
  Write $r=[\lambda(p_0)/2]$, $Q(x)=\sum_j b_jx^j$ and $P(x)=\sum_j a_j x^j$. 
Let $v=\pi/x$, $y_\infty=\pi^{g+1-r}y/x^{g+1}$. Then an equation of $W_1$ with $v(p_1^*)=0$ is  
  \begin{equation}
    \label{eq:at_p1*}
     y^2_\infty+Q_\infty(v)y_\infty=P_\infty(v),
  \end{equation} 
with
  \begin{equation}
    \label{eq:Qinf_Pinf}
    Q_\infty(v)=\sum_{j\ge 0} (b_j \pi^{j-r})v^{g+1-j}, \quad
    P_\infty(v)=\sum_{j\ge 0} (a_j \pi^{j-2r})v^{2g+2-j}. 
   \end{equation}

   Suppose that $\varepsilon(W)=1$. Then $\delta(p_0)=2r-1$ and
   $\lambda(p_0)=2r$. This implies that  
  $\nu(b_j\pi^{j-r})>0$ if $j\ge r$. Hence
  $\ord_0 \bar{Q}_\infty(v) \ge (g+1)-r+1=g-r+2$ and
  $\ord_0 \bar{P}_\infty(v) =2g-2r+3$ ($\nu(a_{2r-1})=1$)
  is odd and $< 2 \ord_0 \bar{Q}_\infty(v)$. 
  As $\lambda(p_0)$ is even, $(W_1)_k$ is reduced. Then Lemma~\ref{lem:change_y}(3)
  says that $\delta(p_1^*)=2g-2r+3=2g+2-\delta(p_0)$. 

  If $\varepsilon(W)=0$, then $(W_1)_k$ is non-reduced,
  $\nu(a_{2r+1})=0$ and 
  \[ \delta(p_1^*)=\ord_0(\overline{\pi^{-1}P_\infty}(v))=(2g+2)-(2r+1)
    =2g+2-\delta(p_0).\]  
\end{proof}

Let $W, W'$ be two non-isomorphic minimal Weierstrass  of $C$. We study the
multiplicities of the points in the chain linking $W$ to $W'$. 

\begin{lemma} \label{lem:WW'}
  Let $g$ be even. Let $W, W'$ be two non-isomorphic minimal Weierstrass models.
  Consider the chain $(W_i)_{0\le i\le n}$ 
  connecting $W$ to $W'$ (Definition~\ref{dfn:ccW}(1)).   
  Then the following properties are true: 
 \begin{enumerate}[{\rm (1)}]
 \item for all $0\le i\le n-1$, $W_i\wedge W_{i+1}=\{ p_i \}$,
   $W_{i+1}\wedge W_{i}=\{ p_{i+1}^{*} \}$ (see Remark~\ref{rmk:p'0} below)
   consist in rational points over $k$;  
 \item $W_k$ is reduced but $(W_1)_k$ is not;
  \item $\lambda(p_0)=g+1$ and
   $\nu(\Delta_{W_1})=\nu(\Delta_{W})+2(2g+1)$; 
 \item $\lambda(p_1^{*})=\lambda(p_1)=g+2$;
 \item  $\delta(p_1^{*})=\delta(p_1)=g+1$ and $\delta(p)=0$ 
   for all $p\in (W_1)_k \setminus \{ p_1^*, p_1 \}$; 
 \item $\delta(p_0)=g+1$. 
\end{enumerate}
\end{lemma}

\begin{proof} 
  (1) As $W_i\wedge W_{i+1}=W_i\wedge W_n$ and
  $\nu(\Delta_{W_n})\le \nu(\Delta_{W_i})$, Remark~\ref{rmk:p0_nr} implies that
$W_i\wedge W_{i+1}$ consists in a  rational point. The same is true for 
$W_{i+1}\wedge W_{i}$ by symmetry.

  (2)-(3)  Recall that if $\lambda'(p_0)$ denotes the maximum of the
  $\lambda(p)$'s when $p$ runs through the rational points of
  $(W_1)_k\setminus \{ p_1^*\}$ 
  (\cite{LTR}, D\'efinition 13, where $p_1^*$ is denoted by $p_0'$), 
  then by \cite{LTR}, Lemme 9(c),
  whatever $W$ is minimal or not,    
  \begin{equation}\label{eq:lambda'}
    \lambda'(p_0)\leq 2[(\lambda(p_0)+1)/2]-\varepsilon(W). 
  \end{equation} 
By \cite{LRN}, Proposition~4.3(1), we have $\lambda(p_0)\le g+1$.  
It then follows from
\cite{LTR}, Proposition 3(iv),  that $\lambda'(p_0)\geq g+2$. Applying
Inequality~\eqref{eq:lambda'} above, we find 
\[\lambda(p_0)=g+1, \ \varepsilon(W)=0, \ \lambda'(p_0)=g+2.\]  
As $\lambda(p_0)$ is odd, $(W_1)_k$ is non-reduced.
The equality $\nu(\Delta_{W_1})=\nu(\Delta_W)+2(2g+1)$
is given by Formula~\eqref{eq:disc-p0}. 

(4) We have $\lambda(p_1^*)=g+2$ (Equality~\eqref{eq:lambdap1*}).
Let us compute $\lambda(p_1)$.  
First $\lambda(p_1)\le \lambda'(p_0)=g+2$. 
Applying Inequality~\eqref{eq:lambda'} to $p_1\in W_1$ we find 
$\lambda'(p_1)\le g+1$. On the other hands, 
applying  \cite{LTR}, Proposition 3(ii) 
to $p_1\in W_1\wedge W'$ implies that 
$\lambda(p_1)\ge g+2$, hence $\lambda(p_1)=g+2$. 

(5) By (4) and by Lemma~\ref{lem:d0}(2) applied to $W_1$, we find 
$\delta(p_1^*), \delta(p_1)\ge g+1$. 
By Lemma~\ref{lem:d0}(2)-(3), $\delta(p_1^*)=\delta(p_1)=g+1$ and
$\delta(p)=0$ for the other points of $(W_1)_k$.

(6) By (4) and (5), we have $\delta(p_1^*)=\lambda(p_1^*)-\varepsilon(W_1)$
and is odd.
Applying Lemma~\ref{lem:delta_p0_p1*} to the point $p_1^*\in W_1$ (we
have $W=W_1(p_1^*)$), we get
$\delta(p_0)=2g+2-\delta(p_1^*)=g+1$. 
\end{proof}

\begin{remark}\label{rmk:p'0}
  In \cite{LTR}, the point $p_i^*\in W_i \wedge W_{i-1}$  of $W_i$ is
  denoted by $p'_{i-1}$. It seems more natural to index points of
  $W_i$ with $i$. 
\end{remark}

\subsection{Extremal and inner minimal Weierstrass models}
When $C$ has more than one minimal Weierstrass model, then they are part
of a chain of Weierstrass models. 

\begin{theorem}\label{chain-mwm} Suppose that $g$ is even and that
  $C$ has more than one minimal Weierstrass model. Then there exists
  a chain of Weierstrass models (see Definition~\ref{dfn:ccW}(1)) 
  \[ W_0, \  W_1, \dots, \ W_n \] 
  with an even $n\geq 2$  such that the minimal Weierstrass models of $C$
  are exactly the $W_{i}$'s for even $i\le n$.  
  Moreover,   $W_{i}\wedge W_{i-1}$ (for $i\ge 1$) and 
  $W_i\wedge W_{i+1}$ (for $i\le n-1$) consist 
  respectively in rational points $p_{i}^*, p_i$ with the following properties. 
 \begin{enumerate}[\rm (1)]
 \item We have
$\lambda(p_0)= \delta(p_0)=\lambda(p_{n}^*)=\delta(p_{n}^*)=g+1$. 
\item For odd $1\le i\le n-1$, we have
  \begin{enumerate} 
  \item $\nu(\Delta_{W_i})=\nu(\Delta_{W_0})+2(2g+1)$, 
  \item     $\delta(p)=0$ for all $p\in (W_i)_k \setminus \{ p_{i}^*, p_i \}$ and 
    \[\lambda(p_i^*)=\lambda(p_{i})=g+2, \   \delta(p_i^*)=\delta(p_{i})=g+1,\]
  \item $(W_i)_k$ is non-reduced and the canonical morphism $((W_i)_k)_{\mathrm{red}}
    \to (W_i/\qi)_k=\PP^1_k$ is an isomorphism.   
 \end{enumerate}
\item For even $2\le i\le n-2$, 
  \begin{enumerate} 
  \item $\lambda(p_i^*)= \lambda(p_{i})=\delta(p_i^*)=\delta(p_{i})=g+1$; 
  \item  $(W_i)_k$ has only two singular points    $p_i^*, p_{i}$. 
    The latter are cusps and rational over $k$ and 
the normalization of $(W_i)_k$ is isomorphic to $\PP^1_k$.  
  \end{enumerate}
\end{enumerate}
\end{theorem}

\begin{proof} Let $n>0$ be the maximal 
distance between two minimal Weierstrass models of $C$
(there are only finitely such models, see \cite{LTR}, Corollaire 4).
Let $W_0, W_n$ be a pair realizing the maximal distance. 
We will show that the chain $(W_i)_{0\le i\le n}$ connecting
$W_0$ to $W_n$ satisfies the properties of the theorem. 
We first prove that Properties (1)-(3) hold without assuming the maximality
of $n$. 

(1) is follows from Lemma~\ref{lem:WW'} with $W=W_0, W'=W_n$.  

(2)-(3) We obtain (2.a) and (2.b) for $i=1$ by 
applying Lemma~\ref{lem:WW'} to $W_0, W_n$. 
As $(W_1)_k$ is non-reduced by \ref{lem:WW'}(2), the 
isomorphism in (2.c)  is always true. 
As $\lambda(p_1)=g+2$ and $W_2=W_1(p_1)$,
Formula~\eqref{eq:disc-p0} implies that 
\[
\nu(\Delta_{W_2})=\nu(\Delta_{W_1})-2(2g+1) = \nu(\Delta_{W_0}),
\]
hence $W_2$ is minimal. Applying Lemma~\ref{lem:WW'}
to $W_2, W_n$ and to $W_2, W_0$, we get (3.a) for $i=2$. 
Lemma~\ref{lem:d0}(4) implies that for all
$p\in (W_2)_k\setminus \{ p_2^*, p_2 \}$, we have
$\delta(p)\le 1$, hence $p$ is a smooth point. Let
$\Gamma_2$ be the normalization of $(W_2)_k$.
As $\delta(p_2^*), \delta(p_2)$ are odd, 
\cite{LTR}, Lemme 6(b) implies that $p_2^*, p_2$ are cusps.  
Again by \cite{LTR}, Lemme 6(b) and \cite{LB}, Proposition 7.5.4,
we have 
\[
g=p_a((W_2)_k)=g(\Gamma_2)+\sum_{p\in (W_2)_k} [\delta(p)/2][k(p):k]. 
\] 
As $\delta(p_2^*)=\delta(p_2)=g+1$, this immediately imply (3.b).
So (2) and (3)  are proved for $i=1, 2$. 
Then they can be proved by induction on $i$. Note that $n$ 
is even by (2.a). 

It remains to show that under the maximality hypothesis on $n$,
any  minimal Weierstrass model $U$ of $C$ 
is one of the $W_i$'s with $i$ even. We can suppose $U\ne W_0$. 
Consider the  
chain of Weierstrass models
\[ U_0=W_0, \ U_1, \dots, U_m=U. \] 
We have $m=d(U_0, U_m)\le n$.  Let
$u_0\in W_0\wedge U_1$. Then $U_{1}=W_0(u_0)$ by 
Lemma~\ref{lem:Wp0}(2). 
If $u_0\ne p_0$, as $p_0\in W_0\wedge W_n$, Lemma~\ref{sum-dist}
implies that $d(U_1, W_n)=n+1$, contradiction with the hypothesis on
$n$. So $U_1=W_1$. 

Let $u_1\in W_1\wedge U_2$.
We have $u_1\ne p_1^*$ because otherwise $U_2=W_1(p_1^*)=W_0=U_0$. 
By (2.b),  if $u_1\ne p_1$, then $\lambda(u_1)\le \delta(u_1)+1=1$. 
But applying (2.b) to the chain $W_0, W_1, U_2$, we get 
$\lambda(u_1)=g+2$, contradiction. So $u_1=p_1$ and $U_2=W_2$. 
We see inductively that $U_i=W_i$ for $i\le m$. In particular
$U=W_m$ and $m$ is even by (2.a).
\end{proof} 

\begin{remark} \label{rmk:crit_2} Let $g$ be even. Recall that
  by \cite{LRN}, Proposition 4.6, 
  $C$ has more than one minimal Weierstrass model if and only if for any
  (or some) minimal Weierstrass $W$ of $C$, $W_k$ is reduced and
  there exists $p_0\in W_k(k)$ such that $\lambda(p_0)=g+1$, and 
  $p_1\in W(p_0)_k(k)\setminus (W(p_0)\wedge W)$ with $\lambda(p_1)\ge g+2$.
  If such a $p_1$ exists, then $\lambda(p_1)=g+2$ and $W(p_0)(p_1)\ne W$ is minimal. 
\end{remark}

\begin{definition} \label{extremal-MWM} Under the hypothesis of
  Theorem~\ref{chain-mwm},  we call $W_0, W_n$   
\emph{the extremal minimal Weierstrass models of $C$}. 
The other minimal Weierstrass models are said to be 
\emph{inner}. We call $(W_i)_{0\le i\le n}$ the
\emph{minimality chain of Weierstrass models of $C$}.  
\end{definition}

Let $(W_i)_i$ be as above and let $Z_i=W_i/\qi$. 
Let $Z_0\vee Z_1 \vee ... \vee Z_n$ be the smallest
model of $C/\qi$ dominating all the $Z_i$'s. This is the minimal desingularization
of $Z_0\vee Z_n$ as we saw before. Its normalization in $K(C)$, denoted by 
$W_0\vee W_1 \vee ... \vee W_n$ is also the smallest model of $C$ dominating
all the $W_i$'s. See also Definition~\ref{dfn:ccW}(3). Let $\Gamma_i$ be the
strict transform of $((W_i)_k)_{\mathrm{red}}$. We saw above that
$\Gamma_i$ is a projective line over $k$ (of multiplicity $2$ in
$(W_0\vee W_1 \vee ... \vee W_n)_k$)  
if $i$ is odd, and is rational
if $i$ is even and $\ne 0, n$. We will see later (Theorem~\ref{regular-even})
that $\Gamma_i$ is in fact
also a projective line over $k$ for these $i$'s and $\Gamma_0$, $\Gamma_n$ are the
normalizations of $(W_0)_k$, $(W_n)_k$ at $W_0\wedge W_n$ and $W_n\wedge W_0$
respectively.

\begin{figure} [h] 
 \centering 
\begin{tikzpicture}[scale=1] 
  \coordinate (P1) at   (0, 1.5);
  \coordinate (Q1) at  (1.5, 0);
  \draw  (P1) node [below left] {$\Gamma_0$} to [bend left] (Q1); 
  \draw (1, 0) -- (2.5, 1.5);
  \draw (3, 1.4) node {$2\Gamma_1$};
  \draw (2.6, 0) node {$\Gamma_2$}; 
  \draw (1.5, 1.5) -- (3,0);
  \draw[dashed]  (2,0.5) -- (5, 0.5);
  \draw (4,0) -- (5.5,1.5); 
  \draw (4.5,1.5) -- (6,0);
  \draw (6.7, 0) node {$2\Gamma_{n-1}$}; 
  \coordinate (P2) at   (5.5, 0);
  \coordinate (Q2) at  (7, 1.5); 
\draw (Q2) node [below right] {$\Gamma_n$} to [bend right] (P2); 
\end{tikzpicture} 
\caption{The closed fiber of $W_0\vee \cdots \vee W_n$.}
 \label{figure:mwm} 
\end{figure} 

\begin{corollary} \label{cor:equa_term} Keep the hypothesis of
Theorem~\ref{chain-mwm}. 
  Then there exists an equation
  \[ y^2+Q(x)y=P(x) \]
  of $W_0$ such that
  \begin{enumerate}[\rm (1)] 
\item for all even $i\le n$, 
  \[ y_i^2+\pi^{-(g+1)i/2}Q(\pi^i x_i)y_i
=  \pi^{-(g+1)i} P(\pi^i x_i) \]
with $x_i=x/\pi^{i}$ and $y_i=y/\pi^{(g+1)i/2}$, is an equation of $W_{i}$. 
\item If $P(x)=\sum_j a_jx^j$ and $Q(x)=\sum_j b_j x^j$, then
  $\nu(a_{g+1})=0$, and for all $j\ge 0$, 
\[
\nu(a_j)\ge n(g+1-j), \quad \nu(b_j)\ge n(g+1-2j)/2.  
\] 
\end{enumerate}
\end{corollary}

\begin{proof} Take an equation given by Proposition~\ref{prop:same-y}. As
  $\delta(p_0)=g+1$ is odd, $g+1=\ord_0 \bar{P}(x)$, so $v(a_{g+1})=0$. 
  Using Theorem~\ref{chain-mwm}(1)-(3), we have $m_{i}= (g+1)i/2$ if $i$
  is even. This implies (1). Taking $i=n$ in (1) we get (2). 
\end{proof}

\begin{remark}\label{rmk:more_than_one} Conversely, suppose that
  $C$ has an equation satisfying the 
  inequalities of \ref{cor:equa_term}(2), then the equation of
  \ref{cor:equa_term}(1) defines a minimal Weierstrass model $W_i$ for any 
  even $i\le n$. Indeed, $(W_i)_k$ is reduced and has two rational points
  $x_i=0, \infty$ with $\lambda=g+1$. So the other points of $(W_i)_k$
  have multiplicity $ \lambda\le 1$ by Lemma~\ref{lem:d0}(4)
  and we can apply the criterion of \cite{LRN}, Proposition 4.3(1). 
\end{remark}

The structure of the closed fibers of the extremal minimal Weierstrass
models are given below.

\begin{proposition}\label{prop:W0} Keep the hypothesis of
  Theorem~\ref{chain-mwm}.  
Then  $(W_0)_k$ is geometrically integral over $k$ with a cusp at
   $p_0$. Its normalization at $p_0$  has arithmetic genus
   equal to $g/2$, and the point lying over $p_0$ is rational over $k$.
\end{proposition} 

\begin{proof} Denote $\Gamma=(W_0)_k$. As $\delta(p_0)=g+1$ is odd
  by Lemma~\ref{lem:WW'}(6), $p_0$ is
  a cusp (\cite{LTR}, Lemme 6.b). This implies that $\Gamma$ is irreducible
  (otherwise it is union of two $\PP^1_k$, hence the singular points are
  all double points). As $\delta(p_0)$ does not change by field
  extension, $\Gamma$ is integral and geometrically irreducible,
  hence geometrically integral ($k$ is perfect). As $p_0$ is a cusp,
  the point of the normalization of $\Gamma$ lying over $p_0$
  is rational over $k$.  A 
computation similar  to \cite{LB}, Proposition 7.5.4 shows that 
\[ g=p_a(\Gamma)=p_a(\Gamma_0)+[\delta(p_0)/2], \] 
hence  $p_a(\Gamma_0)=g/2$.
\end{proof}

\begin{example} \label{ex:Wn}
  To construct examples of $C$ with more than one minimal
  Weierstrass model, it suffices to use Remark~\ref{rmk:more_than_one}.
  But the conditions of Corollary~\ref{cor:equa_term}(2) are not enough
  to decide whether $i=0, n$ give extremal 
  Weierstrass models. Below we give examples of minimality chain
  $(W_i)_{0\le i\le n}$ with prescribed $n$, 
  $\Gamma_0$ and $\Gamma_n$ (the respective normalizations of
  $(W_0)_k$, $(W_n)_k$ at $p_0$ and $p_{n}^*$). 

  Let $g=2g_1, n\ge 2$ be even. Let $F_n(x_n), G_n(x_n)\in \cO_K[x_n]$ with
  $\deg F_n(x_n)\le g_1$,  $G_n(x_n)$ monic of degree $2g_1+1$ such that
  the affine curve
  \begin{equation} \label{eq:W_n_infty}
    y_n^2+\bar{F}_n(x_n)y_n = \bar{G}_n(x_n)
  \end{equation} 
  over $k$ is integral and such that
  $\delta(p)\le g$ for all its rational points. This affine curve 
  will be our $(W_n)_k \setminus \{ p_n^* \}$. 
  Let $Q_\infty(t), P_\infty(t)\in \cO_K[t]$ with the similar properties
  for the affine curve
  \begin{equation} \label{eq:W_0_infty}
    z^2+Q_\infty(t)z=P_\infty(t). 
  \end{equation}
  This affine curve will be our
  $(W_0)_k\setminus \{ p_0 \}$. Let $Q_0(x)=x^{g_1}Q_\infty(1/x)$
  and $P_0(x)=x^{2g_1+1}P_\infty(1/x)$. 
Let
\[
G_0(x)=\pi^{(g+1)n}G_n(\pi^{-n}x), \quad  
F_0(x)=\pi^{(g+1)n/2}F_n(\pi^{-n}x) \in \cO_K[x]. 
\]
Consider the equation 
\begin{equation} \label{eq:glue_0n} 
y^2+Q_0(x)(x^{g_1+1}+ F_0(x))y=P_0(x)G_0(x) 
\end{equation} 
and suppose it defines a (smooth) hyperelliptic curve $C$ 
of genus $g$ over $K$. If $\chara(k)\ne 2$, 
we can just take $Q_\infty(t)=F_n(x_n)=0$ and $P_\infty(t), G_n(x_n)$ separable.  

Let $W_0$ be the Weierstrass model of $C$ defined by
Equation~\eqref{eq:glue_0n} above. Let us show that $W_0$ is minimal.
The closed fiber $(W_0)_k$ has an equation
\[
y^2+\bar{Q}_0(x)x^{g_1+1}y=\bar{P}_0(x)x^{2g_1+1}. 
\]
Let $p_0\in (W_0)_k(k)$ be the point with $x=0$.
It is straightforward to check that
$\mu_0(F_0(x))\ge g_1+1$ and $\mu_0(G_0(x))=g+1$. 
Hence $\mu_0(P_0(x)G_0(x))=g+1$ is odd, because
$P_0(0)=1$. Similarly,   
if $\chara(k)=2$, 
\[ \mu_0(Q_0(x)(x^{g_1+1}+F_0(x)))\ge \mu_0(x^{g_1+1}+F_0(x))=g_1+1
> (g+1)/2.\]
This implies that $\lambda(p_0)=g+1$ by Remark~\ref{rmk:lambda'}(1).
The affine curve $(W_0)_k\setminus \{ p_0 \}$ is defined by
Equation~\eqref{eq:W_0_infty} with $t=1/x$ and $z=y/x^{g+1}$. 
Therefore $\lambda(p)\le \delta(p)\le g$ for all $p\in (W_0)_k(k) \setminus
\{ p_0 \}$ by construction. By \cite{LRN}, Proposition 4.3, $W_0$ is minimal. 

Let $x_n=x/\pi^n$ and $y_n=y/\pi^{(g+1)n/2}$. Then we have the relation 
\[
y_n^2+Q_0(\pi^n x_n)(\pi^{n/2}x_n^{g_1+1} + F_n(x_n))y_n=P_0(\pi^n x_n)G_n(x_n) 
\]
which mod $\pi$ is Equation~\eqref{eq:W_n_infty}. So this equation
defines a Weierstrass model $W_n$. Using the formula at the bottom of
\cite{LTR}, page 4581, we find $\nu(\Delta_{W_n})=\nu(\Delta_{W_0})$.
So $W_n$ is minimal.

By the hypothesis on Equations \eqref{eq:W_0_infty} and \eqref{eq:W_n_infty},
on each of $(W_0)_k$ and $(W_n)_k$, there is at most one point with
$\delta=g+1$ (the points $x=0$ and $x_n=\infty$).  
It follows from Theorem~\ref{chain-mwm}(3.a) that $W_0$ and
$W_n$ are extremal.  
\end{example}

\subsection{Number of minimal Weierstrass models} \label{subsect:mwm}

Let $C$ be a hyperelliptic curve over $K$ of genus $g\ge 2$.  
We can ask how many minimal Weierstrass models (up to isomorphisms)
it can have. Denote this (finite) number by $m_C$.

\begin{definition} \label{dfn:mini_disc} 
The \emph{minimal discriminant} $\nu(\Delta_{C})$ of $C$ is
$\nu(\Delta_W)$ for any minimal Weierstrass model of $C$. 
\end{definition}

\begin{theorem} \label{thm:bound_n} Let $g$ be even.
  Then the number $m_C$ of minimal Weierstrass models of $C$ satisfies 
  \begin{equation}
    \label{eq:ineq_2}
     m_C \le 1+ \frac{\nu(\Delta_C)}{2g(g+1)}. 
  \end{equation} 
  This inequality is sharp. 
\end{theorem}

The proof will be given at the end of this subsection.
Note that, when $\chara(k)\ne 2$, one has 
  \[
|\mathrm{Art}(\cC)| \le \nu(\Delta_{C})  
  \] 
 where $\mathrm{Art}(\cC)$ is the Artin conductor
 of the minimal regular model of $C$  (\cite{OS}, Theorem 1.1). 
 Using the (very weak) inequality $m_C-1\le |\mathrm{Art}(\cC)|$, we
also get also a weak bound $m_C-1\le \nu(\Delta_{C})$. 
\medskip

We first prove same preliminary results on discriminants. 

\begin{lemma}\label{lem:disc_prod} Let $| \ |$ be an absolute value on $K$ associated to the
  valuation of $K$.  Endow $K^{\text{alg}}$ with
  a prolongation of $| \ |$. 
  Let $F_0(x), F_1(x) \in K[x]$ of respective degrees $d_0, d_1$ and
  leading coefficients $c_0, c_1$. Suppose that for all roots 
  $\alpha$ of $F_0(x)$ and all roots $\beta$ of $F_1(x)$ in $K^{\text{alg}}$ 
we have $|\alpha|>|\beta|$. Then 
\[ |\disc(F_0(x)F_1(x))| = |c_1|^{2d_0}|F_0(0)|^{2d_1} |\disc(F_0(x))| |\disc(F_1(x))|.\]   
\end{lemma}

\begin{proof} We can enlarge $K$ and assume that $F_0(x), F_1(x)$
  have all its roots $\alpha_i, \beta_j$ in $K$. Then 
  \[
    \disc(F_0F_1)=(c_0c_1)^{2(d_0+d_1)-2} \prod_{i_1< i_2} (\alpha_{i_1}-\alpha_{i_2})^2
    \prod_{j_1<j_2} (\beta_{j_1}-\beta_{j_2})^2 \prod_{i, j} (\alpha_i-\beta_j)^2.  
  \]
  So
    \[
|\disc(F_0F_1)|=  |\disc(F_0)||\disc(F_1)| 
|c_1^{2d_0}c_0^{2d_1}\prod_{i, j} (\alpha_i-\beta_j)^2|.  
  \]
We have 
  \[
    |c_1^{2d_0}c_0^{2d_1}\prod_{i, j} (\alpha_i-\beta_j)^2| =
    |c_1^{2d_0}c_0^{2d_1}| |\prod_{i} \alpha_i|^{2d_1} = |c_1^{2d_0}| |F_0(0)|^{2d_1} 
  \]
and the lemma is proved.   
\end{proof}

\begin{notation} \label{nota:nun}
  For any $n\ge 1$, denote by $\nu_{n, 0}$ the Gauss valuation on $K(x)$
  associated to the variable $x/\pi^n$. It is determined  by
  \[ \nu_{n,0}\left(\sum_j a_jx^j\right)=\min_j \{ \nu(a_j)+nj \}\]
 on $K[x]$. 
\end{notation}

\begin{lemma} \label{lem:discF}
  Let $F(x)=\sum_{j\ge 0} a_jx^j\in \cO_K[x]$ with
  $\bar{F}(x)\ne 0$. Let $n\ge 1$ and let $e \ge 0$ be an integer such that
  $\nu(a_e)+ne=\nu_{n,0}(F)$.  
  Then
  \[ \nu(\disc(F))\ge n e(e-1). \]  
\end{lemma}

\begin{proof} We can suppose $K$ is complete and $F(x)$ has all its roots in
  $K$. Let $d=\ord_0\bar{F}(x)$. By Hensel's lemma, we can decompose 
  $F(x)=F_0(x)F_1(x)$ in $\cO_K[x]$
  with $F_1(x)$ monic and $\bar{F_1}(x)=x^d$. This implies
  that $F_0(0)\in \cO_K^*$. Lemma~\ref{lem:disc_prod} implies that
\[|\disc(F)|=|\disc(F_0)| |\disc(F_1)|\le |\disc(F_1)|.\] 
We have
\[ \nu_{n,0}(F_1)=\nu_{n,0}(F)-\nu_{n,0}(F_0)=\nu_{n,0}(F)\] 
because $F_0(0)\in \cO_K^*$. We can write
$F_1(x)=\pi^{\nu_{n,0}(F)}G(x_n)$ with $x_n=x/\pi^n$ and $G(x_n)\in \cO_K[x_n]$.
As $\deg \bar{G}(x_n)\ge e$, $F_1(x)$ has at least $e$ roots $\beta$
with $|\beta|\le |\pi|^n$, the other roots $\beta'$ have $|\beta'|<1$. 
So $|\disc(F_1)|\le |\pi|^{ne(e-1)}$ and the lemma is proved. 
\end{proof}

\begin{lemma} \label{lem:discF_u}
  Fix positive integers $n, d, e$ with $e\le 2d-1$. Consider  
  the  indeterminates
\[t, \quad \underline{u}= \{ u_0, \dots, u_{2d}\},\quad 
  \underline{v}=\{ v_0, \dots, v_d\}\]
and    the polynomials
  \[ {\bf Q}(x)=\sum_{0\le j<e/2} t^{\lceil ne/2\rceil-nj} v_j x^j
    + \sum_{e/2 <j\le d} v_j x^j\in \Z[t, \underline{u}, \underline{v}] [x]  
  \] 
  \[ {\bf P}(x)=\sum_{0\le j\le e} t^{ne-nj} u_j x^j
    + \sum_{e<j \le 2d} u_j x^j\in \Z[t, \underline{u}, \underline{v}] [x]. 
  \]
  Then
  \begin{equation}
    \label{eq:disc_univ}
    2^{-2d}\disc(4{\bf P}+{\bf Q}^2)\in
    t^{ne(e-1)}\Z[t, \underline{u}, \underline{v}]. 
  \end{equation}
  \end{lemma}

\begin{proof} Let ${\bf F}=4{\bf P}+{\bf Q}^2$.  We know that
    $2^{-2d}\disc(\bar{{\bf F}})\in \Z[\underline{u}, \underline{v}, t]$
    (\cite{LTR}, \S 2). As
    $\Z[\underline{u}, \underline{v},t]$ is an UFD, it is enough to prove that
    $t^{ne(e-1)} \mid \disc({\bf F})$ in $\Q(\underline{u}, \underline{v})[t]$.  
    Endow $L:=\Q(\underline{u}, \underline{v})(t)$ with the discrete
    valuation $\nu_L$ trivial on $\Q(\underline{u}, \underline{v})$
    and such that $\nu_L(t)=1$. We have
    \[ {\bf F}(x)= \sum_{0\le j \le e} t^{ne -nj } a_j x^j +
      \sum_{e < j \le 2d} a_j x^j  \]
    with $a_j \in \Q[\underline{u}, \underline{v}, t]$. We have
    $\bar{\bf F}(x)\ne 0$ because the leading coefficient of ${\bf F}(x)$ is
    $4u_{2d}+v_d^2 \in \cO_L^*$.

    To compute $a_e$, we notice that if $j_1+j_2=e$ and $j_1\le j_2$, then
    $j_1<e/2$ and $\lceil ne/2\rceil-nj_1>0$. Thus
    $a_e\in 4u_e+ t\Z[ \underline{v}]\subset \cO_L^*$.   
This implies that $\nu_{n,0}({\bf F})=ne$. The 
    lemma then follows from Lemma~\ref{lem:discF}. 
  \end{proof}

\begin{corollary}[See also \cite{LTR}, Lemme 13] \label{cor:dl} Let $C$ be a hyperelliptic curve of genus
  $g\ge 1$. Let $U_0, \dots, U_n$ be a chain of Weierstrass models of $C$. 
  Write $U_{i+1}=U_i(p_i)$ for all $0\le i\le n-1$. 
  Suppose that there exists $e\ge 2$ such that $2[\lambda(p_i)/2]\ge e$ for
  all $i\ge 0$. Then
  \[
\nu(\Delta_{U_0}) \ge ne(e-1).  
\]
In particular, if all the $p_i$'s are singular points, then $n\le \nu(\Delta_{U_0})/2$. 
\end{corollary}

\begin{proof} With the notation of Proposition~\ref{prop:same-y}(1), we
  have $m_n\ge \lceil ne/2\rceil$ and an equation
  $y^2+Q(x)y=P(x)$ of $U_0$ such that  
\[ \pi^{-m_n} Q(\pi^{n}x_{n}), 
\pi^{-2m_n} P(\pi^{n}x_{n}) \in \cO_K[x_n]. 
\]
Write $Q(x)=\sum_j b_j x^j$ and $P(x)=\sum_j a_jx^j$. 
  Then $b_j\in \pi^{\lceil ne/2\rceil-nj} \cO_K$ for $j<e/2$ and
  $a_j\in \pi^{ne-nj} \cO_K$ for $j\le e$. Therefore there exists
  a ring homomorphism 
  \[ \Z [t, \underline{u}, \underline{v} ]  \to \cO_K,  \quad
    t\mapsto \pi, \quad {\bf Q}(x) \mapsto Q(x), \quad
    {\bf P}(x)\mapsto P(x).\] 
    This implies by Lemma~\ref{lem:discF_u} (with $d=2g+2$) 
    that $\Delta_{U_0}\in \pi^{ne(e-1)}\cO_K$. 
\end{proof}

\noindent {\it Proof of Theorem~\ref{thm:bound_n}}. We can suppose
that $m_C\ge 2$. Let $W_0, W_n$ be the extremal minimal Weierstrass models
of $C$ (Definition~\ref{extremal-MWM}). Then $n=2m_C-2$. With the
notation of Proposition~\ref{prop:same-y} (applied to the minimality
chain of $C$), we have  $m_n=n(g+1)/2$ using Theorem~\ref{chain-mwm}. 
Applying Corollary~\ref{cor:dl} to the minimality chain of $C$, we get 
$\nu(\Delta_{W_0})\ge ne(e-1)$, where $e=g+1$. This implies
Inequality~\eqref{eq:ineq_2}.

  Consider the example $\cO_K={\mathbb C}[[t]]$ and $W_0$ defined by the
  equation
  \[ y^2=(x^{g+1}+1)(x^{g+1}+t^{n(g+1)}).\] 
  Then the minimal Weierstrass models are the
  \[ y_i^2= (\pi^{i(g+1)} x_i^{g+1}+1)(x_i^{g+1}+t^{(n-i)(g+1)}),
\quad y_i=y/\pi^{i(g+1)/2}, \ x_i=x/\pi^i,  
    \] 
for the even $i\le n$. We have $\nu(\Delta_{W_0})=ng(g+1)$ 
and Inequality~\eqref{eq:ineq_2} is an equality. 
\qed

\subsection{Minimal Weierstrass models and base changes} \label{mwm-bc} 

Our primary interest is of arithmetic nature, so $k$ is not algebraically
closed in general.
A natural question is if $C$ has a unique (resp. more than one)
minimal Weierstrass model, is the same property holds for $C_{K'}$ over
an extension of discrete valuation fields $K'/K$ ? 

Our general settings in this subsection is the following: 
$\cO_{K'}/\cO_K$ is an extension of discrete valuation rings such that
the residue field $k'$ of $\cO_{K'}$ is also perfect. We denote by
$e_{K'/K}=\nu_{K'}(\pi)$, where $\nu_{K'}$ is the normalized valuation on $K'$. 
For example $e_{K'/K}=1$ if $K'$ is the completion of $K$ or if $K'/K$ is unramified.
Recall that Weierstrass models are normal in our definition. 

\begin{lemma} \label{lem:wbc} Let $W$ be a Weierstrass model of $C$. Then
  $W_{\cO_{K'}}$ is a Weierstrass model of $C_{K'}$ if $W_k$ is reduced or if
  $e_{K'/K}=1$.   
\end{lemma}

\begin{proof} The generic fiber of $W':=W_{\cO_{K'}}$ over $\cO_{K'}$ is $C_{K'}$. What
  we need is the normality. If $W_k$ is reduced,
  then $W'_{k'}=W_{k'}$ is reduced because $k$ is perfect. So
  $W'$ is normal (\cite{LTR}, Lemme 2(a)). Suppose $W_k$ is non reduced and
  $e_{K'/K}=1$. If $\chara(k)\ne 2$, the normality of
  $W'$ follows from \cite{LTR}, Lemme 2(d). Suppose $\chara(k)=2$, then
  we take an equation of $W$ satisfying the conditions of
  \cite{LRN}, Lemma 2.3(1) (see Algorithm 6.1 therein). Such an equation
  also satisfies the conditions of \cite{LRN}, Lemma 2.3(1) over $\cO_{K'}$,
  and the latter implies  that $W'$ is normal.
\end{proof}

\begin{lemma} \label{lem:mini_down}  Let $W$ be a Weierstrass model of $C$ 
  with reduced $W_k$.
  \begin{enumerate}[\rm (1)] 
  \item If $W_{\cO_{K'}}$ is minimal, then so is $W$.
  \item If $W$ is minimal and $e_{K'/K}=1$, then $W_{\cO_{K'}}$ is minimal.
  \item If $K$ is dense in $K'$, then any Weierstrass model $W'$ of $C_{K'}$ is
    isomorphic to $W_{\cO_{K'}}$ for some Weierstrass model $W$ of $C$. 
  \end{enumerate} 
\end{lemma}

\begin{proof} (1) We will use the minimality criterion
  \cite{LRN}, Proposition 4.3. 
  Let $p_0\in W_k(k)$. Then it induces a point $p'_0\in W'_{k'}(k')$
  and it is clear that $\lambda(p'_0)\ge \lambda(p_0)$ because an equation of
  $W$ is also an equation of $W'$.  
  If $g$ is even, the aforementioned criterion implies that 
  $W$ is minimal. If $g$ is odd, and there exists $p_0\in W_k(k)$ with
  $\lambda(p_0)\ge g+2$, then $\lambda(p'_0)\ge g+2$ and 
  \cite{LRN}, Proposition 4.3(2) implies that $\lambda(p_0)=g+2$. Note that
  that $W'_{k'}$ is reduced. Any $q\in W(p_0)_k(k)$ induces a
  $q'\in W'(p'_0)_{k'}(k')$. So $\lambda(q)\le \lambda(q')\le g+2$ and $W$
  is minimal.

  (2) is a consequence of Proposition~\ref{prop:mini}(1) and (2.a).

  (3) The quotient $Z':=W'/\qi$ is a smooth model of $\PP^1_{K'}$. It is
  clear that $Z'=Z_{\cO_{K'}}$ for some smooth model $Z$ of $\PP^1_K$. The
  normalization $W$ of $Z$ in $K(C)$ is the desired Weierstrass model of $C$.   
\end{proof}

\begin{proposition} \label{prop:more_than_one} Let $g\ge 2$ be even.
  Suppose that $C$ has more than one minimal Weierstrass model. Let
  $(W_i)_{0\le i\le n}$ be its minimality chain    (Definition~\ref{extremal-MWM}).  
\begin{enumerate}[\rm (1)]
\item Let $W=W_i$ for some even $i\le n$. Then $W_{\cO_{K'}}$ is a
  minimal Weierstrass model of $C_{K'}$ over $\cO_{K'}$.
  In particular, $C_{K'}$ has more than one minimal 
  Weierstrass model.
\item We have 
  \[
\nu_{K'}(\Delta_{C_{K'}})=e_{K'/K} \nu_K(\Delta_{C}). 
  \] 
\item If $e_{K'/K}=1$, then $((W_i)_{\cO_{K'}})_{0\le i\le n}$ is the minimality
  chain of $C_{K'}$. 
    \end{enumerate}
 \end{proposition}

\begin{proof} (1) By Theorem~\ref{chain-mwm}, $W_k$ is reduced and has
    a rational point $w_0$ with $\delta(w_0)=g+1$. Lemma~\ref{lem:wbc}
    says that $W':=W_{\cO_{K'}}$ is a Weierstrass model of $C_{K'}$. Let 
    $w'_0\in W'_{k'}=W_{k'}(k')$ the point  
    lying over $w_0$. Then it is easy to see that, as a general fact
    (because $k$ is perfect), $\delta(w'_0)=g+1$.  Lemma~\ref{lem:d0}(4) 
     implies that $\lambda(p')\le g+1$ for all $p'\in W'_{k'}$. So $W'$ is 
     minimal by \cite{LRN}, Proposition 4.3(1). 
     
  (2) follows from (1).  
  
  (3) By Lemmas~\ref{lem:wbc} and \ref{lem:mini_down}, for all $i\le n$,
  $W'_i:=(W_i)_{\cO'_K}$ is a Weierstrass model of $C_{K'}$, minimal if
  $i$ is even. As $e_{K'/K}=1$, $(W'_i)_i$ is also a chain.  
  It remains to show that $W'_0$ and $W'_n$ are extremal.
We will use repeatedly the fact that if $p\in W_k$ and
  $p'\in W_{k'}$ is a point lying over $p$, then $\lambda(p')=\lambda(p)$
  (Remark~\ref{rmk:lambda'}(\ref{lambda_1})). 

  Suppose that $W'_n$ is inner. 
  By Theorem~\ref{chain-mwm}, there exists a chain $W'_n, U'_{n+1}, U'_{n+2}$
  of Weierstrass models of $C_{K'}$ such that $U'_{n+1}\ne W'_{n-1}$ and
  $U'_{n+2}$ is minimal. Let $p'_{n}\in W'_n\wedge U'_{n+1}$. 
  We consider ${p'_n}^{*}\in W'_n\wedge W'_{n-1}$ as a point  
  of $W'_n\wedge W'_{n-1}$. Then $p'_n\ne {p'_n}^{*}$.
  Let $p_n$ be the image of $p'_n$ in $W_n$. Then $\lambda(p_n)=\lambda(p'_n)=g+1$
  (Theorem~\ref{chain-mwm}). As $\lambda(p_{n}^*)=g+1$, Lemma~\ref{lem:d0}(4)
  implies that $p_n$ is rational over $k$. As $e_{K'/K}=1$, we have  
  $U'_{n+1}=W'_n(p'_n)=W_n(p_n)_{\cO_{K'}}$. For similar reasons,
  the image $p_{n+1}$ of $p'_{n+1}\in U'_{n+1}\wedge U'_{n+2}$ in $W_{n+1}:=W_n(p_n)$ 
  is rational over $k$ and $U'_{n+2}=(W_{n+2})_{\cO_{K'}}$ where 
  $W_{n+2}:=W_{n+1}(p_{n+1})$. As $\lambda(p_n)=g+1$ and $\lambda(p_{n+1})=g+2$,
  Formula~\eqref{eq:disc-p0} implies that $\nu(\Delta_{W_{n+2}})=
  \nu(\Delta_{W_n})$, hence $W_{n+2}$ is minimal and $W_n$ is not extremal,
contradiction. 
\end{proof}

\begin{corollary}[See also Corollary~\ref{cor:stable}] Suppose that $g$ is
  even. If $C$ has more than one minimal 
    Weierstrass model, then $C$ does not have potentially good reduction. 
  \end{corollary}

  \begin{proof} Suppose that $C$ has good reduction over some extension of
    discrete valuation fields $K'/K$. Extending $K'$ again we can suppose the
    residue field $k'$ algebraically closed, hence perfect.
    By Proposition~\ref{prop:stab}, $C_{K'}$ has a unique minimal Weierstrass
    model. Contradiction with Proposition~\ref{prop:more_than_one}(1).
    One can also use Proposition~\ref{prop:more_than_one}(2) as $\nu(\Delta_W)=0$
    if and only if $W$ is smooth. 
\end{proof}

Now we consider the case when $C$ has a unique minimal Weierstrass model.
In general this uniqueness property is not preserved by extensions of $K$.
We will mostly focus on unramified extensions to get more chance to
preserve the uniqueness property.  
First we recall and make more precise some results from \cite{LTR}.

\begin{proposition} \label{prop:mini}
  Let $W$ be a minimal Weierstrass model of $C$ over $\cO_K$.
  \begin{enumerate}[\rm (1)] 
  \item There exists an extension $K'/K$ with $e_{K'/K}=1$ such that
    $W_{\cO_{K'}}$ is non minimal if and only if 
    $g$ is even, and there exists $p_0\in W_k$ with $\lambda(p_0)=g+2$.
    \item Suppose that $W$ satisfies the conditions of {\rm (1)}, then
  \begin{enumerate}
  \item $W_k$ is non-reduced, $[k(p_0) : k]=2$, and $\lambda(p)=1$ for all
    $p\in W_k\setminus \{ p_0 \}$;
  \item $W$ is the unique minimal Weierstrass model of $C$;
\end{enumerate}     
\item Keep the hypothesis of (2). Let $K'/K$ be an extension with $e_{K'/K}=1$ and
  residue field $k'$.
  \begin{enumerate}
\item Suppose that $k(p_0)\subseteq k'$. Then $W_{\cO_{K'}}$ is not minimal,  
  and for any $w'\in W_{k'}(k')$ lying over $p_0$, $(W_{\cO_{K'}})(w')$
  is a minimal Weierstrass model of $C_{K'}$.  
  \item Let $K''/K$ be any extension of discrete valuation fields such that
    the residue field $k''$ of $K''$ is perfect and contains $k(p_0)$, then
    $C_{K''}$ has more than one minimal Weierstrass model and
(see Definition~\ref{dfn:mini_disc}) 
    \[\nu_{K''}(\Delta_{C''})= e_{K''/K}(\nu_K(\Delta_C)-2(2g+1)).\] 
  \item If $k'$ does not contain $k(p_0)$, then $W_{\cO_{K'}}$ is the unique
    minimal Weierstrass model of $C_{K'}$.
\end{enumerate} 
\end{enumerate}
\end{proposition}

\begin{proof} (1) - (3.a) are contained in \cite{LTR}, Proposition 4 and its  
  proof under the hypothesis $K'/K$ unramified. Everything adapt well
  to the general case by Remark~\ref{rmk:lambda'}(\ref{lambda_1}).

  (3.b) By Lemma~\ref{lem:mini_down}(3), one can suppose $K''$ complete. 
  Then $K''$ contains an unramified subextension $K'/K$ with residue field
  $k'=k(p_0)$. Let $w_1\ne w_2\in W_k(k')$ be the points lying over $p_0$. 
  Then $\lambda(w_1)=\lambda(w_2)=g+2$.  
  By (3.a), $W_{\cO_{K'}}(w_i)$, $i=1, 2$, are minimal with 
  discriminants of valuation $\nu(\Delta_{W})-2(2g+1)$ using
  Formula~\eqref{eq:disc-p0}.
  Thus $C_{K'}$ has more   than one minimal Weierstrass model and
  $\nu_{K'}(\Delta_{C'})=\nu_K(\Delta_C)-2(2g+1)$. Then (3.b) for $K''$ follows
  from Proposition~\ref{prop:more_than_one}.

  For (3.c), notice that $p_0$ induces a quadratic point of $W_{k'}$ with 
  $\lambda=g+2$, then apply (2.b).
   \end{proof}

\begin{remark} \label{rmk:mini_mini}
  Suppose $g$ is even. It turns out that a minimal Weierstrass model $W$ of
  $C$ is not dominated by the minimal regular model of $C$ if and
  only if it satisfies the hypothesis of Proposition~\ref{prop:mini}(1)
  (\cite{LTR}, Propositions 5 and 6). 
  \end{remark}

\begin{corollary} \label{cor:mini_nu} Suppose that $g$ is even.
  Let $W$ be a minimal Weierstrass model of $C$ such that $W_k$
  is non-reduced. 
  \begin{enumerate}[\rm (1)]
  \item If $\lambda(p)\le g+1$ for all $p\in W_k$, then for
    any unramified extension $K'/K$, $W_{\cO_{K'}}$ is the unique
    minimal Weierstrass model of $C_{K'}$, 
  \item otherwise,  there exists a quadratic point
    $p_0\in W_k$ with $\lambda(p_0)=g+2$, and $C_{K'}$ has
    more than one minimal Weierstrass model for any unramified
    extension $K'/K$ with residue field containing $k(p_0)$. 
   \end{enumerate}
\end{corollary}

\begin{proof} (1) For any unramified extension $K'/K$, the
  closed fiber of $W_{\cO_{K'}}$ is non-reduced. 
  If $\lambda(p)\le g+1$ for all $p\in W_k$,  then the same is true for
  the rational points of the closed fiber of $W_{\cO_{K'}}$
  (Remark~\ref{rmk:lambda'}(\ref{lambda_1})),  this implies that  
  $W_{\cO_{K'}}$ is the unique minimal Weierstrass model of $C_{K'}$
  by \cite{LRN},  Proposition 4.6. 

  (2) Let $p_0\in W_k$ with $\lambda(p_0)\ge g+2$. As $W$ is 
  minimal,   by \cite{LRN}, Propositions 4.3(1),
  we have $\lambda(p)\le g+1$ for all $p\in W_k(k)$. 
    So $[k(p_0):k]\ge 2$. By Lemma~\ref{lem:d0}(3), we have $[k(p_0):k]=2$
  and $\lambda(p_0)=g+2$.  The last part is Proposition~\ref{prop:mini}(3b). 
\end{proof}

\begin{proposition} \label{prop:only_one} Let $g\ge 2$ be even. 
  Suppose that $C$ has a unique minimal Weierstrass model $W$ and $W_k$
  is reduced. Let $K'/K$ be an extension with $e_{K'/K}=1$ and perfect residue
  field $k'$.
 \begin{enumerate}[\rm (1)] 
  \item If $C_{K'}$ has  more than one minimal Weierstrass model, then 
  $W_k$ has a quadratic point $p_0$ with $\lambda(p_0)=g+1$ and
  $k(p_0)\subseteq k'$.
\item Suppose that a $K'$ as in (1) exists. Let $K_0/K$ be an extension
  with $e_{K_0/K}=1$ and residue field equal to $k(p_0)$, then
  $C_{K_0}$ has more than one minimal Weierstrass model. 
\end{enumerate}
\end{proposition}

\begin{proof} (1) By Lemma~\ref{lem:mini_down}(2), $W':=W_{\cO_{K'}}$ is a minimal
Weierstrass model of $C_{K'}$, but not the unique one by our hypothesis.
According to \cite{LRN}, Proposition 4.6(2.b),
there exists $p'_0\in W'_{k'}(k')$ with $\lambda(p'_0)=g+1$ and 
$p'_1\in W'(p'_0)_{k'}(k')\setminus (W'(p'_0)\wedge W')$ with $\lambda(p'_1)\ge g+2$.
Let $p_0\in W$ be the image of $p'_0$. Then 
$\lambda(p_0)=g+1$ (Remark~\ref{rmk:lambda'}(\ref{lambda_1})) and
$k(p_0)\subseteq k(p'_0)=k'$.

Let us show that $p_0$ is not rational over $k$. Suppose the contrary, then 
  $W'(p'_0)=W(p_0)_{\cO_{K'}}$. Let $p_1^*\in W(p_0)\wedge W$.
  This is a rational point with $\lambda(p_1^*)=g+2$
by Equality~\eqref{eq:lambdap1*}. 
  Let $p_1\in W(p_0)$ be the image of $p'_1$. Then $p_1\ne p_1^*$.
  Lemma~\ref{lem:d0}(3) implies
  that $p_1$ is rational and $\lambda(p_1)=g+2$. Formula~\eqref{eq:disc-p0}
  implies that $W(p_0)(p_1)\ne W$ is minimal, contradiction. 
  By Lemma~\ref{lem:d0}(4), $[k(p_0):k]=2$.

  (2) We know that $W_{\cO_{K_0}}$ is minimal by Lemma~\ref{lem:mini_down}(2). If
  it is unique, then its closed fiber $W_{k(p_0)}$ has a quadratic point (over
  $k(p_0)$) with 
  $\lambda=g+1$ by applying (1) to the curve $C_{K_0}$. But the $k(p_0)$-rational
  points lying over $p_0$ already have multiplicities $\lambda=g+1$, this
  prevents the other points to having $\lambda \ge 2$ by Lemma~\ref{lem:d0}(4).
  Contradiction.
 \end{proof}

\begin{example} \label{ex:mini_nu} We can wonder if the inverse of
  Proposition~\ref{prop:only_one}(1) holds. 
  We construct below examples which shows that the answer is no in
  general. 

Let $f(x)\in \cO_K[x]$ be 
  a monic polynomial of degree $2$ such that $\bar{f}(x)\in k[x]$ is irreducible.
  Let $\ell\ge g+1$ and $h(x)\in \pi^{g+1} \cO_K[x]$ of degree $\le g+1$.
  Consider the hyperelliptic curve $C$ over $K$ of equation
  \[
y^2+h(x)y=f(x)^{g+1} + \pi^{\ell}u, \quad u\in \cO_K^*  
\]
(if it is smooth). This equation defines a Weierstrass model $W$ of $C$ with
reduced $W_k$, and the point $p_0=\{ \pi = f(x) = y=0 \}\in W_k$
is quadratic over $k$, with 
$\lambda(p_0)=g+1$. This implies that all the other points have 
multiplicity $\lambda\le 1$, hence $W$ is the unique minimal 
Weierstrass model. 
Let $K'=K[T]/(f(T))$. This is an unramified extension of $K$ with
residue field $k(p_0)$.

If $\ell\ge 2(g+1)$, then  
\[
(y/\pi^{g+1})^2+(\pi^{-g-1}h(x))(y/\pi^{g+1})=(f(x)/\pi^2)^{g+1} + \pi^{\ell-2(g+1)}u 
\]
is a minimal Weierstrass model of $C_{K'}$, different from $W_{\cO_{K'}}$, hence
$C_{K'}$ has more than one minimal Weierstrass model. 
However, if $\ell \le 2g+1$, then $W_{\cO_{K'}}$ is the unique
minimal Weierstrass model of $C_{K'}$.
  \end{example}

\subsection{Stably minimal Weierstrass models}\label{subsect:stm}
Let $W$ be a Weierstrass model of $C$. M\"uller and Stoll  
  (\cite{MuSt}, \S 5) call $W$ \emph{stably minimal} if for any finite extension 
   $K'/K$ of discrete valuation fields, $W_{\cO_{K'}}$ is minimal. They proved
   that when $g=2$, then $W$ is stably minimal if and only if
   $\delta(p)\le 3=g+1$ for all $p\in W_k$ ({\it op. cit.}, Lemma 5.3). 
   The same result can be proved in higher genus by similar methods as
   follows. 
     
 \begin{proposition} \label{prop:stm} Let $g\ge 2$. Let $W$ be a
  Weierstrass model  of $C$ with reduced fiber $W_k$.
 Then $W_{\cO_{K'}}$ is minimal for all extensions  
  of discrete valuation fields $K'/K$ if and only if
 $\delta(p)\le g+1$ for all $p\in W_k$.  
 \end{proposition}

 \begin{proof}  We will use repeatedly the minimality
   criterion \cite{LRN}, Proposition 4.3.

   Denote $W'=W_{\cO_{K'}}$. 
   Suppose that $\delta(p)\le g+1$ for all $p\in W_k$. Then for
   all $p'\in (W')_{k'}=W_{k'}$, we have
   $\lambda(p')\le \delta(p')=\delta(p)\le g+1$, 
   where $p$ is the image of $p'$ in $W_k$. Thus $W'$ is  
   minimal.

   Conversely, suppose that there exists $p_0\in W_k$ with
   $\delta(p_0)\ge g+2$. After extending $K$ (which does not change
   $\delta(p_0)$) we can assume that $p_0\in W_k(k)$. Let 
   $y^2+Q(x)y=P(x)$ be an equation of $W$ with $x(p_0)=0$. 
We notice that if $F(x)=\sum_{j\ge 0} c_j x^j\in \cO_K[x]$, 
     then for any $d\ge 1$ and any finite   
     extension $K'/K$ with $e_{K'/K}\ge d$, we have
      \[ \nu_{K'}(c_j) +j\ge d, \quad \text{for all }  j< \ord_0 \bar{F}(x). \]    
      Applying this to $Q(x)$ and $P(x)$, we see that
      if $e_{K'/K} \ge g+3$, then $\lambda(p'_0)\ge \min\{ \delta(p_0), g+3\}$ 
      for the point 
      $p'_0\in W_{k'}$ lying over $p_0$. If $g$ is even,  or $g$ is odd and
      $\delta(p_0)\ge g+3$, then $W'$ is not minimal. 

      Suppose now $g$ is odd and $\delta(p_0)=g+2$. Replacing $K$ with a
      $K'$ as above, we can assume that $\lambda(p_0)=g+2$. 
      Then an equation of $W(p_0)$ is
      \[
y_1^2+ \pi^{-(g+1)/2} Q(\pi x_1)y_1= \pi^{-(g+1)} P(\pi x_1), \quad x_1=x/\pi  
      \] 
      (\cite{LTR}, Lemme 7(d)).  Let $p_1\in W(p_0)_k$ be the point with
      $x_1=0$.  
      If $e_{K'/K}\ge 2g+4$, then similarly to the above we have
      $\lambda(p_1)=1+\delta(p_1)=g+3$. Therefore $W'$ is not minimal.
       \end{proof} 

 For a Weierstrass model $W$ such that $W_k$ is non-reduced, 
 $W_{\cO_K'}$ is not normal (hence is not a Weierstrass model in our
 terminology) as soon as $e_{K'/K}\ge 2$,  and its discriminant 
 is bigger than that of its normalization. The natural question is then
 whether the normalization of $W_{\cO_{K'}}$ (which is a Weierstrass model)
 is minimal.  

 \begin{proposition} \label{prop:stm2} Let $g\ge 2$.
   Let $W$ be a Weierstrass model of $C$ such that $W_k$ is non-reduced.
   If $\chara(k)\ne 2$, then  
the normalization of $W_{\cO_{K'}}$ is minimal for all extensions  
  of discrete valuation fields $K'/K$ if and only if
 $\delta(p)\le g+1$ for all $p\in W_k$.   
 \end{proposition}
  
 \begin{proof} We use again the minimality criterion 
   \cite{LRN}, Proposition 4.3.   

   An equation of $W$ is
   $y^2=\pi P_0(x)$ with $P_0(x)\in \cO_K[x]$ and 
   $\bar{P}_0(x)\ne 0$ in $k[x]$. For any extension $K'/K$ with uniformizing
   element $t$, an equation 
of $W'$, the normalization of $W_{\cO_{K'}}$, is 
\[ z^2=t^{e-2[e/2]} P_0(x), \quad \text{where } 
    e=e_{K'/K},  \ z=y/\pi^{[e/2]}.       
    \]
    Suppose $\delta(p)\le g+1$ for all $p\in W_k$.
    Then for any $p'\in W'_{k'}$, we have 
    \[ \lambda(p')\le (e-2[e/2])  + \delta(p') \le  (e-2[e/2]) + g+1.\]
    If $e$ is 
    even, then $\lambda(p')\le g+1$ and $W'$ is minimal. If $e$ is odd, then
    $W'_{k'}$ is non-reduced and $\lambda(p')\le g+2$, thus $W'$ is again 
    minimal. 

    Conversely, if there exists $p_0\in W_k$ with $\delta(p_0)\ge g+2$,
    then after finite extension of $K$ we can suppose $p_0\in W_k(k)$. As
    in  the proof of Proposition~\ref{prop:stm}, taking  $e_{K'/K}$ odd
    and big enough, we get
    $\lambda(p_0')=1+\delta(p'_0)\ge g+3$. Hence
    $W$ is non minimal. 
     \end{proof}
     
 \begin{remark} Proposition~\ref{prop:stm2} does not hold when 
   $\chara(k)=2$ as shows the following example. Let $K$ be an unramified
   extension 
   of $\Q_2$ with infinite residue field. 
Take a monic $F_0(x)\in \cO_K[x]$ such that $\bar{F}_0(x)\in k[x]$ is separable 
      of degree $g+1$ and $F_0(0)\in \cO_K^*$. Consider $W$ defined by
      \[ y^2=2 (F_0(x)^2+ 2 x^{g+2}).\]
      Then $\delta(p)\le 2$ for all $p\in W_k$. Let $K_1=K[\sqrt{2}]$ and 
$y_1=2^{-1}y-\sqrt{2}^{-1}F_0(x)$. We have  
      \[
y_1^2+\sqrt{2}F_0(x)y_1= x^{2g+1}.   
     \]
     As $2g+1$ is odd, this equation is reduced mod $\sqrt{2}$, hence
     defines the normalization $W_1$  of $W_{\cO_{K_1}}$.
The point $p_0:=\{ x=0 \}$ in the closed fiber of $W_1$ 
     has $\delta(p_0) = 2g+1$,  so $W_1$ is not stably minimal by
     Proposition~\ref{prop:stm}, hence the normalization of
     $W_{\cO_{K'}}$ is not minimal for some finite extension $K'$ of $K_1$. 
 \end{remark}
  
\subsection{Algorithm for determining all minimal Weierstrass models}
In  \cite{LRN}, \S 6 we gave an algorithm to find a minimal Weierstrass model of $C$. Below is an algorithm, based on Theorem~\ref{chain-mwm}, 
to find all of them.   
This algorithm has been implemented in PARI \cite{pari} by Bill Allombert.

\begin{algo} \label{algo:mwm} Suppose that $g$ is even. This algorithm
  takes a minimal Weierstrass model $U_0$ as input. The output is $U_0$ if it is
  the unique minimal Weierstrass model (Step (\ref{st50})), otherwise
  (Step (\ref{st5})) we get the extremal minimal
  Weierstrass models $W_0, W_n$ together with the points
  $p_0\in W_0\wedge W_n$ and $p_n^*\in W_n\wedge W_0$, as well as
  the number $m$ of minimal Weierstrass models.  

\begin{enumerate}[\rm (1)] 
\item If $\varepsilon(U_0)=1$ or if there is no $u\in U_0(k)$ with
  $\lambda(u)=g+1$, go to (\ref{st50}).
\item Set $U=U_0$, $S=\emptyset$, $m=1$, $r=1$.   
\item \label{st3} If there exists $u\in U(k)\setminus S$ with $\lambda(u)=g+1$. 
  \begin{enumerate}[\rm (a)]
    \item If $m=1$, set $u_0=u$. 
    \item If there exists $u_1\in U(u)(k) \setminus (U(u)\wedge U)$
      with $\lambda(u_1)=g+2$. 
      Then $m\gets m+1$,
      $S\gets U(u)(u_1)\wedge U(u)$,  $U\gets U(u)(u_1)$, go to  (\ref{st3}).  
    \item Otherwise,
      \begin{enumerate}
      \item if $r=1$, go to (\ref{st3f}).
      \item otherwise, go to (\ref{st4}). 
      \end{enumerate} 
  \end{enumerate}
\item \label{st3f} If $m=1$, go to (\ref{st50}). 
\item   If $m>1$, let $p_n^*$ be the point of $S$, $W_n=U$. 
Then $r\gets 2$, $U\gets U_0$, $S\gets \{ u_0 \}$, go to (\ref{st3}).   
 \item \label{st4} $n=2m-2$, $W_0=U$ and $p_0$ to be the point of $S$,
   go to (\ref{st5}).  
 \item \label{st50} $U_0$ is the unique minimal Weierstrass model. 
\item \label{st5} There are exactly $m$   
  minimal Weierstrass models,  $W_0$ and $W_{n}$ are the extremal ones, 
and $\{ p_0 \}=W_0\wedge W_n$, $p_n^*=W_n\wedge W_0$. 
\end{enumerate} 
\end{algo} 
\medskip

\noindent {\bf Explanations.}
We use the notation of Theorem~\ref{chain-mwm}. 

Step (1) This is a sufficient condition for $U_0$ to be unique
(\cite{LRN}, Proposition~4.6).

Step (2) We think $U_0$ as some $W_{2i_0}$.  The variable $m$ is the number
of minimal Weierstrass models.  At the first round, $r=1$, we look for $W_n$.

Step (3) We have $U=W_{2i}$.

When $r=1$, we have $i\ge i_0$. We look for the point $p_{2i}$ 
which is different from $p_{2i}^*\in S$ (empty if $U=U_0$). The conditions
$\lambda=g+1$  and that of (3.b) follow from 
Theorem~\ref{chain-mwm}. The point $u_1$ is $p_{2i+1}$. When this 
point is found, $W_{2i+2}=W_{2i}(p_{2i})(p_{2i+1})$, we get one 
more minimal Weierstrass model, and $S=\{ p_{2i+2}^* \}$.  

Step (5)  When in Step (3) there is no more $p_{2i}$ (Step (3.c.i)),
the last $U$ is $W_n$ and $S=\{ p_n^* \}$.   
Then we start the second round $r=2$ to find $W_0$. We come back to
$U_0=W_{2i_0}$. We rerun (3) for $U=W_{2i}$ ($i\le i_0$) and look for
$p_{2i}^* \in W_{2i}$ different from $p_{2i}\in S$ (which is $u_0$ when $i=i_0$). 

Step (6) The second round is terminated and we get $W_0$ as well $p_0$. 
\qed
\medskip

A practical question is then how to find the rational points
with $\lambda=g+1$ (or $\ge g+1$). The following lemma gives
strong restrictions on the potential candidates. 

\begin{lemma} Suppose that $W$, defined by Equation~\eqref{eq:start},
  is normal, and if $\chara(k)=2$,  that $(Q, P)$ satisfies the
  conditions of Lemma 1.3 in \cite{LRN}.  Let
  $\varepsilon=\varepsilon(W)$.
  Let $p_0\in W_k(k)$ with $x$-coordinate
  $\bar{c}$. Suppose that $\lambda(p_0)\ge d$ for some positive
  integer $d$.
\begin{enumerate}[\rm (1)] 
\item If $\chara(k)\ne 2$, then $\ord_{\bar{c}} \overline{\pi^{-\varepsilon}(4P+Q^2)} \ge d-\varepsilon$.
\item Suppose $\chara(k)=2$.
  \begin{enumerate}
  \item If $\bar{Q}\ne 0$, then $\ord_{\bar{c}} \bar{Q}\ge d/2$.
  \item If $\bar{Q}=0$ and $P\notin k[x^2]$, then
    $\ord_{\bar{c}} \bar{P}' \ge d-1$.
  \item If $\varepsilon=1$, then 
    $$\ord_{\bar{c}} \overline{\pi^{-1}Q}\ge (d-2)/2,\quad
    \ord_{\bar{c}} \overline{\pi^{-1}P}\ge d-1. $$ 
  \end{enumerate}
\end{enumerate}
\end{lemma}

\begin{proof} Same as that of \cite{LRN}, Lemma 3.5. 
\end{proof}

\end{section}

\begin{section}{Applications to the regular models and the stable reduction}
  \label{MRM} 

  It turns out that having more than one minimal Weierstrass
  model implies some interesting consequences on $C$. Suppose further that $g$
  is even. Let $W_0, W_n$ be the extremal Weierstrass models
  (Definition~\ref{extremal-MWM}). Then the canonical model
  (hence the minimal regular model) of $C$
  dominates $W_0$ and $W_n$ (Proposition~\ref{prop:positive}). The
  model $W_0\vee W_n$ (Definition~\ref{dfn:ccW}(3)) 
  is semi-stable at the point lying $p_{0,n}$ over $W_0\wedge W_n$,
  and it is solved by a chain of $(n-2)/2$ projective 
  lines $\PP^1_k$ (Theorem~\ref{regular-even}). The resolution of
  the other singular points of $W_0\vee W_n$ have the same shape as
  for singular points arising from curves of genus $g/2$ (Corollary~\ref{cor:gl}). 

  Over the completion $\widehat{K}$ of $K$, $C_{\widehat{K}}$ can be decomposed, 
  as rigid analytic space, as the union of two affinoid curves (plus an
  open annulus), each of
  them can be embedded in a smooth projective curve of genus $g/2$ over
  $\widehat{K}$. 
  This implies that the (potential) stable reduction of $C$ is union of two
  stable curves of genus $g/2$ (Corollary~\ref{cor:stable}). 

\subsection{Blowing-up  Weierstrass models and regularity}

Let $W$ be a Weierstrass model of $C$. In this subsection we do not
assume $g$ is even, neither $W$ is minimal.   Let $p_0\in W_k(k)$ be a singular
point of $W$. The first step in Lipman's desingularization sequence
\cite{Lip} consists in blowing-up $W$ along $p_0$ and then take the
normalization. Let
 \[ \wW\to W\]
 be the resulting morphism of models of $C$.
 Then $\wW=W\vee W(p_0)$ (see \cite{LTR}, \S 6.3).  Its closed fiber is the union of the strict transforms of $W_k$ and of $W(p_0)_k$, intersecting at some 
 points lying over $p_0$. In this subsection we give conditions for  
 $\wW$ to be regular at these points (Proposition~\ref{prop:reg}). 
 The results will be used for the extended Tate's algorithm
 Proposition~\ref{prop:exT} and Theorem~\ref{regular-even}. 
 
Consider an equation 
\begin{equation}
  \label{eq:Eq2}
  y^2+Q(x)y=P(x) 
\end{equation} 
of $W$  with $x(p_0)=0$ and satisfying the conditions of
 Proposition~\ref{prop:same-y} 
 with $n=0$ (with $Q(x)=0$ if $\chara(k)\ne 2$).  
 We will first give a local equation of $\wW$ around the above intersection
 points. Let $r=[\lambda(p_0)/2]$.
 Write $Q(x)=\sum_{j\ge 0} b_jx^j$ and $P(x)=\sum_{j\ge 0} a_jx^j $ and  
 \begin{equation}
   \label{eq:ajr}
    b_{j}=b_{j,r-j}\pi^{r-j} \ (j\le r-1); \quad
  a_j=a_{j, 2r-j}\pi^{2r-j} \ (j\le 2r-1).
 \end{equation}
  We have $b_{j,r-j}, \ a_{j,2r-j} \in {\cO_K}$.
  Put
  \[ v=\pi/x, \quad z=y/x^r \] 
  and divide Equation~\eqref{eq:Eq2} above by $x^{2r}$. Then we get the equation 
  $F(x,v,z)=0$ where
  \begin{multline}  \label{eq:zxv} 
F(x,v,z):= z^2+\left(\sum_{j\le r-1} b_{j,r-j} v^{r-j} + \sum_{j\ge r} b_j
x^{j-r}\right)z  \\
-\left(\sum_{j\le 2r-1} a_{j,2r-j}v^{2r-j}+\sum_{j\ge 2r} a_j x^{j-2r}\right).      
  \end{multline}
  
 \begin{lemma} Let $A=\cO_K[x,v]$ with the relation $xv=\pi$, and let $B=A[z]$ with
$F(x,v,z)=0$. Then $B$ is the integral closure of $A$ in $K(C)$. 
\end{lemma}

  \begin{proof} (This is a special case of \cite{Ld}, Proposition 3.3.) Clearly $B$ is finite and free over the regular ring $A$.
 To prove the lemma it is enough to check that  it is integrally closed. As $B$ is a complete intersection, it is enough to check that it is integrally closed at prime ideals  of height $1$ (Serre's criterion $(R_1)+(S_2)$, see \cite{LB}, Theorem 8.2.23).

 As $B\otimes_{\cO_K} K=K[x, 1/x, y]\subset K(C)$ is integrally closed, we only have to consider prime ideals $\p\subset B$ of height $1$ containing $\pi$.
 So $x\in \p$ or $v\in \p$. Suppose for instance $x\in \p$.  The
 proof is the same when $v\in \p$. 
 We have $v\notin \p$ because $B/(x, v)$ is Artinian, hence $(x, v)$ has
 height $0$ in $B$. 
 Therefore $B_\p$ is a  localization of $B_v$. The latter defines an
 affine open subscheme of $W(p_0)$, so $B_\p$ is integrally closed. 
  \end{proof}

  Let $\wZ=\wW/\qi=Z\vee Z(q_0)$, where $q_0\in Z=W/\qi$ is the image of $p_0$. 
  This is a semi-stable regular model of $C/\qi$.
  Denote by $\tilde{q}_0\in \wZ(k)$ the intersection point of the two irreducible
  components of $\wZ_k$. See Figure~\ref{inter-ZZ'} with $Z'=Z(q_0)$.  
  An affine open neighborhood of $\tilde{q}_0$
 in $\wZ$ is $\Spec \cO_K[x,v]/(xv-\pi)$, with $\tilde{q}_0$ of coordinates
 $x=v=\pi=0$.  

\begin{proposition} \label{prop:reg} Keep the above notation. Recall that $r:=[\lambda(p_0)/2]$. Let $p_1^*\in W(p_0)\wedge W$. 
\begin{enumerate}[\rm (1)]
    \item The morphism $\wW\to \wZ$ is \'etale over $\tilde{q}_0$ if and only if
\[\varepsilon(W)=0, \quad  \delta(p_0)=\lambda(p_0)=2r. \]  
 \item \label{item:232} Suppose that $\wW\to \wZ$ is not \'etale over $\tilde{q}_0$. 
   Then there is a unique point  $\tilde{p}_0\in \wW$ lying over $\tilde{q}_0$,
   rational over $k$. 
The scheme $\wW$ is regular at $\tilde{p}_0$ if and only if 
\[
   \varepsilon(W)=0, \quad \delta(p_0)=2r+1, 
 \]
 or, setting $r^*=[\lambda(p_1^*)/2]$,
 \[ \varepsilon(W(p_0))=0, \quad \delta(p_1^*)=2r^*+1.    \]
 If $\delta(p_0)<\lambda(p_0)$, then $\wW$ is regular at $\tp_0$. 
\item \label{reg-net} Suppose that $\wW$ is regular at $\tp_0$, but not \'etale over
  $\tilde{q}_0$. Then the strict transform $\Gamma$ of $(W_k)_{\mathrm{red}}$ in
  $\wW$ is smooth at $\tp_0$ if and only if either $\varepsilon(W)=1$ or, 
 $\varepsilon(W)=0$ and $\delta(p_0)=2r+1$. 
\item Keep the hypothesis of (\ref{reg-net}).
     Then $\wW_k$ has normal  crossings at $\tp_0$ if and only if
    $\varepsilon(W)=1$ or  $\varepsilon(W(p_0))=1$.
\end{enumerate} 
\end{proposition} 

\begin{proof} We restrict ourselves to  the (more difficult) case $\chara(k)=2$.
  Then by Proposition~\ref{prop:same-y}(2) we have
  $\boxed{a_{2r}\in \pi\cO_K}$.  

  (1) The fiber of $\wW\to \wZ$
  above $\tilde{q}_0$ is $\Spec k[z]/(z^2+\bar{b}_rz-\bar{a}_{2r})$. So being
  \'etale is equivalent to $\bar{b}_r\in k^*$. 
  This implies that $\varepsilon(W)=0$ and $\lambda(p_0)=\delta(p_0)=2r$. 

  Conversely, assuming that  the above equalities hold, we have to show that
  $b_r\in \cO_K^*$.
  Suppose $b_r\in \pi\cO_K$. Then $\ord_0(\bar{Q}(x))\ge r+1$. As
  $\varepsilon(W)=0$, this implies that $\ord_0(\bar{P}(x))=2r$, hence
  $a_{2r}\in \cO_K^*$, absurd.
  \smallskip
  
  (2) The computations in (1) show that the pre-image of
  $\tilde{q}_0$ in $\wW$ is then a single rational $\tp_0$ and
  that  $\boxed{b_r\in \pi\cO_K}$. 
  Locally at $\tp_0$, $\wW$ is the closed subscheme of
the affine $3$-space $\Spec \cO_K[x,v,z]$ defined by 
\[ xv-\pi=0, \quad F(x,v,z)=0\] 
where (see Equality~\eqref{eq:zxv}) 
\begin{equation} \label{eq:zxv2}
   F(x,v,z)=z^2+a_{2r+1}x+a_{2r-1,1}v+ \epsilon_2(x,v,z),
\end{equation} 
for some $\epsilon_2(x,v,z)\in (x^2, xv, v^2, xz, vz)\cO_K[x,v,z]$. 
The point $\tp_0$ is $x=v=z=\pi=0$. It is regular in $\wW$  
    if and only if
    \begin{equation}
      \label{eq:ar_reg}
    a_{2r+1}\in \cO_K^* \ \text{ or } \ a_{2r-1,1}\in \cO_K^*.   
    \end{equation}

    (2.1) It is easy to see that $a_{2r+1}\in \cO_K^*$ if and only if
    $\varepsilon(W)=0$  
    and $\delta(p_0)=2r+1$ (note that $\ord_0 (\bar{Q}(x))\ge r+1$).

    (2.2) Suppose that $a_{2r+1}\in \pi\cO_K$. 

    If $\varepsilon(W(p_0))=0$  and $\delta(p_1^*)=2r^*+1$, then
    by symmetry ($W=W(p_0)(p_1^*)$), $\tp_0$ is regular.

    (2.3) Suppose $a_{2r+1}\in \pi \cO_K$ and 
    $a_{2r-1,1}\in  \cO_K^*$. Then $\lambda(p_0)=2r$ and  
    $\varepsilon(W(p_0))=0$. Let us show $\delta(p_1^*)=2r^*+1$.
    By Equality~\eqref{eq:lambdap1*}, 
   $\lambda(p_1^*)=2g+2-2r+\varepsilon(W)$,  
    hence $r^*=g+1-r$. An equation of $W(p_0)$ is \eqref{eq:at_p1*}. 
    With the notation of \eqref{eq:Qinf_Pinf}, the conditions
    $\lambda(p_0)=2r$, $a_{2r-1,1}\in \cO_K^*$ and
    $a_{2r}, b_r\in \pi \cO_K$ imply that 
\[ \ord_0\bar{P}_\infty(v)= 2g+3-2r< 2\ord_0\bar{Q}_\infty(v). 
\] 
Therefore $\delta(p_1^*)=2g+3-2r=2r^*+1$ by 
Lemma~\ref{lem:change_y}(\ref{crit_delta}).

(2.4) If $\delta(p_0)<\lambda(p_0)$, then $\varepsilon(W)=1$ and
$\delta(p_0)=2r-1$, hence $a_{2r-1,1}\in \cO_K^*$ and $\wW$ is regular at $\tp_0$. 
    \smallskip

    (3) Locally at $\tp_0$, $\wW_k$ is the subscheme of $\Spec k[x,v,z]$ defined by
the equations
 \begin{equation}\label{eq:ww_k}   
   xv=0, \quad z^2+\bar{a}_{2r+1}x+\bar{a}_{2r-1,1}v+\bar{\epsilon}_2(x,v,z)=0
 \end{equation} 
 (see Part (2)).  

 If $\varepsilon(W)=1$, then $\Gamma \to (W_k)_{\mathrm{red}}\simeq \PP^1_k$
is finite birational, hence an isomorphism.  Suppose now $W_k$ reduced. 
Then $\Gamma$ is  defined by the equations 
\[ 
  v=0, \quad z^2+\bar{a}_{2r+1}x+\bar{\epsilon}_2(x,0,z)=0 
\]
with $\bar{\epsilon}_2(x,0,z)\in (x^2, xz)k[x,z]$.
Now $\Gamma$ is smooth at $\tp_0$ if and only if
$\bar{a}_{2r+1}\ne 0$. We saw in (\ref{item:232}.1) that this is equivalent
to $\delta(p_0)=2r+1$. 
\smallskip

(4) First suppose that
$\varepsilon(W(p_0))=1$. By Part (2) we have $\varepsilon(W)=0$. In (2.1)
we saw that $\bar{a}_{2r+1}\ne 0$. Let $\Gamma$ 
be the strict transform of $(W(p_0)_k)_{\mathrm{red}}$. Using
Equation~\eqref{eq:ww_k}, we see that $\Gamma'=V(x,v)$. Hence 
the intersection number $\Gamma.\Gamma'$ at $\tp_0$ is 
\[\dim_k (k[x,v, z]/(v, z^2+\bar{a}_{2r+1}x+\bar{\epsilon}_2(x,0, z), x, z))=1.\]   
By symmetry the same result holds if $\varepsilon(W)=1$.

Suppose now that $\varepsilon(W)=\varepsilon(W(p_0))=0$. Then
the intersection number is 
\[\dim_k (k[x,v, z]/(v, z^2+\bar{a}_{2r+1}x+\bar{\epsilon}_2(x,0, z),
  x, z^2+\bar{a}_{2r-1,1}v+ \bar{\epsilon}_2(0,v,z)))=2 
\]
because the ideal we quotient by is $(x,v,z^2)$. Hence the intersection is not
transverse. 
 \end{proof}

\subsection{Extended Tate's algorithm} \label{subsect:eTa} Let $\cE$ be a
pointed Weierstrass model of an elliptic curve $E$ (``pointed'' means that
the neutral element of
$E$ lies in the smooth locus of $\cE$). Then any singular point $p_0\in \cE$
is rational over $k$ and $\delta(p_0)\le 3$. Tate's algorithm \cite{Tate}
determines the minimal desingularization of $\cE$ at $p_0$. As we will see
in Proposition~\ref{prop:exT}, this algorithm relies only on the condition 
$\delta(p_0)\le 3$ and not on the genus of the generic fiber. 
See also Remark~\ref{rmk:e1d3} for the case when $W_k$ is non-reduced.

More precisely, let $W$ be a Weierstrass model of $C$ such that $W_k$
is reduced and let $p_0\in W_k(k)$ with $\delta(p_0)\le 3$.  Then
the minimal desingularization of $W$ at $p_0$ follows  the same steps
as in Tate's algorithm for elliptic curves \cite{Tate}. 
The final result has the same shape as for elliptic curves
(Corollary~\ref{cor:exT}). This is not surprising at least when
$\chara(k)\ne 2$ because $\Spf \widehat{\cO}_{W, p_0}$ can then be
embedded into the formal completion of a Weierstrass model of
some elliptic curve. I do not know whether this holds when 
$\chara(k)=2$. At any rate, it is useful to have an explicit desingularization
algorithm.  The ours, as that of Tate,  consists just in computing repeatedly
multiplicities $ \lambda$ and $\delta$.

\begin{lemma}\label{lem:l2d1} Let $W$ be a Weierstrass model of $C$ with
  non-reduced $W_k$, let $p_0\in W_k$ with $\delta(p_0)=1$. Then
  $\lambda(p_0)=2$ and the blowing-up $\wW\to W$ along $p_0$ solves the
  singularity. The exceptional locus is $\PP^1_{k(p_0)}$ of multiplicity
  $1$ in $\wW_k$, and it intersects transversely the strict transform of
  $W_k$.
\end{lemma}

\begin{proof} The equality $\lambda(p_0)=2$ comes from
Lemma~\ref{lem:d0}(2). To prove the second part of the lemma we can extend $K$ by an unramified
extension and suppose that $p_0\in W_k(k)$. Let
\[ y^2+(\sum_j b_j x^j)y=\sum_j a_j x^j \]
be an equation of $W$ satisfying the condition of Proposition~\ref{prop:same-y}
for the chain $W, W_1=W(p_0)$. As $W_k$ is non-reduced and $\delta(p_0)=1$,
we have $\pi \mid b_j$, $\pi^2\mid a_0$. We can write the above equation as
\[
y^2+\pi(a_{1,1}x+ b_{0,1}y+\pi a_{0,2})+ \epsilon_3=0 
\]
with $\epsilon_3\in (\pi, x, y)^3\cO_K[x,y]$ and $a_{1,1}\in \cO_K^*$
(see Notation~\eqref{eq:ajr}). So the
tangent cone is nondegenerate. The desingularization of $W$ at 
$p_0$ is well known.
\end{proof}

Let  $f: \mathcal X\to W$ be the minimal desingularization at $p_0$.
Let $\cK$ be a Kodaira symbol for elliptic curves
(see \cite{LB}, \S 10.2.1).   If $W$ is singular at $p_0$, we will say that  
\emph{$(W, p_0)$  has type $\cK$} if the exceptional locus of $f$ 
has the same configuration 
as the exceptional locus of the minimal desingularization
of the minimal Weierstrass model of some elliptic curve over $K$, having
reduction of type $\cK$. See Figure~\ref{figure:KN} for an example.
By same configuration, we mean that there is
an isomorphism between the two exceptional loci as curves over $k$
endowed with the reduced structure, and the isomorphism preserves
the multiplicities of the irreducible components. With some extra efforts
(especially for types III$^*$ and II$^*$), one could prove that they are
isomorphic as $k$-curves, but we do not need such a statement here.
Moreover,   except the types III and IV$_2$, the irreducible components of 
the exceptional locus of $f$ have intersection number $0$ or $1$ 
with the strict transform of $W_k$. 

In the statement of the next proposition we will only indicate the
type over an algebraic closure $\bar{k}$ of $k$. The precise
type over $k$ (with the notation of \cite{LB}, \S 10.2.1) is given in
the course of the proof. 

\begin{figure}    
\centering 
\begin{tikzpicture}[scale=1] 
\draw (-1.4, 0.2) node {$\cX_{\bar{k}}:$}; 
  \coordinate (P3) at  (-0.2, 1.3);
  \coordinate (Q3) at  (1.25, 0.8);
  \draw  (P3)  [below left]  to [bend left] (Q3);
  \draw (-0.2, 1.6) node {$\Gamma_0$};
  \draw[dashed] (0.4, 1.5) to [bend left] (-0.5, 0.8);
  \draw[dashed] (0.6, 1.4) to [bend left] (-0.5, 0.6); 
  \draw (0.5, 0) -- (3.7, 0);
  \draw (1.5, -0.2) node {3};  
  \draw (1, -0.5) -- ( 1, 1.6);
  \draw (0.8, 0.5) node {2};  
  \draw (2, -0.5) -- ( 2, 1.6);
  \draw (1.8, 0.5) node {2};  
\draw (1.6, 1) -- (2.4, 1);  
  \draw (3, -0.5) -- ( 3, 1.6); 
  \draw (2.8, 0.5) node {2};  
  \draw (2.6, 1) -- (3.4, 1);
\end{tikzpicture} 
\caption{Type IV$^*$ for $(W, p_0)$.}
\label{figure:KN}
\end{figure} 

\begin{proposition}\label{prop:exT} Let $W$ be a Weierstrass model with
  reduced $W_k$, and let  $p_0\in W_k(k)$ with $\delta(p_0)\le 3$. 
  \begin{enumerate}[\rm (1)]
  \item If $\delta(p_0)\le 1$, then $W_k$ is smooth at $p_0$.
  \item Suppose $\delta(p_0)=2$. Then $p_0$ is an ordinary double point in
    $W_k$, and $(W, p_0)$ has type {\rm I}$_\nu$  over $\bar{k}$
    for some   $\nu\ge 2$ if $\lambda(p_0)\ge 2$. 
  \item If $\delta(p_0)=3$ and $\lambda(p_0)=1$, then $p_0$ is a cusp
    (similar to the type {\rm II}). 
  \item If $\delta(p_0)=3$ and $\lambda(p_0)=2$, then $(W, p_0)$ has type
    {\rm III} or {\rm IV} over $\bar{k}$. 
  \item Suppose $\delta(p_0)=\lambda(p_0)=3$. 
    Let $W_1=W(p_0)$ and let $p_1^*$ be the point of $W_1\wedge W$.
    \begin{enumerate}[\rm (a)] 
    \item If $\delta(p)\le 1$ for all $p\in (W_1)_k\setminus \{ p_1^*\}$,
      then $(W, p_0)$ has type ${\rm I}_0^*$ over $\bar{k}$.
    \item If there exists $p_1 \in (W_1)_k\setminus \{ p_1^*\}$ with
      $\delta(p_1)=2$, then $(W, p_0)$ has type {\rm I}$^*_\nu$ over $\bar{k}$
      for some $\nu\ge 1$. 
    \item Suppose that there exists $p_1 \in (W_1)_k\setminus \{ p_1^*\}$ with
      $\delta(p_1)=3$.  Then $p_1$ is rational over $k$. Let $W_2=W_1(p_1)$. 
      \begin{enumerate}[\rm (i)] 
      \item If $\lambda(p_1)=2$, then $(W, p_0)$ has type {\rm IV}$^*$
        over $\bar{k}$. 
      \item If $\lambda(p_1)=3$ and there exists $p_2\in W_2 \setminus
        \{ p_2^* \}$ with $\delta(p_2)=1$, then $(W, p_0)$ has type
        {\rm III}$^*$.
      \item If $\lambda(p_1)=3$ and $\delta(p)=0$ for all $p\in W_2 \setminus
        \{ p_2^* \}$, then $(W, p_0)$ has type {\rm II}$^*$.
      \item If none of the above conditions is satisfied, then $\lambda(p_1)=4$, 
    \[
\nu(\Delta_{W_2})=\nu(\Delta_{W})+4(2g+1)(g-2), 
\]
and $W\vee W_2$ is semi-stable and regular at $p_{0,2}$, the point lying over
$p_0$ and $p_2^*$.   Moreover, there is at worst one point
$p_2\in (W_2)_k\setminus \{ p_2^* \}$ with $\delta(p_2)\ge 2$, and we then
have $\delta(p_2)=2, 3$ and $p_2$ rational over $k$.  
\end{enumerate}
\end{enumerate}
\end{enumerate}  
\end{proposition}

\begin{proof} We indicate in boldface the corresponding stage in Tate's algorithm \cite{Tate}. See also \cite{Silv}, \S IV.9 for more details.
Let $W_1=W(p_0)$ and let  
\[ y^2+(\sum_j b_j x^j)y=\sum_j a_j x^j\]
be an equation of $W$ satisfying the conditions of
Proposition~\ref{prop:same-y} for the chain $W, W_1$.
As in \cite{Tate} and in \eqref{eq:ajr}, for $i\ge 1$,
we set $b_{j,i}=\pi^{-i}b_j$ and $a_{j,i}=\pi^{-i}a_j$.  

(1) ({\bf Step 1}) follows from Lemma~\ref{lem:d0}(1). 
\smallskip 

(2) ({\bf Step 2}) We can suppose $\lambda(p_0)\ge 2$ (hence $=2$).
Then $W\vee W_1$ is semi-stable and regular at $\tp_0$ 
(Proposition~\ref{prop:reg}(1)).
An equation of $(W_1)_k\setminus \{ p_1^* \}$ is 
\[ 
\bar{y}_1^2+(\bar{b}_{0,1}+\bar{b}_1\bar{x}_1) \bar{y}_1 =
\bar{a}_{0,2}+\bar{a}_{1,1}\bar{x}_1+\bar{a}_2\bar{x}_1^2 .
\]

(2.1) If the above affine conic is smooth, then we have type I$_2$. Otherwise
it has a unique singular point $p_1$,  and the latter is rational over $k$ with
$\delta(p_1)=2$.  We have $(W_1)_k$ reduced because
$\lambda(p_0)=2$ is even.  
We repeat the same operation with $(W_1, p_1)$. After finitely many steps
(Corollary~\ref{cor:dl}), we get a $W_{\ell}$, $\ell \ge 1$,  with no singular
point away from $p_\ell^*$.

We now determine the precise type over $k$. It depends on whether the
polynomial 
\[ f_0(T)=T^2+\bar{b}_1T-\bar{a}_2\in k[T]\]
is split or not. One can check that this polynomial is separable
because $\delta(p_0)=2$.  
Modify our initial equation of
$W$ to satisfy the conditions of Proposition~\ref{prop:same-y}(1)
for the chain $W, W_1, \dots, W_\ell$. This does not change the values of
$\bar{b}_1, \bar{a}_2$. Let $1\le i\le \ell$. The equation
of $(W_i)_{k} \setminus \{ p_i^*\}$ is 
\[
\bar{y}^2_i+(\bar{b}_{0,i}+\bar{b}_1\bar{x}_i) \bar{y}_i =
\bar{a}_{0,2i}+\bar{a}_{1,i}\bar{x}_i+\bar{a}_2\bar{x}^2_i.
\] 

(2.2) If $f_0(T)$ is split, then 
the above affine conic has two rational points at infinity. So it is 
union of two affine lines over $k$ intersecting at $p_{i}$ if $i<\ell$. The 
type is I$_{2\ell}$ if $W_\ell$ is smooth away from $p^*_\ell$, and
is I$_{2\ell+1}$ otherwise.  

(2.3) If $f_0(T)$ is irreducible, it 
defines a quadratic extension $k'/k$. The above equation defines a
projective conic with only one (quadratic) point at $\bar{x}_i=\infty$. Its
normalization is isomorphic 
to $\PP^1_{k'}$ if $i<\ell$. The type is I$_{2\ell, 2}$ if $W_\ell$ is
smooth away $p^*_\ell$ and is I$_{2\ell+1, 2}$ otherwise.  
\smallskip

(3) ({\bf Step 3}) The cumulative conditions
$\boxed{\pi \mid {\bf a}_3, {\bf a}_4, {\bf a}_6, {\bf b}_2}$ in \cite{Tate} 
(coefficients in boldface are that in \cite{Tate}) 
   means $\delta(p_0) \ge 3$. Under our hypothesis,
  $\boxed{\delta(p_0)=3}$. If $\lambda(p_0)=1$, then $W$ is regular   at $p_0$.
  The point $p_0$ is a cusp because $\delta(p_0)=3$ is odd. So the type is II. 

  (4) ({\bf Steps 4, 5}) The further condition $\boxed{\pi^2 \mid {\bf a}_6}$
  in \cite{Tate} means that $\lambda(p_0)\ge 2$. So
$\lambda(p_0)=2$ or $3$ as it can't exceed $\delta(p_0)$. 
By Proposition~\ref{prop:reg}(2), $W\vee W_1$ is regular at the point $\tp_{0}$
lying over $p_0$ and $p_1^*$, and the strict transform 
$\Gamma_0$ of $W_k$ is smooth at point lying over $\tp_{0}$.
Its local equation (see Equation~\eqref{eq:zxv}) is  
\begin{equation} \label{eq:step45}
  z^2+b_{0,1}vz=a_{0,2}v^2+a_{1,1}v+a_2+a_3x+ {\epsilon}_2(x,v,z)
\end{equation} 
with ${\epsilon}_2(x,v,z)\in (xz, x^2)$. 
An equation of $(W_1)_k\setminus \{ p_1^* \}$ is 
\[ \bar{y}_1^2+\bar{b}_{0,1}\bar{y}_1=
  \bar{a}_{0,2}+\bar{a}_{1,1}\bar{x}_1\]
because $\delta(p_0)=3$ implies that $\boxed{\bar{b}_1=\bar{a}_2=0}$. 

Suppose that $\boxed{\lambda(p_0)=2}$.
If $\bar{a}_{1,1}\ne 0$, the above equation defines a smooth 
conic, transverse to $\Gamma_0$ by looking at Equation~\eqref{eq:step45}.
The type is III.  

If $\bar{a}_{1,1}=0$, then the type is IV if the polynomial
$T^2+\bar{b}_{0,1}T-a_{0,2}\in k[T]$ is split, and is IV$_{2}$ otherwise
(we use Equation~\eqref{eq:step45} to study the intersection multiplicities
at infinity).  
Note that this polynomial is separable because  $(W_1)_k$ is reduced.
\smallskip

(5) ({\bf \S 8. - Algorithm continued}) Suppose $\boxed{\lambda(p_0)=3}$.
This is the further condition $\boxed{\pi^3 \mid {\bf b}_6}$ in \cite{Tate}.  
By Proposition~\ref{prop:reg}(4), $W\vee W_1$ is regular 
at $\tp_0$ and $(W\vee W_1)_k$ has normal crossings at $\tp_0$. 
Recall that $\delta(p_0)=3$. An equation of $W_1$ is
\begin{equation} \label{eq:step5}
  y_1^2+\pi( b_{0,2}+x_1R(x_1)) y_1 =
  \pi (a_{0,3}+a_{1,2}x_1+a_{2,1}x_1^2+ a_3x_1^3+
  \pi x_1^4 S(x_1)) 
\end{equation} 
with $R(x_1), S(x_1)\in \cO_K[x_1]$ and we have
$a_3\in \cO_K^*$ because $\delta(p_0)=3$. 
So the possible singular points of 
$W_1 \setminus \{ p^*_1\}$ correspond to the zeros of the degree $3$ polynomial
\[
p(\bar{x}_1):=\bar{a}_{0,3}+\bar{a}_{1,2}\bar{x}_1+\bar{a}_{2,1}\bar{x}_1^2+
\bar{a}_3\bar{x}_1^3\in k[\bar{x}_1]. 
\] 

(5.a) ({\bf First branch: 6)}) Suppose $p(\bar{x}_1)$ is separable.
Then $\delta(p)\le 1$ for all  
$p\in (W_1)_k \setminus \{ p_1^* \}$. Any point with $\delta(p)=1$
is singular and is solved as in Lemma~\ref{lem:l2d1}. The type for
$(W, p_0)$ is I$_0^*$ if $p(x)$ is split, is I$^*_{0,2}$ if
$p(x)$ has a unique root in $k$ (so the other one is quadratic),
and is I$^*_{0,3}$ if $p(x)$ is irreducible. 
\smallskip

(5.b) ({\bf Second branch: 7)}) Suppose that $p(\bar{x}_1)$
has a simple root and 
a double roots (then they are both in $k$). The simple root gives rise
to a singular point of $W_1$ with $\delta(p)=1$ and rational over $k$, it is 
solved as in Lemma~\ref{lem:l2d1}. The double root gives rise to
a rational point $p_1$ with  $\boxed{\delta(p_1)=2}$. 
By Lemma~\ref{lem:d0}(2), we have $\lambda(p_1)=2$ or $3$, hence
$[\lambda(p_1)/2]=1$. 
Let $W_2=W_1(p_1)$. Modify our initial equation of $W$ to make it
satisfying the conditions of Proposition~\ref{prop:same-y}(1) for the
chain $W, W_1, W_2$.  Then $\boxed{\pi \mid a_{0,3}, a_{1,2}}$,  and
$\boxed{a_{2,1}\in \cO_K^*}$.

Similarly to Equality~\eqref{eq:zxv}, 
a local equation of $W_1\vee W_2$ at the point
$\tp_1$ lying over $p_1$ and $p_2^*$ is
\[
z^2+b_{0,2}vz-a_{0,4}v^2-a_{2,1}x_1v + \epsilon_3=0, \quad \epsilon_3 \in
(x_1,v,z)^3 
\] 
with $x_1v=\pi$. So $W_1\vee W_2$ is singular at $\tp_1$. It is easy to
check that it is solved by a single blowing-up along $\tp_1$, and the
exceptional divisor is $2\PP^1_k$, intersecting transversely the strict
transform of $(W_1)_{k}$ at a rational point, and that of $(W_2)_k$ at another
rational point. It remains to solve the singularities 
of $W_2\setminus \{ p_2^*\}$. An equation of
$W_2$ is 
\[
y_2^2+(b_{0,2}+\pi x_2R_2(x_2)) y_2=a_{0,4}+\pi a_{1,3}x_2+\pi a_{2,1}x_2^2+
\pi^2 x_2^3S_2(x_2) 
\]
with $R_2(x_2), S_2(x_2)\in \cO_K[x_2]$.

(5.b.1) Suppose that $\boxed{\lambda(p_1)=2}$. Then $(W_2)_k\setminus \{ p_2^*\}$
is the (reduced) affine conic of equation  
\[ \bar{y}_2^2+\bar{b}_{0,2}\bar{y}_2= \bar{a}_{0,4}. \] 
The type of $(W, p_0)$ is I$^*_1$  or  I$^*_{1,2}$ depending on
whether $T^2+\bar{b}_{0,1}T-\bar{a}_{0,2}\in k[T]$ has split or not. 

(5.b.2) Suppose now that $\boxed{\lambda(p_1)=3}$. Then
$\nu(b_j)+2j\ge 3$ and $\nu(a_j)+2j\ge 5$ for all $j\ge 0$.
Similarly to the beginning of (5), the singular points of $W_2\setminus \{ p^*_2 \}$ depends on the polynomial 
\[ \bar{a}_{0,5}+\bar{a}_{1,3}\bar{x}_2+\bar{a}_{2,1}\bar{x}_2^2\in
k[\bar{x}_2].\]  
The type of $(W, p_0)$ is I$^*_2$ or I$^*_{2,2}$ if this polynomial has
two distinct roots in $k$ or is irreducible.  

(5.b.3) Suppose that the above polynomial has a double root. Then
we have a point $p_2\in (W_2)_k$ with $\delta(p_2)=2$.
As $\varepsilon(W_2)=1$, the pair $(W_2, p_2)$ satisfies the
conditions of (5.b) 
and we can start again with the above process. This process eventually
terminates by Corollary~\ref{cor:dl}. At the end we get a type
I$^*_{\nu}$ or I$^*_{\nu,2}$ for some $\nu\ge 1$. 
\smallskip 

(5.c) ({\bf Branch 3 begins: 8)}). Suppose that $p(\bar{x}_1)$ has a
triple root. It induces a rational point $p_1$ such that
$\boxed{\delta(p_1)=3}$. Modify the equation of 
$W$ to satisfy the conditions of Proposition~\ref{prop:reg}(1) for the
chain $W, W_1, W_2$.
Then $\boxed{\bar{a}_{0,3}=\bar{a}_{1,2}=\bar{a}_{2,1}=0}$. 

(5.c.i) An equation of $(W_2)_k \setminus \{ p^*_2 \} $ is then 
\[
\bar{y}_2^2 + \bar{b}_{0,2} \bar{y}_2 = \bar{a}_{0,4}. 
\] 
Suppose $\lambda(p_1)=2$.
The polynomial $T^2+\bar{b}_{0,2}T-\bar{a}_{0,4}\in k[T]$ is then separable.
Hence $(W_2)_k \setminus \{ p_2^* \}$ is smooth. 
One can check that after blowing-up twice the singular point belonging to the
strict transform of $W_k$ there is no more singular point. The first blowing-up
introduces two irreducible components of multiplicity $2$, and the last one an
irreducible component of multiplicity $3$. The type of $(W, p_0)$ is IV$^*$
if $T^2+\bar{b}_{0,2}T-\bar{a}_{0,4}$ is split and is IV$^{*}_{2}$ otherwise. 

(5.c.ii)-(5.c.iii) ({\bf Branch 3 continues: 9) and 10)}). 
When $\lambda(p_1)=3$, one can check 
(and I did) that the successive blowing-ups of the singular points
produce the same result as the output of Tate's algorithm.

(5.c.iv) The cumulative hypotheses are $\delta(p_0)=\lambda(p_0)=3$ and
$\delta(p_1)=3$. As $\lambda(p_1)\le \delta(p_1)+1$, we have
$\lambda(p_1)=4$. The formula on $\Delta_{W_2}$ is obtained by 
applying twice Formula~\eqref{eq:disc-p0}.

Let us show that the strict transform of $(W_1)_k$ in 
$W\vee W_1 \vee W_2$ is $2\PP^1_k$ 
intersecting transversely each of the strict transforms of $W_k$ and 
$(W_2)_k$ at a single rational point.
This will imply that $W\vee W_2$ is semi-stable and regular 
at $p_{0,2}$. 

We saw at (5) that $W\vee W_1$ is regular and semi-stable at $\tp_0$
(which is rational because $\sigma(\tp_0)=\tp_0$).
Looking at the equation~\eqref{eq:step5} of $(W_1)_k$, we see that 
$\delta(p)=0$ for all
$p\in (W_1)_k \setminus \{ p_1^*, p_1\}$. Therefore $W_1$ is
regular away from $\{ p_1^*, p_1\}$. As $\delta(p_1)<\lambda(p_1)$, 
Proposition~\ref{prop:reg}(2) implies that 
$W_1\vee W_2$ is regular and semi-stable at $\tp_1$.
As $((W_1)_k)_{\mathrm{red}}\simeq \PP^1_k$, the statement on
$W\vee W_2$ is proved. 

Finally the statement on the $\delta$'s for 
$p\in (W_2)_k \setminus \{ p_2^* \}$  follows from  
Lemma~\ref{lem:d0}(4) because $\delta(p_2^*)=2g-1$
by Lemma~\ref{lem:delta_p0_p1*}.
\end{proof}

\begin{corollary} \label{cor:exT} Let $W$ and $p_0$ be as in Proposition~\ref{prop:exT}.
  Then the minimal desingularization of $W$ at $p_0$ is same as
  that of a (not necessarily minimal) Weierstrass model $\cE$
  of an elliptic curve   at a singular point of $\cE_k$: the
  exceptional locus consists 
  in a chain of $\PP^1_k$'s of multiplicity $1$ connected to a curve
  corresponding to some Kodaira symbol $\cK$.
\end{corollary}

\begin{proof} If we are in the case (5.c.iv), we start again with
  $(W_2, p_2)$. By Corollary~\ref{cor:dl}, the process terminates
  after finitely many iterations.  
\end{proof}

\begin{corollary} \label{cor:g2desing} Keep the hypothesis of
  Proposition~\ref{prop:exT}. Suppose further that $g=2$ and
  $W$ is the unique minimal Weierstrass model, or $W, W'$ are the 
  extremal minimal Weierstrass models of $C$ (Definition~\ref{extremal-MWM})
  and $p_0\notin W \wedge W'$. Then $(W, p_0)$ has type some Kodaira symbol
  $\cK$. 
\end{corollary}

\begin{proof} Indeed, the case (5.c.iv) can not happen by the uniqueness
 of $W$.   
\end{proof}

\begin{remark} \label{rmk:e1d3} For completeness, let us consider the case of  a
  Weierstrass model $W_1$ with non-reduced closed fiber $(W_1)_k$ 
  and a point $p_1\in (W_1)_k(k)$ singular in $W_1$ such that $\delta(p_1)\le 3$.
    If $\delta(p_1)=1$, the singularity is solved by Lemma~\ref{lem:l2d1}. 
  If $\delta(p_1)=2$ (resp. $\delta(p_1)=3$), then the pair $(W_1, p_1)$
  is exactly same as in Proposition~\ref{prop:exT}(5.b) (resp. (5.c)),
  thus the desingularization of $p_1$ follows from Proposition~\ref{prop:exT}. 
  \end{remark} 

 \subsection{Minimal regular models}\label{subsect:mrm}

Suppose $g$ is even and $C$ has more than one minimal Weierstrass model.
Let $(W_i)_{0\le i\le n}$ the minimality chain of Weierstrass models
of $C$ (Definition~\ref{extremal-MWM}).  
Consider the morphism $W_0\vee W_n\to Z_0\vee Z_n$ where $Z_i=W_i/\qi$.
As $p_0\in (W_0)_k$ is a cusp (Proposition~\ref{prop:W0}), there is only one point $p_{0,n}\in W_0\vee W_n$
lying over the intersection point of $(Z_0\vee Z_n)_k$ and it is rational over
$k$. This is also the intersection point of the strict transforms of
$(W_0)_k, (W_n)_k$.

Recall that a model $Y$ of $C$ is \emph{semi-stable} at a point $p\in Y_k$
if $p$ is an ordinary double point of $Y_k$. The \emph{thickness} $e_p$ of $Y$
at $p$ is then defined as \cite{LB}, Definition~10.3.23. The minimal
desingularization of $Y$ at $p$ is given by a chain of $e_p-1$ projective
lines over $k(p)$ ({\it op. cit.}, Corollary~10.3.25). 

\begin{theorem}\label{regular-even} Let $C$ be a hyperelliptic curve over $K$ of even genus $g\ge 2$. Let $(W_i)_{0\le i\le n}$ be the minimality chain of 
  Weierstrass models. 
  Let $\cC$ be the minimal regular model of $C$ over $\cO_K$.
  \begin{enumerate}[\rm (1)] 
  \item We have a morphism of models $f : \cC\to W_0\vee W_n$.
  \item The model $W_0\vee W_n$ is semi-stable at $p_{0,n}$ of thickness $n/2$. 
  \item The divisor $\cC_k$ has normal crossings in an open neighborhood of
    $f^{-1}(p_{0, n})$. More precisely, if $\Theta_{i}$ is the strict transform of
    $(W_{i})_k$ in $\cC$, then the $\Theta_{i}$, for the even $i$'s,
    form a chain as in Figure~\ref{figure:regu},
    with $\Theta_i \simeq \PP^1_k$ when $i\ne 0, n$.
  \end{enumerate} 
  \end{theorem}

\begin{proof} (1) As $(W_0)_k, (W_n)_k$ are reduced (Lemma~\ref{lem:WW'}(2)), 
this follows from Proposition~\ref{prop:mini}(2.a)-(3).  

(2) We have to desingularize $p_{0,n}\in W_0\vee W_n$.
Consider the model
\[ Y:=W_0 \vee W_1 \vee ... \vee W_n\]
of $C$      (see the discussions after Definition~\ref{extremal-MWM}). 
    Denote by $\Gamma_i\subset Y$ the strict transform of $((W_i)_k)_{\mathrm{red}}$ for
    all $0\le i\le n$. Let $\psi : Y\to W_0 \vee W_n$ be the canonical morphism. 

    Let us study the structure of $Y$ around $\psi^{-1}(p_{0,n})$.
    Let $0\le i\le n-1$.  By Theorem~\ref{chain-mwm},
    if $i$ is even, then $\varepsilon(W_i)=0$,
    $\lambda(p_i)=\delta(p_i)=2g+1$, and
    $\varepsilon(W_i(p_i))=\varepsilon(W_{i+1})=1$. It follows from
    Proposition~\ref{prop:reg}(2) and (4) that $\Gamma_i\cap \Gamma_{i+1}=
    \{ p_{i, i+1} \} $ consists in a rational point, regular in $Y$,
    and that the intersection is transverse.   
    For similar reasons the result also holds when $i$ is odd. 

    It remains to see what happen to the points of $Y$ away from the
    $p_{i, i+1}$'s. For all $0\le i \le n$, the canonical morphism
    $Z_0 \vee Z_1 \vee ... \vee Z_n\to Z_i$ is an isomorphism above
    $Z_i\setminus ((Z_i\wedge Z_{i-1}) \cap (Z_i\wedge Z_{i+1}))$
    (see also Figure~\ref{tree}) if $i\ne 0, n$, and above
    $Z_0\setminus (Z_0\wedge Z_1)$ (resp. $Z_n\setminus (Z_n\wedge Z_{n-1})$)
    if $i=0$ (resp. $n$). 
    So the similar statement holds for the canonical morphisms 
    $Y\to W_i$. By Theorem~\ref{chain-mwm}, $W_i \setminus \{ p_i^*, p_i \}$ 
    is regular for all $0<i<n$ and $\Gamma_i$ is rational, hence
    isomorphic to $\PP^1_k$ because it is smooth.  Therefore
    in an open neighborhood of $\psi^{-1}(p_{0,n})$, 
    $Y$ is regular with normal crossings divisor $Y_k$. 
    We have $\Gamma_i^2=-1$ for odd $i$'s.
    The contraction $Y\to Y'$ of these exceptional divisors gives 
    a model $Y'$ dominating $W_0\vee W_n$, regular in an open neighborhood
    of the pre-image of $p_{0,n}$. The strict transforms in $Y'$ 
    of the $(W_i)_k$'s, for even $0<  i < n$, form a chain of $\PP^1_k$
    with self-intersection $-2$. This is the minimal desingularization
    of $W_0\vee W_n$ at $p_{0,n}$. This proves (2).

(3)  Keep the notation of (2).   The possible singular points of
    $Y'$ are those of $W_0\setminus \{ p_0 \}$ and $W_n \setminus \{ p_n^*\}$.
    In the minimal desingularization $Y''$ of $Y'$, $\psi^{-1}(p_{0,n})$ is
    not affected. None of the strict transforms of $(W_0)_k, (W_n)_k$ are
    exceptional divisors by (1). So $Y''=\cC$ and the theorem is proved. 
  \end{proof}

\begin{figure}    
 \centering 
\begin{tikzpicture}[scale=1] 
  \coordinate (P3) at   (0, 2);
  \coordinate (Q3) at  (1.25, 0);
  \draw  (P3) node [below left] {$\Theta_0$} to [bend left] (Q3); 
  \draw[dashed] (0.8, 1.5) to [bend left] (-0.5, 1);
  \draw[dashed] (1, 1.2) to [bend left] (-0.5, 0.8); 
  \draw (1, 0) -- (2.5, 1.5);
  \draw (2.8, 1.5) node {$\Theta_2$} ; 
  \draw (1.5, 1.5) -- (3,0);
  \draw (2.6, 0) node {$\Theta_4$} ; 
  \draw[dashed]  (2,0.5) -- (5, 0.5);
  \draw (4,0) -- (5.5,1.5); 
  \draw (4.5,1.5) -- (6,0);
  \coordinate (P4) at  (5.75, 0);
  \coordinate (Q4) at  (7, 2); 
  \draw (Q4) node [below right] {$\Theta_n$} to [bend right] (P4);
  \draw[dashed] (6,1.2) to [bend right] (7.5, 0.75);
  \draw[dashed] (6.2, 1.5) to [bend right] (7.6, 1);
  \draw (5.3, 1.8) node {$\Theta_{n-4}$};
  \draw (6.5, 0) node {$\Theta_{n-2}$}; 
\end{tikzpicture} 
\caption{The closed fiber of the minimal regular model $\cC$ of $C$.}
\label{figure:regu}
\end{figure}

\begin{remark}\label{rmk:thick}  Keep the hypothesis of Theorem~\ref{regular-even}.
  \begin{enumerate}[\rm (1)] 
  \item If $0\le i<j \le n$ are even, then $W_i\vee W_j$ is semi-stable with
thickness $(j-i)/2$ at the point $p_{i,j}$ defined similarly to $p_{0,n}$.
Indeed, the proof of Part (2) of the theorem does not need $W_0, W_n$ to be extremal.
\item To compare Figures~\ref{figure:mwm} and \ref{figure:regu}:
  with the above notation, the canonical morphism $\Theta_i\to \Gamma_i$ 
  is an isomorphism for even $i\ne 0, n$. The canonical morphism
  $\Gamma_0\to (W_0)_k$ is the normalization at $p_0$, but
  $\Theta_0\to\Gamma_0$ needs not to be an isomorphism.
  Similarly for $\Theta_n\to \Gamma_n\to (W_n)_k$.     
  \end{enumerate}
 
\end{remark}

\begin{corollary} \label{cor:bc} Keep the hypothesis of
  Theorem~\ref{regular-even}. Let $\cO_{K'}$ be a discrete valuation ring
  dominating $\cO_K$, with perfect residue field $k'$ and field of fractions
  $K'$. Let $e_{K'/K}=\nu_{K'}(\pi)$ be the ramification index. 
  Then $C_{K'}$ has at least $1+ (ne_{K'/K})/2$ minimal Weierstrass models
  over $\cO_{K'}$.
\end{corollary}

\begin{proof} By Proposition~\ref{prop:more_than_one}, $(W_0)_{\cO_{K'}}$
  and $(W_n)_{\cO_{K'}}$ are minimal Weierstrass models. As
  $W_0\vee W_n$ is semi-stable at $p_{0,n}$ of thickness $n/2$, 
  \[ (W_0)_{\cO_{K'}}\vee  (W_n)_{\cO_{K'}}=(W_0\vee W_n)_{\cO_{K'}}\]
  is semi-stable of thickness $e_{K'/K}n/2$.
  So by Remark~\ref{rmk:thick}(1), there are $1+(e_{K'/K}n/2)$ minimal
  Weierstrass models in the chain connecting  
$(W_0)_{\cO_{K'}}$ to $(W_n)_{\cO_{K'}}$. 
\end{proof}

\noindent {\bf Dualizing sheaves.} Keep the hypothesis of Theorem~\ref{regular-even}.
Let $Y:=W_0\vee W_n$. As $W_0, W_n$ are local complete intersection and
$Y$ is semi-stable at $p_{0,n}$ (Theorem~\ref{regular-even}),
$Y$ is also a local complete intersection.
Recall that if $W_0$ is given by a Weierstrass equation $y^2+Q(x)y=P(x)$, 
then the global sections of its dualizing sheaf $\omega_{W_0/\cO_K}$ has a
basis : 
 \[
   H^0(W_0, \omega_{W_0/\cO_K}) = \bigoplus_{0\le d\le g-1} \frac{x^ddx}{2y+Q(x)} \cO_K. 
 \] 

 \begin{proposition} \label{prop:omega} Keep the above notation. Then
   we have the following relations of submodules of
$H^0(C, \Omega^{1}_{C/K})$ and of $\det H^0(C, \Omega^{1}_{C/K})$: 
\begin{equation}
  \label{eq:dualizing_inc}
  H^0(\cC, \omega_{\cC/\cO_K})\subseteq  H^0(Y, \omega_{Y/\cO_K})
 \subseteq H^0(W_0, \omega_{W_0/\cO_K}),   
\end{equation}
  \[
\det H^0(Y, \omega_{Y/\cO_K})=\pi^{\frac{ng^2}{8}} \det H^0(W_0, \omega_{W_0/\cO_K}).   
\]
\end{proposition}

 \begin{proof} As $Y$ is normal and $\omega_{Y/\cO_K}$ is locally free, we have
   \[ H^0(Y,  \omega_{Y/\cO_K}) =
   H^0(Y \setminus \{ p_{0,n}\},  \omega_{Y/\cO_K}).\]
   As $Y_k \setminus \{ p_{0,n} \}$ is the disjoint union of  
$(W_0)_k\setminus \{ p_0 \}$ and $(W_n)_k\setminus \{ p_n^* \}$, we have 
   \[ 
\begin{matrix} 
  H^0(Y,  \omega_{Y/\cO_K}) & =& 
  H^0(W_0\setminus \{ p_0\}, \omega_{W_0/\cO_K})  
\cap H^0(W_n\setminus \{ p_n^*\}, \omega_{W_n/\cO_K})  \\ 
 &=&  H^0(W_0, \omega_{W_0/\cO_K}) \cap H^0(W_n, \omega_{W_n/\cO_K}).\hfill 
\end{matrix} 
 \] 
 This implies the second inclusion of \eqref{eq:dualizing_inc}.
 The first one is proved in the same way.
 
Now let $y^2+Q(x)y=P(x)$ be an equation of $W_0$ as given by
  Corollary~\ref{cor:equa_term}.  So  an equation of
  $W_n$ is
  \[
y_n^2 +Q_n(x_n)y_n=P_n(x_n), \quad  x_n=x/\pi^n, y_n=y/\pi^{(g+1)n/2}
  \] 
and $Q_n(x_n)=\pi^{-(g+1)n/2}Q(x)$. A basis
of $H^0(W_n, \omega_{W_n/\cO_K})$ is made of the 
\[ \eta_d:=\pi^{(g-2d-1)n/2} \frac{x^ddx}{2y+Q(x)}, \quad 0\le d\le g-1.\]   
So a basis of $H^0(W_0, \omega_{W_0/\cO_K}) \cap H^0(W_n, \omega_{W_n/\cO_K})$ is
\[ \{ \eta_d \ | \ 0\le d\le -1+ g/2\} \cup 
  \{ x^ddx/(2y+Q(x)) \ | \ g/2 \le d \le g-1\}. \]
Then the desired equality follows. 
\end{proof}

\subsection{Canonical models} \label{subsect:cm}
  First we recall some facts about the canonical models. 
  Let $C$ be a smooth projective geometrically connected curve of genus $g\ge 2$
  over $K$. Let $\cC$ be the minimal regular model of $C$ over $\cO_K$. 
  
\begin{definition} \label{def:can}  
The \emph{canonical model} $\cC^{\can}$ of $C$ over $\cO_K$ is the
normal model of $C$ obtained by contracting all the $(-2)$-curves in $\cC_k$
({\it i.e.},  the irreducible components $\Gamma$ of $\cC_k$ such that
$\deg(\omega_{\cC/\cO_K}|_{\Gamma})=0$).  See \cite{LB}, Definition 9.4.21. The canonical model is unique up to unique isomorphism. 
\end{definition}

The canonical model is locally complete intersection,
hence has an invertible dualizing sheaf $\omega_{\cC^{\can}/\cO_K}$. 
The adjunction formula 
\begin{equation} \label{eq:adj_can} 
  2g(C)-2= \sum_{\Gamma} d_\Gamma \deg(\omega_{\cC^{\can}/\cO_K}|_{\Gamma})
\end{equation}
where the sum runs through the irreducible components of $\cC^{\can}_k$, 
and $d_\Gamma$ is the multiplicity of $\Gamma$ in $\cC^{\can}_k$, see the
proof of \cite{LB}, Proposition 9.4.24. As by construction 
$\deg(\omega_{\cC^{\can}/\cO_K}|_{\Gamma})>0$, we see
in particular that $\cC^{\can}_k$ has at most $2g(C)-2$ irreducible 
components. 

The model $\cC^{\can}$ is in general singular. Let
$f : \cC\to \cC^{\can}$ is the contraction map. Then for any $q\in \Z$ we have
\begin{equation} \label{eq:can_o}
 f_*(\omega_{\cC/\cO_K}^{\otimes q})=\omega_{\cC^{\can}/\cO_K}^{\otimes q},
\quad
f^*(\omega_{\cC^{\can}/\cO_K}^{\otimes q})=\omega_{\cC/\cO_K}^{\otimes q}
\end{equation} 
(see \cite{LB}, Corollary 9.4.18).  In particular, 
$f_*\omega_{\cC/\cO_K}=\omega_{\cC^{\can}/\cO_K}$. This implies that the singular
points of $\cC^{\can}$ are \emph{rational singularities} ({\it i.e.}, $R^1f_*\cO_{\cC}=0$)
by \cite{Lip}, Theorem 27.1. See also \cite{Art}, Corollary 3.4(1). 
In {\it op. cit.}, the schemes are assumed to be excellent. 
Here we can either suppose $\cO_K$ is excellent, {\it e. g.} $\chara(K)=0$, or
use the existence of resolutions of singularities (because $C$ is smooth).

\begin{proposition} \label{prop:volume_can} Let $C$ be a smooth
projective geometrically connected curve of positive genus over $K$.
Let $\cA$ be the N\'eron model of the Jacobian $\Jac(C)$ of $C$ and
let $\omega_{\cA/\cO_K}=e^*\Omega^1_{\cA/\cO_K}$ where $e\in \cA(\cO_K)$ is the
neutral section.  Suppose that the gcd of the
multiplicities in $\cC_k$ of the irreducible components of $\cC_k$
is $1$ ({\it e.g.}, $g=2$). Then we have a canonical isomorphism 
\[
  \det \omega_{\cA/\cO_K} \simeq \det H^0(\cC^{\can}, \omega_{\cC^{\can}/\cO_K}).
\]  
\end{proposition}

\begin{proof} By \cite{BLR}, Theorem 9.5/4, the canonical morphism  
\begin{equation}
  \label{eq:Jac_Pic}
   \Pic^0_{\cC/\cO_K}\to \mathcal A^0.   
\end{equation} 
is an isomorphism. This induces canonical isomorphisms 
\[ \omega_{\cA/\cO_K} \simeq H^1(\cC, \cO_{\cC/\cO_K})^{\vee} 
\simeq H^0(\cC, \omega_{\cC/\cO_K}) = H^0(\cC^{\can}, \omega_{\cC^{\can}/\cO_K})\]  
(\cite{LLR}, Proposition 1.1.b, Proposition 1.3.b and Equality~\eqref{eq:can_o}
with $q=1$). 
They extend the canonical isomorphisms
$H^0(A, \Omega^1_{A/K}) \simeq H^1(C, \cO_C)^{\vee}
\simeq H^0(C, \Omega^1_{C/K})$. 
Taking $\det$ gives the desired isomorphism.

Note that in general, if $d$ is the gcd of the multiplicities as defined
in the proposition, then $d \mid g-1$ (apply \cite{LB}, Exercise 9.1.11 to
$D=\cC_k/d$). In particular, $d=1$ if $g=2$.  
\end{proof} 

Now we go back to the hypothesis and notation of \S~\ref{subsect:mrm}. We
want to relate the minimal Weierstrass models of $C$ to $\cC^{\can}$. 
We know that the minimal Weierstrass models of $C$ 
are dominated by $\cC$ (Remark~\ref{rmk:mini_mini}). 
The question is then whether they are also dominated by $\cC^{\can}$.  

\begin{proposition} \label{prop:positive} Suppose that $g$ is even and that 
$C$  has more than one minimal Weierstrass model over $\cO_K$.
Then the extremal minimal Weierstrass models $W_0$, $W_n$
are dominated by $\cC^{\can}$,  while the other $W_i$'s (for even $2\le i\le n-2$)
are contracted to a single point in $\cC^{\can}$.    
\end{proposition}

\begin{proof}   
  The statement on the $W_i$, $2\le i\le n-2$, follows from 
Theorem~\ref{regular-even}. 

Let $\cC$ be the minimal regular model of $C$ over $\cO_K$.
Denote by $\Theta_i$ the strict transform of $(W_i)_k$ for even $i$'s
as in Theorem~\ref{regular-even}.
Suppose that $W_n$ is not dominated by $\cC^{\can}$. This 
  means that $\Theta_n$ is a $(-2)$-curve: it is isomorphic to $\PP^1_k$, and
  has self-intersection number equal to $-2$. We will construct  
  Weierstrass models $W_{n+1}, W_{n+2}$ such that
  $W_{n-2}, W_{n-1}, W_n, W_{n+1}, W_{n+2}$ is a chain
  and $W_{n+2}$ is minimal,  which will contradict   the hypothesis $W_n$ extremal. 
  \smallskip
  
\noindent {\it Construction $W_{n+1}, W_{n+2}$.}  By Theorem~\ref{regular-even}, $\Theta_n$ intersects transversely   $\Theta_{n-2}$ at a rational point $p_{n-2, n}$.
  Thus  $\Theta_n$ must intersect
  transversely, at a rational point $\tp_n\ne p_{n-2, n}$, an irreducible 
  component $\Theta_{n+2}\ne \Theta_{n-2}$ of $\cC_k$, of multiplicity $1$ in $\cC$.
  It is clear that $\tp_n$ is a fixed point of $\sigma$.
  Let $Z_i=W_i/\qi$.  
  Let $f : \cC_1\to \cC$ be the blowing-up along $\tp_n$. 
  As $\cC$ is regular
  and semi-stable at $\tp_n$,  the image of $\tp_n$ in  $\cC/\qi$ is 
  a double point of  thickness $2$ (\cite{LB}, Proposition 10.3.48(c)).  
  An easy local study shows that $\cC_1/\qi \to \cC/\qi$ is the
  blowing-up of $\cC/\qi$ along the image of $\tp_n$. The exceptional divisor
  for this blowing-up induces a smooth model $Z_{n+1}$ of $C/\qi$ and 
  $Z_n, Z_{n+1}, Z_{n+2}$ is a chain. Therefore, if $W_{n+1}$
  denotes the normalization of $Z_{n+1}$ in $K(C)$, the strict transform of 
  $((W_{n+1})_k)_{\mathrm{red}}$ in $\cC_1$ is the exceptional divisor $\Theta_{n+1}$ of
  $\cC_1\to \cC$. Let $p_n\in (W_n)_k(k)$ the image of $\tp_n$. Then
  $p_n\ne p_n^*$ and we have a chain $W_n, W_{n+1}, W_{n+2}$ of Weierstrass models
  with $W_{n+1}=W_n(p_n) \ne W_{n-1}$ and $W_{n+2}=W_{n+1}(p_{n+1})$ for
  some $p_{n+1}\in (W_{n+1})_k(k) \setminus \{ p_{n+1}^* \}$. In a neighborhood
  of $f^{-1}(\tp_n)$, $\cC_1$ coincides with $W_n\vee W_{n+1} \vee W_{n+2}$. 
\smallskip

\noindent{\it The Weierstrass model $W_{n+2}$ is minimal}.  
To this end, we will show that $\lambda(p_n)=g+1$ and $\lambda(p_{n+1})=g+2$.
By the above description, the morphism $\cC\to W_n$ is an isomorphism away
from $p_n^*, p_n$ and the latter are cusps because $\Theta_n\to (W_n)_k$
is the normalization map and is bijective.  
As $\Theta_n\simeq \PP^1_k$, the genus formula gives  
  \[
g=p_a(W_k)=[\delta(p_n)/2]+[\delta(p_n^*)/2]=[\delta(p_n)/2] + g/2 
  \] 
and $\delta(p_n)$ is odd (\cite{LTR}, Lemme 6(b)), thus $\delta(p_n)=g+1$. 
As $\varepsilon(W_n(p_n))=\varepsilon(W_{n+1})=1$, $\lambda(p_n)$ is odd,
and by Proposition~\ref{prop:reg}(2), we have $[\lambda(p_n)/2]=g$,
hence $\lambda(p_n)=g+1$. 

Next we show that $\lambda(p_{n+1})=g+2$.
We have $\delta(p_{n+1}^*)=g+1$ by Lemma~\ref{lem:delta_p0_p1*}.
As $W_{n+1}\setminus \{ p_{n+1}^*, p_{n+1} \}$ is smooth,
Lemma~\ref{lem:d0}(2)-(3) imply that 
$\delta(p_{n+1})=g+1$.
Let $r^*=[\lambda(p_{n+2}^*)/2]$. As $\varepsilon(W_{n+1})=1$, we have
$\lambda(p_{n+2}^*)=2r^*+1$. Proposition~\ref{prop:reg}(2) implies that
$\delta(p_{n+2}^*)=2r^*+1$. By Lemma~\ref{lem:delta_p0_p1*} applies to
$p_{n+2}^* \in W_{n+2}$, we get $g+1=\delta(p_{n+1})=2g+2-(2r^*+1)$,
hence $r^*=g/2$. Finally Equality~\eqref{eq:lambdap1*} applied
to $p_{n+2}^*\in W_{n+2}$ implies that
$\lambda(p_{n+1})=2g+2-2[r^*/2]=g+2$. 
By Formula~\eqref{eq:disc-p0}, 
\[ \nu(\Delta_{W_{n+2}})=\nu(\Delta_{W_{n+1}})- 2(2g+1)= \nu(\Delta_{W_n}).  
\] 
Therefore $W_{n+2}$ is minimal.  
\end{proof}

\begin{remark} When $C$ has a unique minimal Weierstrass model $W$, the latter
  is not dominated by $\cC^{\can}$ in general, even when $g=2$. Indeed, 
  it can happen that $\cC^{\can}_k=2\Gamma$ with $p_a(\Gamma)=1$
  (see  {\it e.g.}, \cite{NU}, page 155, type [II]). Then
  $\cC^{\can}$ can not dominate $W$
  because otherwise $\cC^{\can}$ would be  isomorphic 
  to $W$. But this would imply that $\varepsilon(W)=1$ and $(W_k)_{\mathrm{red}}\simeq
  \PP^1_k$. 
\end{remark}

\subsection{Rigid analytic structure}  \label{subsect:rig}
We suppose $K$ complete in this subsection. Then 
any separated algebraic variety over $K$ can be
endowed with the structure of a rigid analytic 
space. We also use Raynaud's formal schemes point of view. 

  Let $X$ be a reduced connected affine curve over (a perfect field) $k$. The
  \emph{genus of $X$} is the arithmetic genus of the compactification of
  $X$ with smooth boundary ({\it i.e.} the points at infinity are smooth).
  Let us recall: 

  \begin{theorem}[\cite{FM}, \cite{vdP}]\label{thm:FM} Let $R$ be a smooth affinoid 
    curve over $K$ with canonical reduction $X$ of genus $g$. Let
    $\Gamma$ be the compactification of $X$ with smooth boundary. 
    Then $R$ is an open subspace of a smooth projective curve $D$
    of genus $g$, and $X$ is an affine open subset of some reduction of $D$. 
  \end{theorem}

  \begin{proof} 
    See \cite{FM} \S 2.6, Th\'eor\`eme 6. The statement was proved 
    in \cite{vdP}, Theorem 1.1 for complete algebraically closed ultrametric
    valued fields. 
  \end{proof}

  \begin{remark} \label{rmk:lg}
    Let $C$ be a smooth projective curve over $K$. Let $W$ be a
  model of $C$ with reduced $W_k$ and let $p_0\in W_k(k)$
  be a singular point of $W_k$. 
  Let $\Gamma$ be the normalization of $W_k$ at $p_0$. This is
  also the compactification of $W_k\setminus \{ p_0 \}$ with smooth boundary. 
Let $\whW$ be the formal completion of $W$ along its closed fiber. 
Then the generic fiber $(\whW\setminus \{ p_0 \})_K$ of
$\whW\setminus \{ p_0 \}$ is a smooth affinoid curve over $K$, with
canonical reduction $W_k\setminus \{ p_0 \}$. The above theorem
says more precisely that
$\whW \setminus \{ p_0 \}$ can be embedded into the formal completion
$\widehat{\mathcal D}$ of a model $\mathcal D$ of a smooth  
projective curve $D$ over $K$ and such that 
$W_k\setminus \{ p_0 \} \to \mathcal D_k$ extends to an isomorphism
$\Gamma\simeq \caD_k$, as represented in the following commutative
diagram: 
\[  \label{eq:FM} 
\begin{tikzcd}
\whW \setminus \{ p_0 \} \arrow[r, hook]  & \widehat{\mathcal D} \\ 
W_k \setminus \{ p_0 \} \ar[u, hook] \ar[r, hook] & \Gamma\simeq \mathcal D_k. \ar[u, hook] 
\end{tikzcd} 
\] 
 (The horizontal arrows are open immersions and the vertical ones are
 closed immersions). 
In particular $g(D)=p_a(\Gamma)<p_a(W_k)=g(C)$.  
For any $p\in W_k \setminus \{ p_0 \}\subset \Gamma$,
we have $\widehat{\cO}_{W, p} =\widehat{\cO}_{\mathcal D, p}$. 
This implies that the resolution of the singularity of $W$ at $p$ is the same
as that of $\mathcal D$. The advantage is that the genus drops with $D$.
\end{remark}

Now come back to the situation with $C$ hyperelliptic of genus $g\ge 1$. 
Let $W$ be a Weierstrass model with integral closed fiber and fix a
singular point $p_0\in W_k(k)$. Then $\sigma(p_0)=p_0$. 
We have $p_a(\Gamma)=g-[\delta(p_0)/2]$. 
Let $Z=W/\qi$ and let $q_0\in Z_k(k)$ be the image of $p_0$,
 we would like to extend the $2$-cyclic cover
 $\whW \setminus \{ p_0 \} \to \whZ \setminus \{ q_0 \}$ into a
 $2$-cyclic cover of $\whZ$: 
 \begin{equation}
 \label{eq:2c} 
\begin{tikzcd}
\whW \setminus \{ p_0 \} \ar[r, hook]\ar[d] & \widehat{\mathcal D} \ar[d]\\ 
\whZ \setminus \{ q_0 \} \ar[r, hook]          & \whZ
\end{tikzcd} 
 \end{equation} 
 for a suitable Weierstrass model $\mathcal D$ of a suitable hyperelliptic
 curve $D$ over $K$.   
 Note that by GAGA for formal schemes (\cite{EGA_3}, Corollaire 5.1.6),
 any finite cover of $\whZ$ 
 is the formal completion of a projective curve over $\cO_K$
 along the closed fiber.
 
 \begin{definition} 
 In the above diagram, we call $\mathcal D$ a
 \emph{hyperelliptic compactification of $\whW \setminus \{ p_0\}$}
 if $W_k\setminus \{ p_0 \} \to \caD_k$ induces an isomorphism 
$\Gamma\simeq \caD_k$ (equivalently $g(D)=p_a(\Gamma)$). 
  \end{definition}

 \begin{proposition}\label{prop:tame} Suppose that $\chara(k)\ne 2$.
   Then $\whW \setminus \{ p_0\}$ admits a hyperelliptic 
   compactification.  
 \end{proposition}

 \begin{proof} 
   Write an equation $y^2=F(x)$ of $W$ with $p_0$ at infinity. 
  Then 
  \[ [(\deg \bar{F}-1)/2]=g-[\delta(p_0)/2]=p_a(\Gamma).\]  
  By Hensel's lemma, $F(x)=F_1(x)(1+\pi F_2(x))$
  with $\deg F_1(x)=\deg \bar{F}(x)$ and $F_2(x)\in \cO_K[x]$. There exists
  $u(x)\in \cO_K\{ x \}$ (see the notation \eqref{eq:rest})  such that $u(x)^2=1+\pi F_2(x)$. Therefore
  $\whW\setminus \{ p_0 \}$ is isomorphic to the formal completion of the
  affine scheme $z^2=F_1(x)$. This equation defines $D$ and the
  desired Weierstrass model $\mathcal D$. 
\end{proof}

Proposition~\ref{prop:tame} does not hold in general 
when the residue characteristic is equal to $2$ (see Example~\ref{ex:no_comp}
below). In the rest of this 
subsection we try to address the problem of compactification in residue
characteristic $2$. Let $K^{\text{alg}}$ be an algebraic closure of $K$.
We denote by $| \ |$ an absolute value on $K^{\text{alg}}$ extending 
an absolute value on $K$ associated to $\nu_K$. 

\begin{example} \label{ex:no_comp} Suppose that there exists $c\in K$ with
  $1>|c|>|2|>0$. Consider the curve $C$ over $K$ having a Weierstrass model
  $W$ given by the equation
  \[ y^2 +c y = x^{2g+1}. \]
  Let $p_0$ be the zero of $x$ in $W_k$. Then
  $W_k\setminus \{ p_0\} \simeq \mathbb A^1_k$ and $\Gamma\simeq \PP^1_k$.
  The ramification points of $C\to \PP^1_K$ are above the pole of $x$ and
  the zeros $4x^{2g+1}+c^2$ in $K^{\text{alg}}$.  
They all have $| \ |>1$. Hence for any 
$\caD$ as in Diagram~\eqref{eq:2c}, $D\to Z_K$ is 
ramified in at least $2g+2$ points and 
$g(D)\ge g$. So $\caD_k$ is not isomorphic to $\Gamma$. 
\end{example}
 
In the above example, there are too many ramification points of $C\to Z_K$
specializing to $W_k\setminus \{ p_0\}$.
This is in fact a general obstruction to the  
existence of a hyperelliptic compactification (see Lemma~\ref{lem:degx}
below). Let $W$ be defined by an equation 
\[ 
  y^2+Q(x)y=P(x), \quad \text{with } \  x(p_0)=0
\] 
and such that
$\delta(p_0)=\min \{ 2\ord_0\bar{Q}(x), \ord_0\bar{P}(x) \} \ge 2$.
We have 
\[ (\whZ \setminus \{ q_0 \})_K = \{ q \in Z_K  \ | \  |x(q)|\ge 1 \},\]
and the formal fiber 
(\cite{LB}, Definition 10.1.39) of $\whZ$ at $q_0$ is 
\[ \whZ_+(q_0)=\{ q \in Z_K \ | \ |x(q)|<1 \}.\] 
For any polynomial $F(x)\in K[x]$, denote by $\cont(F)$ the \emph{content of
$F(x)$} (defined up to multiplication by $\cO_K^*$). By convention $\cont(0)=0$. 

\begin{lemma} \label{lem:degx} Keep the above notation and assume that
  $\delta(p_0)$ is odd. 
  Let $R_{C/Z_K}$ be the ramification divisor of  $C\to Z_K$.  
  \begin{enumerate}[\rm (1)]  
  \item Let $\caD$ be as in Diagram~\eqref{eq:2c}. 
We have 
  \[
2g(D) + 1 \ge  \deg\left(R_{C/Z_K}|_{|x|\ge 1}\right):=
  \sum_{p\in C, \, |x(p)|\ge 1}  [k(p):k]\deg_p R_{C/Z_K}.  
  \]
  In particular, if there exists a hyperelliptic compactification $\caD$, then 
  \begin{equation}
    \label{eq:obstr}
2p_a(\Gamma)+1 \ge  \deg\left(R_{C/Z_K}|_{|x|\ge 1}\right).   
  \end{equation} 
\item Write $Q(x)=cQ_0(x)$ where $c=\cont(Q)$. 
Then Inequality~\eqref{eq:obstr} holds if and only if  
either $0\le |c|<|2|$ or,  $|c| \ge |2|$ and $2\ord_0\bar{Q}_0(x)\ge \delta(p_0)$.
  \end{enumerate}
\end{lemma}
 
\begin{proof} (1) As $p_0$ is a cusp, there is only one point
  in $\caD_k$ lying over $q_0$. So $\caD\to Z$ is ramified
  above $q_0$. By the purity of the ramification locus (\cite{LB},
  Exercise 8.2.15(c)), $Z$ being 
  regular, the support of $R_{D/Z_K}$ contains a point lying over
  $\whZ_+(q_0)$. Thus 
  \[ 
  \begin{matrix} 
    \deg R_{D/Z_K}  &=  &\sum_{p\in D, |x(p)|<1} [k(p): k]\deg_p R_{D/Z_K}+ 
    \deg (R_{D/Z_K}|_{|x|\ge 1}) \\
            {}    &  \ge&   1+ \deg (R_{D/Z_K}|_{|x|\ge 1}). \hfill  
   \end{matrix}
 \] 
The condition $|x(p)|\ge 1$ means that $p$ lies over a point of 
$(\whZ \setminus \{ q_0 \})_K$, hence
$p \in (\whW \setminus \{ p_0 \})_K\subset \widehat{\mathcal D}_K$.
As the ramification divisor depends only on the extension of the
formal completions $\widehat{\cO}_{Z, q}\to \widehat{\cO}_{W, p}=
\widehat{\cO}_{\mathcal D, p}$, we have
$\deg_p R_{C/Z_K}=\deg_p R_{D/Z_K}$ if $|x(p)|\ge 1$. Then Part (1)
follows from Riemann-Hurwitz formula $2g(D)+2=\deg R_{D/Z_K}$. 

(2) As $p_a(\Gamma)=g-[\delta(p_0)/2]$,  and 
    \[
    \deg \left(R_{C/Z_K}|_{|x(p)|\ge 1}\right)=2g+2-
    \deg\left(R_{C/Z_K}|_{|x(p)|< 1}\right), 
  \]
  we have 
  \[ \eqref{eq:obstr} \iff \deg\left(R_{C/Z_K}|_{|x(p)|< 1}\right) \ge 2[\delta(p_0)/2]+1=\delta(p_0).\]         
    The divisor $R_{C/Z_K}$ is the pull-back by $C\to Z_K$ of the zero divisor
    of $4P(x)+Q^2(x)$. So
    \[
\deg\left(R_{C/Z_K}|_{|x|<1}\right) 
  = \sum_{q\in Z_K, \ |x(q)|<1} [k(q):k]\ord_{x(q)}(4P(x)+Q^2(x)). 
\]
For any $F(x)\in \cO_K[x]\setminus \pi \cO_K[x]$, Hensel's lemma gives
a decomposition
  \begin{equation}
    \label{eq:Hensel}
F(x)=F_0(x)F_1(x)    
  \end{equation}
  in $\cO_K[x]$ with $F_0(x)$ monic, $\bar{F}_0(x)=x^\delta$
  and $F_1(0)\in \cO_K^*$.  
  This implies that $\sum_{|t|<1} \ord_{x=t} F(x)=\ord_0 \bar{F}(x)$. 

  Let $F=(4P+Q^2)/\cont(4P+Q^2)$. Then Inequality~\eqref{eq:obstr} is equivalent to 
  $\ord_0 \bar{F}(x)\ge \delta(p_0)$. Noting that when $|c|<1$, 
then  $\ord_0\bar{P}(x)=\delta(p_0)$ and the latter is odd, 
  we find that
\[ 
  \ord_0 \bar{F}(x)=
  \begin{cases}
    \delta(p_0) &  \text{ if }  |c|<|2| \\
    \min\{ 2\ord_0 \bar{Q}_0(x), \delta(p_0) \} & \text{ if } |c|=|2| \\
    2\ord_0 \bar{Q}_0(x) & \text{ if } |2|< |c|\le 1. \\
  \end{cases}
  \]
This proves Part (2).   
\end{proof} 

\begin{proposition}\label{prop:extend_h} Let $C$ be a hyperelliptic curve
  over a complete discrete valuation field $K$. Let $W$ be a Weierstrass
  model of $C$ with integral closed fiber and let $p_0\in W_k(k)$ be a
  singular point with odd $\delta(p_0)$. Then 
$\whW\setminus \{ p_0 \}$ has a hyperelliptic compactification 
if and only if Inequality~\eqref{eq:obstr} is satisfied. 
\end{proposition}

The proof will be done in several steps. First, Lemma~\ref{lem:degx}(1)  
implies the ``only if'' part.
To prove the ``if'' part, we will use the formal 
patching technique. Let 
 \begin{equation}
   \label{eq:rest}
   \cO_K\{ x \}= \left\{ \sum_{i\ge 0} a_ix^i \ | \ a_i\in \cO_K, \
     \lim_{i\to +\infty} |a_i| = 0 \right\}\subset \cO_K[[x]] 
     \end{equation}
 the ring of restricted power series. Then
 $\whZ \setminus \{ x = \infty \}=\Spf \cO_K\{ x \}$. 
Consider the $\cO_K$-algebras
$A_0=\cO_K\{ x \}$, $A_1=\cO_K\{ x, 1/x \}$, $A_2=\cO_K[[x]]$, 
and
\[
A_3=\cO_K[[x]]\{ 1/x \}:= \left\{  \sum_{i\in \Z } a_ix^i \ \bigg| \ a_i \in \cO_K, \
  \lim_{i\to -\infty} |a_i| =0 \right\}. 
\] 
Geometrically the first three algebras correspond respectively to the closed
unit disc $|x|\le 1$, the annulus $|x|=1$ and the open unit disc $|x|<1$.
One could think the first space $\Spf A_0$ as the glueing of the last two spaces
``along'' $\Spf A_3$. The formal patching technique says that there is an
equivalence of categories 
\[
  \underline{M}(A_0) \simeq \underline{M}(A_1) \times_{\underline{M}(A_3)}
  \underline{M}(A_2) 
\] 
where $\underline{M}(A)$ denotes the category of coherent $A$-modules. 
Moreover, if $G$ is a finite group, then the same result holds for
the categories of $G$-covers (see references in \cite{Pr}, Theorem 3.4 and
Notation 1.2 for the definition of $G$-covers).  

\begin{remark}\label{rmk:uniform_A3} The algebra $A_3$ is a complete discrete valuation ring dominating
  $\cO_K$, with uniformizer $\pi$ and residue field $k((\bar{x}))$. 
An element $t\in A_3$ satisfies $A_3=
  \cO_K[[t]]\{ 1/t \}$ if and only if $\bar{t}$ in the residue field of $A_3$
  is a uniformizing element of $k[[\bar{x}]]$.  
\end{remark}

\begin{lemma} \label{lem:extend_f}
  Keep the notation of Lemma~\ref{lem:degx} and suppose $\delta(p_0)$ is odd.
  Let  
  \[ B_3=A_3[y]/(y^2+Q(x)y-P(x)). \]
Then there exists $u\in \cO_K^*$, $t\in A_3$ such that $A_3=\cO_K[[t]]\{ 1/t \}$ 
  and
  \begin{enumerate}[\rm (1)] 
  \item   
    \[ B_3 \simeq A_3[z]/(z^2+ct^\ell z-ut) \]
  if $1\ge |c|\ge |2|$ and if $2\ord_0\bar{Q}_0(x)-\delta(p_0)=2\ell-1$ 
  with $\ell \ge 1$;  
\item 
  \[ B_3 \simeq A_3[z]/(z^2-ut) \]
if $|c|<|2|$. 
  \end{enumerate}
\end{lemma}

\begin{proof} We follow the constructions of \cite{Sa}, Proposition 3.3.1
(which does not apply directly here), note that we do not assume $\chara(K)=0$. 

  (1)   
  Let $\delta=\delta(p_0)$ and $q=\ord_{0} \bar{Q}_0(x)$.
  We have a decomposition $P(x)=P_0(x)P_1(x)$ analogue to \eqref{eq:Hensel}. 
  Then
  \[ P_0(x)=u_1x^{\delta}f_1, \quad 
  u_1\in \cO_K^*, f_1\in (1 +\pi x^{-1} \cO_K[1/x])
  \]
and
\[
P_1(x)=u_2f_2, \quad u_2 \in \cO_K^*, f_2\in 1+x\cO_K[x].  
\]
Using a similar decomposition for $Q_0(x)$ we get
\[
  \frac{P(x)}{Q_0(x)^2} = u x^{\delta-2q}f
\]
with
\[ 
  u\in \cO_K^*, \ f \in (1+\pi x^{-1}\cO_K\{ 1/x\})(1+x\cO_K[[x]]).   
\]
As $\delta-2q$ is odd, there exists
$h\in (1+\pi x^{-1}\cO_K\{ 1/x \})(1+x\cO_K[[x]])$ such that
$h^{\delta-2q}=f$.
Let $t=xh$. Then $A_3=\cO_K[[t]]\{ 1/t\}$ and we have the relation
\[
y^2+cQ_0(x)y=uQ_0(x)^2t^{-2\ell+1}. 
 \]  
 As $t, Q_0(x)\in A_3^*$, we have
 \[ B_3=A_3[z]/(z^2+ct^{\ell} z-ut), \quad z=Q_0(x)^{-1}t^{\ell}y. \] 

 (2) Write $c=2c_1$ with $c_1\in \pi\cO_K$. Then similarly to the above
 we have 
 \[ (y+c_1Q_0(x))^2=P(x)(1+c_1^2Q_0(x)^2P(x)^{-1})=ux^{\delta}f
\]
with $u\in \cO_K^*$ and $f\in 1 + x\cO_K[[x]]+\pi \cO_K\{ 1/x \} 1/x$.
Taking $z=(y+c_1Q_0(x))x^{-[\delta/2]}$ and $t=xf$, we get $z^2=ut$. 
\end{proof}

\noindent {\bf End of the proof of Proposition~\ref{prop:extend_h}} 
  Now we are ready to prove the ``if'' part of
Proposition~\ref{prop:extend_h}. Let $G=\qi$. Let 
$B_1=\cO_K\{x, 1/x \}[y]/(y^2+Q(x)y-P(x))$. Let 
\[ B_3=B_1\otimes_{A_1} A_3=A_3[y]/(y^2+Q(x)y-P(x)). 
\]
Let $B_2=A_2[z]/(z^2+ct^\ell z-ut)$ or $A_2[z]/(z^2-ut)$
with the notation of Lemma~\ref{lem:extend_f}. Then $B_2$ is a
$G$-cover of $A_2$ extending the $G$-cover $B_3$ of $A_3$. Moreover
$B_2$ is smooth over $\cO_K$. The formal patching 
theorem we recalled before Remark~\ref{rmk:uniform_A3}
then provides us with a $G$-cover $B$ of $A$. Glueing $\Spf B$ with
$\whW \setminus \{ p_0 \}$ gives the desired hyperelliptic
compactification $\widehat{\mathcal D}$.  
\qed
\medskip

Now we go back to our general settings with $K$ not necessarily complete.
Let $C$ be hyperelliptic of even genus and having more than one minimal
Weierstrass model.  Let $\cC$ be the minimal regular model of $C$ over
$\cO_K$. By Theorem~\ref{regular-even}, $\cC_k$ is the union of two curves
$X_0, X_n$ attached together by a chain of $(n/2)-1$ copies
of $\PP^1_k$, with transverse intersections.

\begin{corollary} \label{cor:gl} Let $\hat{K}$ be the completion of $K$. There
  exist  projective smooth geometrically 
connected curves $C_0,  C_n$ of genus $g/2$ over $\hat{K}$, such that 
the curves $X_0, X_n$ over $k$ are isomorphic respectively to the closed
fiber of the minimal regular model of $C_0$ and that of $C_n$ over $\cO_{\hat{K}}$. 
\end{corollary}

\begin{proof} The minimal regular model and the minimal
  Weierstrass models are compatible with the completion of $K$
  (\cite{LB}, Proposition 9.3.28 and Proposition~\ref{prop:more_than_one}),
  and their closed fibers  
  do not change by passing to the completion. So we can suppose
  $K$ is complete. 

  Consider the projective curve $\mathcal D_0$ over $\cO_K$
  given by Theorem~\ref{thm:FM} as
  compactification of $\whW_0 \setminus \{ p_0 \}$ over $\cO_{K}$.
  Its generic fiber $C_0$ has genus $g/2$  by Proposition~\ref{prop:W0}(1). 
  The model $\mathcal D_0$ and $W_0\setminus \{ p_0\}$ share the same
  singular points (Remark~\ref{rmk:lg}). 
  Let $\cC_0$ be the minimal desingularization of $\mathcal D_0$. Then
  $(\cC_0)_k\simeq X_0$. The only possible exceptional divisor in $\cC_0$
  is $\Gamma_0$. But this would imply that the latter is $\PP^1_k$ of
  self-intersection 
  number $-1$ in $\cC_0$, therefore of self-intersection number $-2$ in
  $\cC$. This is impossible by Proposition~\ref{prop:positive}.
  So $\cC_0$ is minimal. The proof is the same for $X_n$. 
\end{proof} 

\begin{remark}
  Again, in residue characteristic different from $2$,
  the above corollary holds with $C_0, C_n$ hyperelliptic.   
\end{remark}

\begin{example} Keep the hypothesis of Corollary~\ref{cor:gl} above
  and suppose $\chara(k)=2$.
Then using Lemma~\ref{lem:degx}, in all cases below, we can take $C_0, C_n$ hyperelliptic: 
  \begin{enumerate}[\rm (1)] 
  \item when $(W_0)_k \to (W_0/\qi)_k$ and $(W_n)_k \to (W_n/\qi)_k$ are
    separable (because $|c|=1$ and $\delta(p_0), \delta(p^*_n)$ are odd); 
  \item when $C$ is any of the curves constructed in Example~\ref{ex:Wn}
    (proof omitted); 
  \item when $n>2\nu(2)$, because if $\ord_0 \bar{Q}_0(x)\le g/2$,
    then $\nu(c) \ge n/2 > \nu(2)$, hence $|c|<|2|$.  
  \end{enumerate}
\end{example}

\begin{example} The curve in the counterexample \ref{ex:no_comp}
  has a unique minimal Weierstrass model. Let us construct a
  counterexample with $C$ having more than one minimal Weierstrass model.
  Let $g$ be even and $\chara(k)=2$. 
  Let $\ell\ge 3g/2$, $s\ge 3g+2$ be two integers. Consider the hyperelliptic
  curve $C$ defined by the equation
    \[
      y^2+\pi(\pi^\ell + x^{g/2})y=\pi^{s}u+ x^{g+1}+\pi x^{2g+2}, \quad u\in \cO_K^* 
    \]
    (assumed to having discriminant $\ne 0$). Let $W$ be the Weierstrass
    model of $C$ defined by the same equation and let $p_0, p_\infty\in W_k$ be
    respectively the zero and the pole of $x$. We will show that
    $W$ is minimal, $C$ has two minimal Weierstrass models and
    $\whW\setminus \{ p_0 \}$ does not have hyperelliptic compactification. 
    
First check easily that $\delta(p_0)=\lambda(p_0)=g+1$, $\lambda(p_\infty)=1$, 
and $\lambda(p)\le \delta(p)=1$ for the other points. So $W$  is
minimal by \cite{LRN}, Proposition 4.3(1). 

    Put $x=\pi^2x_2$ and $y=\pi^{(g+1)}y_2$. Then
    \[
y_2^2+(\pi^{\ell-g}+ x_2^{g/2})y_2 = \pi^{s-(2g+2)}u + x_2^{g+1} + \pi^{2g+3}x_2^{2g+2}  
\]
defines a new Weierstrass model $W'$. Let $p_2, p_2^*\in W'_k$ be respectively
the zero and the pole of $x_2$. Then  $\delta(p_2)=\lambda(p_2)=g$, 
$\delta(p_2^*)=\lambda(p_2^*)=g+1$ and $\delta(p)=0$ for the other points of
$W'_k$. Thus $W'$ is also minimal.  Using Theorem~\ref{chain-mwm} we
see that $W_0=W, W_1=W_0(p_0), W_2=W'$ is the minimality chain of the curve.
Let $Q(x)=\pi(\pi^{\ell}+x^{g/2})=\pi Q_0(x)$. Then $\cont(Q)=\pi$ and
$2\ord_0(\bar{Q}_0(x))=g < \delta(p_0)$. So $\whW\setminus \{ p_0 \}$ does not have hyperelliptic compactification by Lemma~\ref{lem:degx}(2). 
\end{example}

\subsection{Stable reduction} \label{st-red} 
 
We refer to \cite{LB}, \S 10.3 for the definition and general properties
of stable curves and stable models. 
When $\chara(k)\ne 2$, the stable reduction
of hyperelliptic curves can be determined in terms of the Weierstrass points.
See {\it e.g.} \cite{DDMM}, Theorem 10.3.   

Recall that for any smooth projective geometrically connected curve $C$ of
  genus $\ge 2$ over $K$, there exists  
  a finite separable extension $L/K$ and an extension of discrete
  valuation rings $\cO_L/\cO_K$ such that $C_L$ admits a stable
  model over $\cO_L$ (\cite{DM}, Corollary 2.7). The geometric
  closed fiber of the latter is called the
  \emph{potential stable reduction of $C$}.  It is independent on the choice
  of $L$. 
 
\begin{proposition} \label{prop:stab}
  Let $C$ be a hyperelliptic curve of genus $g\ge 2$ over $K$.
  Then $C$ has geometrically irreducible stable reduction ({\it e.g.}
  good reduction) if and only if $C$ has 
  a unique minimal Weierstrass model and if the latter has geometrically
  integral and stable closed fiber. 
\end{proposition}

\begin{proof} Suppose that $C$ has a stable model $\cX$ over $\cO_K$. Then the
  hyperelliptic involution $\sigma$ of $C$ acts on $\cX$. Let
  $\cZ=\cX/\qi$. As $\cX_k\to \cZ_k$ is dominant, $\cZ_k$ is also
  geometrically integral. As $\cZ_K=\PP^1_K$ and
  $p_a(\cZ_k)=p_a(\cZ_K)=0$, $\cZ_k$ is a 
  geometrically integral conic with rational points, hence $\cZ_k\simeq \PP^1_k$. 
  This also implies that $\cZ\simeq \PP^1_{\cO_K}$. Therefore 
  $\cX$  is a Weierstrass model of $C$. The closed fiber $\cX_k$ is reduced
  (because stable), and all $p\in \cX_k$ are smooth or 
  ordinary double points.
  By \cite{LTR}, Lemme 6(b) and Lemme 7(e), $\lambda(p)\le \delta(p)\le 2$.
  It follows from \cite{LRN}, Proposition 4.6(1) that 
  $\cX$ is the unique minimal Weierstrass model of $C$. 
  
  The converse is immediate.    
\end{proof}

\begin{corollary}[of Theorem~\ref{thm:FM}] \label{cor:stable} Let $C$ be a
hyperelliptic curve of even genus $g\ge 2$ over $K$ (not necessarily complete).
Suppose that $C$ has more than one minimal 
Weierstrass model. Then the potential stable reduction of $C$ is the
union of two stable curves of genus $g/2$ (replace ``stable'' with
``semi-stable and irreducible'' if $g=2$) intersecting transversely at one point.

Moreover, if $\chara(k)\ne 2$, these 
stable curves can be taken as stable reductions of hyperelliptic curves of genus
$g/2$.
\end{corollary}

\begin{proof} First we extend the ground field  to a finite 
extension of  its completion to get the stable reduction of $C$. 
Then $C$ still have more than more one minimal Weierstrass
model (Proposition~\ref{prop:more_than_one}). So we can suppose $K$ complete
$K$ and that  $C$ has stable reduction over $\cO_K$. 
Let $D_0$ (resp. $D_n$) be a compactification of
$\whW_0 \setminus \{ p_0 \}$ (resp. $\whW_n\setminus \{ p_n^* \}$)
as given by Theorem~\ref{thm:FM}. Then $g(D_0)=g/2=g(D_n)$
(Proposition~\ref{prop:W0}(1)), and the stable reduction of $C$ is
the union of the stable reductions of $D_0$ 
and $D_n$ intersecting transversely at a rational point.

If the residue characteristic is different from $2$, then by
Proposition~\ref{prop:tame}, the curves $D_0$ and $D_n$ can be
chosen to be hyperelliptic. 
\end{proof} 

\end{section}

\begin{section}{The case of genus 2 curves} \label{sect:g2} 

In this section we consider the particular case of genus $2$ curves $C$.
They are automatically hyperelliptic.  We will see that the knowledge of
the minimal Weierstrass models of $C$ determine whether $C$ has
reduction type $[\cK_1-\cK_2-m]$ , and whether it has stable reduction
(Theorem~\ref{thm:stableg2}). In both cases, 
we can determine a volume form for the N\'eron model of the Jacobian
of $C$ (Proposition~\ref{prop:volume}), the Tamagawa number and the
conductor of $\Jac(C)$ (\S~\ref{subsect:Tn}-\ref{subsect:cond}) and, if 
moreover $K$ is a local field, the Euler factor of the Jacobian of $C$
(Theorem~\ref{thm:g2-Euler} and Remark~\ref{rmk:Euler_f}). We end the
section with the computation of some arithmetic local invariants of
the modular curve $X_0(22)$ (\S~\ref{exp:22}).

\subsection{Characterization of the reduction types
\texorpdfstring{$[\cK_1-\cK_2-m]$}{[K1-K2-m]}} \label{double-ell}

Let $\cC$ be the minimal regular model of $C$ over $\cO_K$. 
When $k$ is algebraically closed, the configuration of the closed fiber $\cC_k$
is classified by \cite{Ogg} (incomplete) and by
Namikawa  and Ueno \cite{NU}. Among them there are the families of 
types $[\cK_1-\cK_2-m]$, where $\cK_1, \cK_2$ are Kodaira symbols for elliptic curves and  $m\ge 0$, and the types $[\cK_1-\cK_2-\alpha]$ (that we will denote
by $[\cK_1-\cK_2-(-1)]$).  When $\chara(k)\ne 2$, the type of $\cC_{\bar{k}}$ 
 is given by an algorithm described in \cite{LC},  and implemented in
 PARI/gp ({\tt genus2red}) for genus $2$ curves over $\mathbb Q$. 

\subsubsection{The case $m\ge 1$}  Let $k$ be perfect. 
We will  say that $\cC_k$  has \emph{split type $[\cK_1-\cK_2-m]$} if
it is obtained by taking two curves 
$Y_1, Y_2$ over $k$ corresponding respectively 
to the symbols $\cK_1$, $\cK_2$ over $k$ (\cite{Tate} or \cite{LB}, \S 10.2.1),  attach them by a chain of $m-1$ copies of $\PP^1_k$'s, with transverse  intersections at rational points,  and the chain
intersects $Y_i$ transversely at a rational point belonging to a geometrically
irreducible component $\Gamma_i$ of multiplicity $1$ in $Y_i$.
When $m=1$, $\Gamma_1, \Gamma_2$ meet transversely 
at a rational point. Let us call $\Gamma_1, \Gamma_2$ 
the \emph{principal components}. They are exactly the irreducible components 
$\Gamma$ such that $\deg (\omega_{\cC/\cO_K}|_\Gamma)>0$. 

If $\cC_k$ has type $[\cK_1-\cK_2-m]$ over an algebraic closure
 $\bar{k}$ of $k$, but not over $k$, we say it has
\emph{non-split type $[\cK_1-\cK_2-m]$}.
For the purpose of this work, the knowledge of the precise type over
 $k$, rather than just over $\bar{k}$, is only useful for the
 computation of the Tamagawa number (see \S~\ref{subsect:Tn}). 

Let us have a look at the canonical model $\cC^{\can}$ of $C$ when $g=2$.
 By the uniqueness of the canonical model, the action of $\sigma$
on $C$ extends to $\cC^{\can}$.
By the adjunction formula~\eqref{eq:adj_can}, $\cC^{\can}_k$ has at most
two irreducible components. When this is the case, the two
components are geometrically integral and intersect at a single rational point.
This is because $\cC^{\can}_{\bar{k}}$ has still only two irreducible
components.

As usual, when $C$ has more than one minimal Weierstrass model,
we denote by $(W_i)_{0\le i\le n}$ the minimality chain
(Definition~\ref{extremal-MWM}). 

\begin{lemma} \label{lem:can-mini} The closed fiber
$\cC^{\can}_k$ is union of two irreducible components  
of arithmetic genus $1$ if and only if $C$ has more than one minimal Weierstrass model. When this is the case, $\cC^{\can}\simeq W_0 \vee W_n$, 
\end{lemma}

\begin{proof} When $C$ has more than one minimal Weierstrass model,
  we have a morphism of models $\cC^{\can} \to W_0\vee W_n$ by
  Proposition~\ref{prop:positive}. As $\cC^{\can}_k$ has two irreducible
  components, this morphism is finite and birational, hence is an isomorphism.

Conversely suppose that $\cC^{\can}_k$ is union of two irreducible components 
$\Gamma, \Gamma'$ of arithmetic genus $1$. As $\cC^{\can}/\qi$ is a
model of $\PP^1_K$, $\sigma$ can not permute $\Gamma$ and $\Gamma'$
because otherwise $\Gamma\to (\cC^{\can}/\qi)_k$ would be birational,
hence $0=p_a((\cC^{\can}/\qi)_k)\ge p_a(\Gamma)=1$, absurd. 
This implies that $(\cC^{\can}/\qi)_k$ has two irreducible components which
are reduced. Therefore $\cC^{\can}/\qi=Z\vee Z'$ for two smooth models
$Z, Z'$ of $\PP^1_K$. Let $W, W'$ be respectively the normalization in
$K(C)$ of $Z$ and that of $Z'$. Then similarly to the above
$\cC^{\can}\simeq W\vee W'$ and $\Gamma, \Gamma'$ are respectively
the strict transform of $W_k$ and of $W'_k$ in $\cC^{\can}$.   

Let us show that $W, W'$ are minimal. 
Let $p_0\in W\wedge W'$. As $\Gamma\to W_k$ is an isomorphism
 away from $p_0$ and $\Gamma\cap \Gamma'$ is a single point $x_0$,  
 $\Gamma\to W_k$ is the normalization of $W_k$ at $p_0$, and $p_0$ is a cusp,
 so $\delta(p_0)$ is odd. 
 On the other hands
 \[
1=p_a(\Gamma)=p_a(W_k)-[\delta(p_0)/2]=2-(\delta(p_0)-1)/2.
 \]
 So $\delta(p_0)=3$. By Lemma~\ref{lem:d0}(4), $\delta(p)\le 3$ for
 all $p\in W_k\setminus \{ p_0 \}$. Hence $\lambda(p)\le 3=g+1$ for all
 $p\in W_k$. Therefore $W$ is minimal by \cite{LTR}, Proposition 4.3(1).
 For the same reasons $W'$ is minimal. So $C$ has more than one minimal
 Weierstrass model. It follows from the first part of the proof that
 $W, W'$ are the extremal Weierstrass models. 
\end{proof}

\begin{theorem} \label{K1-K2} Suppose $g(C)=2$. Let
  $\cC$ and $\cC^{\can}$ be respectively the minimal regular model and
  the canonical model of $C$ over $\cO_K$. The following properties 
  are equivalent:
\begin{enumerate}[\rm (i)] 
\item $\cC_k$ has split type $[\cK_1-\cK_2-m]$ with $m\ge 1$; 
\item $\cC^{\can}_k$ has two irreducible components, both of 
  arithmetic genus $1$;  
\item $C$ has more than one minimal Weierstrass model
  (see Remark~\ref{rmk:crit_2}).  
   \end{enumerate}
   When these conditions are satisfied,  we have
   (with the notation of Theorem~\ref{chain-mwm}) 
$m=n/2$ and $\cC^{\can}\simeq W_0\vee W_n$. 
\end{theorem}

\begin{proof} (i) $\implies$ (ii).  The model $\cC^{\can}$ is obtained
    by contracting all but the two  
   principal components of $\cC_k$. The images of the latter intersect each
   other transversely at a rational point. Their arithmetic genera are
   equal to $1$ by the adjunction formula.

(ii) $\implies$ (iii) is Lemma~\ref{lem:can-mini}, and 
(iii) $\implies$ (i) by Theorem~\ref{regular-even} and
Corollary~\ref{cor:gl}.

When Condition (iii) is satisfied, $\cC^{\can}\simeq W_0\vee W_n$ by
Lemma~\ref{lem:can-mini}. As $\cC$ is the minimal desingularization of
$\cC^{\can}$, $m=n/2$ by Theorem~\ref{regular-even}(2).  
  \end{proof}

  \begin{remark}
    \begin{enumerate}[\rm (1)] 
    \item   Note that it can happen that $\cC^{\can}_k$ has two irreducible
    components isomorphic to $\PP^1_k$ and intersecting each other
    at a single rational point, with intersection multiplicity $3$
    ({\it e.g.} \cite{NU}, type [V]). In this case, $\cC_k$ is not
    of type $[\cK_1-\cK_2-m]$. 
\item 
    Suppose that the conditions of Theorem~\ref{K1-K2} are satisfied,
then  $(W_0)_k$ is reduced, there is at most one singular point
  $w_0\in W_0 \setminus \{ p_0 \}$, 
  and we have $w_0$ rational over $k$ and $\delta(w_0)\le 3$ because $\delta(p_0)=3$ (Lemma~\ref{lem:d0}(4)). The singularity $w_0$ can be solved with the extended Tate's algorithm~\ref{prop:exT}. 
    \end{enumerate}
  \end{remark}

Next we study the case when $C$ has a unique minimal Weierstrass model.

\begin{proposition} \label{prop:non_split}
Suppose that $C$ has a unique minimal Weierstrass model $W$.
Then the following conditions are equivalent:
\begin{enumerate}[\rm (i)] 
\item $\cC_k$ has non-split type $[\cK_1-\cK_2-m]$ for some $m\ge 1$;
\item there exists an unramified extension $K'/K$ such that $C_{K'}$ has
  more than one minimal Weierstrass model;
\item there exists a quadratic $p_0\in W_k$ such that
  $\lambda(p_0)=3+\varepsilon(W)$ and for some (then for any)
  unramified extension $K_0/K$ with residue field containing
  $k(p_0)$, $C_{K_0}$ has more than one minimal Weierstrass model. 
\end{enumerate}

Furthermore, assuming these conditions satisfied, then 
\begin{enumerate}[\rm (1)]
\item $\cC_{k(p_0)}$ has split type $[\cK_1-\cK_2-m]$ over $k(p_0)$; 
\item if $Y_i$ is the curve over $k(p_0)$ corresponding to $\cK_i$, 
then $Y_2\simeq\tau(Y_1)$ of $k$-curves,
where $\tau$ is the generator of $\Gal(k(p_0)/k)$.  In particular,
$\cK_1=\cK_2$ as Kodaira symbols over $k(p_0)$.
\item The canonical model $\cC^{\can}$ is not a Weierstrass model. 
\end{enumerate}
\end{proposition}

\begin{proof} Note that the formation of the minimal regular model
  commutes with unramified extensions (\cite{LB}, Proposition 9.3.28).
The equivalence of (i) and (ii) is a direct consequence 
of Theorem~\ref{K1-K2};  that of (ii) and (iii) follows 
from Proposition~\ref{prop:only_one} if $W_k$ is reduced.   
When $W_k$ is non-reduced, this follows from
Proposition~\ref{prop:mini}(3) and Corollary~\ref{cor:mini_nu}. 
Note that  the property on $C_{K_0}$ is then automatic when $p_0$ exists. 
When (iii) is satisfied, (1) follows from Theorem~\ref{K1-K2} by
taking an unramified $K_0$ with residue field $k(p_0)$. 

When $\cC_k$ has non-split type $[\cK_1-\cK_2-m]$, 
by Theorem~\ref{K1-K2}, 
  $(\cC^{\can})_{k(p_0)}$ has two irreducible components while $\cC^{\can}_k$
  is irreducible. So the Galois group of $k(p_0)/k$ permutes the
  irreducible components of $(\cC^{\can})_{k(p_0)}$, hence permutes the
  curves $Y_1$, $Y_2$. This proves (2).

  (3) Let $K_0/K$ be a quadratic unramified extension with residue field
  $k(p_0)$. Then $\cC^{\can}=\cC_0^{\Gal(K_0/K)}$, where $\cC_0$ is the
  canonical model of $C_{K_0}$. We saw above that
  $\Gal(K_0/K)=\Gal(k(p_0)/k)$ acts on $\cC_0/\qi$ by permuting the two
  irreducible components of $(\cC_0/\qi)_{k(p_0)}$. Therefore 
the  closed fiber of $\cC^{\can}/\qi$ is integral with function field
containing $k(p_0)$. Its normalization is isomorphic to $\PP^1_{k(p_0)}$.
In particular $\cC^{\can}/\qi\not\simeq\PP^1_{\cO_K}$ and $\cC^{\can}$ is
not a Weierstrass model. 
\end{proof}

\begin{example} \label{ex:Vns} Fix a quadratic extension $k'/k$
and a Kodaira symbol $\cK$ over $k'$. 
We construct examples of $C$ such that
$\cC_k$ has non-split type $[\cK-\cK-m]$.   

  Let $K'/K$ be  an unramified quadratic extension
  with residue field $k'$, let $E'$ be an elliptic curve over $K'$
  whose Kodaira symbol is $\cK$. If $\cK=I_0$, we will suppose
  that the good reduction of $E'$ has a non-trivial $2$-torsion point,
  which is always possible. 
  Let 
\[ 
  u^2+(\alpha_1 v+ \alpha_3)u=v^3+\alpha_2 v^2 + \alpha_4 v +
  \alpha_6, 
\]
be an equation of the minimal Weierstrass model $\mathcal E'$ of $E'$. 
By a change of variables one can suppose that $\pi \mid \alpha_3$. 

Let $\cO_{K'}=\cO_K[\theta]$ and let $f(x)\in \cO_K[x]$ the minimal 
polynomial of $\theta$. 
Let $a_i(x)\in \cO_K[x]$ of degree $\le 1$ with $a_i(\theta)=\alpha_i$.
Fix an integer $\ell\ge 1$. Consider the equation  
  \begin{multline} \label{ex-Vb2} 
    y^2+\pi^{\ell+1} (a_1(x)f(x)+ \pi^{2\ell+1}a_3(x))y= \\
    \pi (f(x)^3+\pi^{2\ell+1}a_2(x)f(x)^2+\pi^{4\ell+2}a_4(x)f(x)+\pi^{6\ell+3}a_6(x)) 
  \end{multline}
  over $\cO_K$ and suppose that its discriminant is non-zero. Then it
  defines a genus $2$ curve $C$ and a Weierstrass model $W$ of $C$.
  There is only one singular point $p_0=\{ \pi=f(x)=y=0 \}$ and we have
  $k(p_0)=k'$, $\lambda(p_0)=4$ and $\delta(p_0)=3$. So $W$ is the
  unique minimal Weierstrass model of $C$ by \cite{LRN}, Proposition 4.6(1). 
  Let us show that the closed fiber of the minimal regular model of
  $C_{K'}$  has split type $[\cK-\cK-(2\ell+1)]$ if $\ell$ is big   enough. 
  
  Let $\theta'\in \cO_{K'}$ be the conjugate of $\theta$. Then
    $\theta-\theta' \in \cO_{K'}^*$.  
  In $\cO_{K'}[x]$, we can write
  \[ f(x)=(\theta-\theta')(x-\theta)(1+ (\theta-\theta')^{-1}(x-\theta)).\]
  Let 
  \[ x_0=(\theta-\theta')(x-\theta)/\pi^{2\ell+1}, \quad z=y/\pi^{3\ell+2}.\]
  Then $f(x)=\pi^{2\ell+1}(x_0 + \pi^{2\ell+1}o_2(x_0))$
  and $a_i(x)=\alpha_i + \pi^{2\ell+1}o_1(x_0)$, where the symbol $o_d(x_0)$
  means an element of $\cO_{K'}[x]$ of degree at most $d$. 
Equation~\eqref{ex-Vb2} becomes 
\[
  z^2+(\alpha_1 x_0 + \alpha_3 + \pi^{2\ell+1}o_3(x_0))z  =     x_0^3+\alpha_2 x_0^2+
  \alpha_4 x_0 +\alpha_6 + \pi^{2\ell+1}o_6(x_0).  
\] 
This defines a Weierstrass model $W'_0$ of $C_{K'}$. The point at infinity
$p'_0$ is rational over $k'$ and has multiplicities $\delta=\lambda=3$
(use the condition $\pi \mid \alpha_3$). This implies that
$W'_0$ is minimal. We have $(W'_0)_{k'} \setminus \{ p'_0 \})\simeq
(\mathcal E')_{k'} \setminus \{ \infty \}$.  So $(\mathcal E')_{k'}$
is the normalization of $(W'_0)_{k'}$ at $p'_0$. 
It is clear that when $\ell$ is big enough,
$W'_0\setminus \{ p'_0\}$ has the same minimal desingularization as
$\mathcal E' \setminus \{ \infty \}$. This implies in particular that
$W'_0$ is not an inner minimal Weierstrass model. 

Similarly, working with
$(\theta'-\theta)(x-\theta')/\pi^{2\ell+1}$, we get another extremal
minimal Weierstrass model $W'_n$, with $n=4\ell+2$. So over $K'$
we have the split type $[\cK-\cK-(2\ell+1)]$, and the non-split
type $[\cK-\cK-(2\ell+1)]$ over $K$. 
\end{example}

\subsubsection{The case $m=0$} The type $[\cK_1-\cK_2-0]$ (for Kodaira
symbols $\cK_i$ over $k$)  is only defined 
when at least one of the $\cK_i$'s is reducible (that is, any curve $Y_i$ over
$k$ corresponding to $\cK_i$ is reducible), and at least one of the
$Y_i$'s is not semi-stable.

Suppose that $Y_1$ is reducible. Then $\cC_k$ is
obtained by replacing a component $\simeq \PP^1_k$ of multiplicity
one by a geometrically irreducible component of $Y_2$ of multiplicity $1$,
and attach the new curve to the remaining part of $Y_2$.
We say that $\cC_k$ has split type. 
See Figure~\ref{K1K20} for two examples. 
If this only holds after some extension of $k$ we say that $\cC_k$ has
non-split type.

\begin{figure}   
\centering 
\begin{tikzpicture}[scale=1] 
  \coordinate (P3) at  (-0.2, 1.3);
  \coordinate (Q3) at  (1.25, 0.8);
  \draw  (P3)  [below left]  to [bend left] (Q3);
  \draw (-0.2, 1.6) node {$g=1$};
  \draw (0.5, 0) -- (3.6, 0);
  \draw (2, -0.2) node {3};  
  \draw (1, -0.5) -- ( 1, 1.6);
  \draw (0.8, 0.5) node {2};  
  \draw[thick] (2.8, -0.5) -- (2.8, 1.6);
  \draw (2.6, 0.5) node {2};  
\draw[thick] (2.4, 1) -- (3.4, 1);  
\draw (3.9, -0.5) node {,}; 
 
\draw (5, 0.25) -- (8.2, 0.25);  
\draw (5, -0.5) -- (6,1);
\draw (5,  1) -- (6, -0.5);
\draw (7, -0.5) -- (8,1);
\draw (7,  1) -- (8, -0.5); 

\end{tikzpicture} 
\caption{Split types $[\tI_0-\tIV^*_{2}-0]$ and  $[\tIV-\tIV-0]$.}
\label{K1K20}
\end{figure} 

Keep the hypothesis $C$ having a unique minimal Weierstrass model $W$.
Proposition~\ref{prop:non_split} rules out the 
case when $C$ has more than one minimal Weierstrass model
over some unramified extension $K'/K$.

\begin{proposition} \label{prop:K1K20} Suppose that $C$ has a unique
minimal Weierstrass model $W$, and that this holds over any
unramified extension $K'/K$.
\begin{enumerate}[\rm (1)] 
\item The curve $\cC_k$ has type $[\cK_1-\cK_2-0]$
over $\bar{k}$ if and  only if $W_k$ is reduced 
and there is a point $p_0\in W_k$ with $\delta(p_0)=3$. We then
have  $\cC^{\can}=W$.
\item Parts (1) and (2) of Proposition~\ref{prop:non_split} hold with
  $m=0$. 
\end{enumerate} 
\end{proposition}

\begin{proof} (1) For any unramified extension $K'/K$,
  by Lemma~\ref{lem:mini_down}(2) and  Corollary~\ref{cor:mini_nu}, 
  $W_{\cO_{K'}}$ is minimal, hence is the
  unique minimal Weierstrass model of $C_{K'}$. So we can suppose
  $k$ algebraically closed. If $\cC_k$  
has type $[\cK_1-\cK_2-0]$, then $\cC^{\can}_k$ is integral. The quotient 
$\cC^{\can}/\qi$ is a model of $\PP^1_K$ with integral closed fiber,
so it is smooth and $\cC^{\can}$ is a Weierstrass model of $C$. 
Let us show that $\cC^{\can}=W$ and has a point with $\delta(p_0)=3$. 

Let $\Gamma$ be the irreducible component of $\cC_k$ shared by
$\cK_1$ and $\cK_2$. Then $\cC^{\can}_k$ is the image of $\Gamma$
by $\cC\to \cC^{\can}$.  Let us start by showing the existence of 
$p_0\in \cC^{\can}_k$ as above. We can suppose that $\cK_2$ is reducible. 
Let $p_2\in \cC^{\can}_k$ be the image of the intersection point of
$\Gamma$ with the other irreducible components corresponding to $\cK_2$.
As
\[ 2=p_a(\cC^{\can}_k)\ge p_a(\Gamma)+[\delta(p_2)/2]=1+[\delta(p_2)/2] ,\]
  we have $2\le \delta(p_2)\le 3$.
  If $\delta(p_2)=3$, we take $p_0=p_2$. 
  Otherwise $\cK_2=\tI_\nu$ for some $\nu \ge 2$. 
Then by the definition of the types $[\cK_1-\cK_2-0]$, $\cK_1$ is not
semi-stable. Thus, if $\cK_1$ is irreducible, then $\Gamma$ contains
a cusp $p_0\ne p_2$ and $\delta(p_0)=3$. If $\cK_1$ is reducible, the image of
intersection point of $\Gamma$ with the other irreducible components
corresponding to $\cK_1$ has $\delta =3$. 
By Lemma~\ref{lem:d0}(4),   we have
$\lambda(p)\le \delta(p)\le 3$ for all $p\in \cC^{\can}_k$, so
$\cC^{\can}$ is minimal (\cite{LTR}, Proposition 4.3(1)) hence is equal to $W$. 

Conversely, suppose that $W_k$ is reduced and
$\delta(p_0)=3$ for some $p_0\in W_k$. Then by Lemma~\ref{lem:d0}(4), there
are at most two singular points (including $p_0$) on $W_k$ and they all have 
$\delta\le 3$. The proposition then follows from Corollary~\ref{cor:g2desing}.

(2) If $p_0\in W_k(k)$, the proof above says that $\cC_k$ has type
$[\cK_1-\cK_2-0]$ over $k$. Otherwise, by Lemma~\ref{lem:d0}(4),
$[k(p_0):k]=2$. The Galois group of $k(p_0)/k$ permutes the two singular
points of $\cC_{k(p_0)}$ lying over $p_0$, each of which gives a $\cK$, so
it permutes the curves $Y_1, Y_2$. 
\end{proof} 

\begin{remark} When  $C$ has a stably minimal model $W$ 
  with reduced $W_k$ (see \S~\ref{subsect:stm}),
  the minimal regular model of $C$ is either semi-stable
  (if $\delta(p)\le 2$ for all $p\in W_k$) or is of type 
  $[\cK_1-\cK_2-m]$ over $\bar{k}$, with 
  $m\ge 0$ (when $\delta(p_0)=3$ for some $p_0\in W_k$).
  This has already been observed by M\"uller and Stoll  
  in \cite{MuSt}, \S 10.  The converse is almost true, we just need to
  add the case when $C$ has a Weierstrass model $W$ with
  a quadratic point $p_0$ of $\lambda(p_0)=4$
  (see Proposition~\ref{prop:mini}(3.b)).  
\end{remark}

\subsubsection{The case $m=-1$} \label{subsect:K1K2a} 
Let $\cK_1, \cK_2$ be Kodaira symbols over $k$ which over $\bar{k}$
belong to $\{ {\tI}_\nu^*, \tII^*, \tIII^*, \tIV^*\}$. So in a curve $Y_i$ over
$k$ represented by $\cK_i$, there is an irreducible
component $\Theta_i\simeq \PP^1_k$ of multiplicity $1$, 
intersecting another irreducible component $\Gamma_{i}\simeq \PP^1_k$
of multiplicity $2$. Then roughly speaking, a curve $\cC_k$ of type
$[\cK_1-\cK_2-(-1)]$ ($[\cK_1-\cK_2-\alpha]$ in Namikawa-Ueno's original
notation) is obtained by removing $\Theta_i$ from $Y_i$ and then by
glueing the two remaining parts along  $\Gamma_{1} \simeq \Gamma_{2}$.  
See Figure~\eqref{K1K2-1} for an example, the vertical component
at left is $\Gamma_{1}=\Gamma_{2}$. The dashed components
are the removed components $\Theta_1, \Theta_2$. 
 We also write 
\[ [\tI^*_0-\tI^*_0-(-1)]:=[\tI^*_{0-0-0}], \quad 
   [\tI^*_\nu-\tI^*_0-(-1)]:=[\tI^*_{\nu-0-0}], \quad
   [\tI^*_\nu-\tI^*_\ell-(-1)]:=[\tI^*_{\nu-\ell-0}],
   \]
   (\cite{NU}, pages 155, 171, 180).   The terminology of split and non-split
   types is similar to the case $m\ge 0$. 

\begin{figure}   
\centering 
\begin{tikzpicture}[scale=1] 
  \draw (1, 1) -- (2, 1);
  \draw (1, 0.5) -- (2, 0.5);
  \draw[dashed] (1, 0.2) -- (2, 0.2);
  \draw (1, 1.5) -- (2, 1.5); 
  \draw (0.5, 0) -- (5.7, 0);
  \draw (3, -0.2) node {3};  
  \draw (1.4, -0.5) -- (1.4, 1.6);
 \draw (1.6, -0.4) node {2};  
  \draw (3.3, 0.5) node {2};  
  \draw (3.5, -0.5) -- (3.5, 1.6);
  \draw (3, 1) -- (4, 1);  
\draw (4.6, 1) -- (5.6, 1);
 \draw (4.8, 0.5) node {2};  
 \draw (5, -0.5) -- (5, 1.6);
 \draw[dashed] (5.5, -0.5) -- (5.5, 0.5);
\end{tikzpicture} 
\caption{A closed fiber $\cC_k$ of type $[\tI_0^*-\tIV^*-(-1)]$.}
\label{K1K2-1}
\end{figure}

So far, when $C$ has a unique minimal Weierstrass model $W$,
we were mostly focused on the case where $W_k$ is reduced.  

\begin{proposition} \label{prop:e1d3} Let $W$ be
 a minimal Weierstrass model of $C$. Suppose that $W_k$ is non-reduced
    and $\delta(p)\le 3$ for all $p\in W_k$. 
  \begin{enumerate}[\rm (1)] 
  \item If $\max_{p\in W_k} \{ \delta(p) \}=1$, then $\cC_{\bar{k}}$ has type
    $[\tI^*_0-\tI^*_0-(-1)]$. 
  \item Suppose that $\max_{p\in W_k} \{ \delta(p) \}=2$. Let
    $\theta=\sum_{\delta(p)=2} 2[k(p):k]$. 
    \begin{enumerate}[\rm (a)]
\item If $\theta=2$,  then $\cC_{\bar{k}}$ has type $[\tI^*_\nu-\tI^*_0-(-1)]$ 
with $\nu\ge 1$. 
\item If $\theta=4$, then  $\cC_{\bar{k}}$ has type
  $[\tI^*_{\nu}-\tI^*_{\ell}-(-1)]$, with $\nu, \ell\ge 1$. 
\item If $\theta=6$, then 
  $\cC_{\bar{k}}$ has type $[\tI^*_{\nu-\ell-s}]$ (\cite{NU}, page 183) for some $\nu, \ell, s\ge 1$. 
    \end{enumerate}
  \item Suppose that there exists $p_1\in W_k$ with $\delta(p_1)=3$ and
    $\lambda(p_1)\le 3$. 
    Then $\cC_{\bar{k}}$ has type $[\cK_1-\cK_2-(-1)]$ with 
    $\cK_1\in \{ \tII^*, \tIII^*, \tIV^*\}$ and:   
    \begin{enumerate}[\rm (a)]
    \item $\cK_2=\tI^*_0$ if $\delta(p)\le 1$ for all
      $p\in W_k \setminus \{ p_1 \}$; 
    \item $\cK_2=\tI^*_\nu$, $\nu \ge 1$, 
      if there exists $p_2\in W_k $ with $\delta(p_2)=2$; 
    \item   $\cK_2 \in  \{ \tII^*, \tIII^*, \tIV^*\}$ if there exists
      $p_2\in W_k\setminus \{ p_1\}$ with $\delta(p_2)=3$.    
    \end{enumerate}
  \item In Case (2.a), $\cC_k$ is split. 
    In Cases (2.b) and (3), if $\cC_k$ is non-split, then there exists
    a quadratic point $p_0\in W_k$ such that $\cC_{k(p_0)}$ is split
    of type $[\cK-\cK-(-1)]$ over $k(p_0)$. 
\item We have $\cC^{\can}=W$. 
 \end{enumerate}
\end{proposition}

\begin{proof} The singular points of $W$ are those $p$ with
  $\delta(p)\ge 1$. Recall that 
\[  \sum_{p\in W_k} \delta(p)[k(p):k]=6\]  
by Lemma~\ref{lem:d0}(3).
\smallskip

\emph{Claim: we have $\lambda(p)\le 3$ for all $p\in W_k$.}   
In Cases (1)-(2), $\lambda(p)\le \delta(p)+1 \le 3$
for all $p\in W_k$. In Case (3), as $\delta(p_1)=3$,
$\lambda(p)\le \delta(p)+1\le 2$ for all non-rational points
by the above equality;  rational points have $\lambda\le 3$ because $W$ is
minimal.

By Corollary~\ref{cor:mini_nu},  $W_{\cO_{K'}}$ is the unique minimal
  Weierstrass model of $C_{K'}$ over any unramified extension $K'$
of $K$.  So we can replace $K$ by an
unramified extension and assume that all singular points of $W$
are rational   over $k$.
  
(1)  Apply Lemma~\ref{lem:l2d1} (even if $p$ is not rational over $k$).

(2) If $\delta(p)=2$,  the desingularization is the same as for the $p_1\in W_1$ in 
 Proposition~\ref{prop:exT}(5.b).
 
(3) The desingularization of $p$ with $\delta(p)=3$ is the same as 
for $p_1\in W_1$ in  Proposition~\ref{prop:exT}(5.c). 

(4) In Case (2.a), the point $p$ with $\delta(p)=2$ is rational.
In (non-split) Cases (2.b) and (3), the point with $\delta(p)\ge 2$ is
quadratic. As for the cases $m\ge 0$, $\Gal(k(p)/k)$ permutes
the curves corresponding to $\cK_1, \cK_2$ over $k(p)$. 

(5) By (1)-(3), the desingularization of $W$ adds parts of curves
corresponding to Kodaira symbols, so the latter are contracted
to points in $\cC^{\can}$ and $W=\cC^{\can}$. 
\end{proof} 

\begin{remark} \label{rmk:nsp3} If in Proposition~\ref{prop:e1d3}(3)
we have $\lambda(p_1)\ge 4$, then $p_1$ is quadratic over $k$ and
$\lambda(p_1)=4$. This implies that $\cC_k$ has non-split type $[\cK-\cK-m]$
for some $m\ge 1$ (Propositions~\ref{prop:mini} and \ref{prop:non_split}).
\end{remark}

\begin{remark} Conversely, if $\cC_{\bar{k}}$ has one of the types given in
  Proposition~\ref{prop:e1d3}, then $\cC^{\can}/\qi$ is a smooth
  model of $\PP^1_K$ (use \cite{Lcd}, Proposition 7(a)). So
  $\cC^{\can}$ is a Weierstrass model $W$. Then it is not hard to show,
  by looking at the singular points of $W$,  that 
  $W$ is minimal and $\delta(p)\le 3$ for all $p\in W_k$.  In summary,
  in this section we covered the case when $C$, of genus $2$,
  has a minimal Weierstrass model $W$ such that $\delta(p)\le 3$
  for all $p\in W_k$.  
\end{remark}

\subsection{Stable reduction of curves of genus 2}\label{stable-g2} 

We characterize in this subsection
the existence (and give the description) of the stable 
reduction of $C$ in terms of the minimal Weierstrass models of $C$
(see Theorem~\ref{thm:stableg2} and Algorithm~\ref{algo:stab}). 
We do not assume that $C$ has more than one minimal Weierstrass model. 

The stable curves over an algebraic closure $\bar{k}$ of $k$ can be divided into
seven types:
\begin{enumerate}[\rm (I)] 
\item smooth;
\item irreducible with one double point (with normalization of genus $1$);
\item irreducible (rational) with two double points;
\item union of two projective lines intersecting each other at three points;
\item union of two smooth curves of genus $1$, intersecting each other at one point;
\item union of a smooth curve of genus $1$ with a rational curve with one double
  point, intersecting each other at one point;
\item union of two rational curves, each with one double point, and
  intersecting each other at one point. 
\end{enumerate}
In the cases (I)--(IV), the quotient of the curve by the hyperelliptic involution 
is isomorphic to $\PP^1_{\bar{k}}$, while in the other cases the quotient is union of
two projective lines intersecting transversely at one point.

A stable curve $X$ over $k$ can be irreducible with $X_{\bar{k}}$ reducible
of type (IV), (V) or (VII) (not for the type (VI) because the two
irreducible components are not isomorphic). 

\begin{definition} Let $X$ be a stable curve of genus $2$ over a field $k$
  such that $X_{\bar{k}}$ is of type (V) or (VII). We will say that $X$ is of
\begin{enumerate}[\rm (1)] 
  \item \emph{split type} (V) if it is union of two smooth connected curves 
of genus $1$ over $k$ intersecting transversely at a rational point; 
\item \emph{non-split} type (V) if it is irreducible, and if $X_{\bar{k}}$ has type (V).
\item \emph{split type} (VII) if it is union of two rational connected curves 
of arithmetic genus $1$ over $k$ intersecting transversely at a rational point; 
\item \emph{non-split type} (VII) if it is irreducible, and if $X_{\bar{k}}$
  has type (VII).
\end{enumerate}
The type (VI) is always considered as split. In \cite{DDMM}, \S 18, a finer
classification of (semi-) stable curves of genus $2$ is given. For example, split (resp. non-split) type (V)
corresponds to $1\times_r 1$ (resp. $1\tilde{\times}_r 1$). 

The \emph{stable reduction of $C$} is, by definition, the closed fiber of the
stable model of $C$ (if it exists). 
\end{definition}

For a Weierstrass model $W$ to be stable
it is enough to be \emph{semi-stable} (that is, $W_k$ is reduced and
$\delta(p)\le 2$ for all $p\in W_k$). Indeed, $W_{\bar{k}}$ is either 
irreducible, or is union of two copies of $\PP^1_{\bar{k}}$ 
intersecting each other at $3$ points. In both cases it is stable.  

\begin{theorem} \label{thm:stableg2} Let $g(C)=2$. 
  \begin{enumerate}[{\rm (1)}] 
  \item The curve $C$ has stable reduction of type {\rm (I), (II, (III)} or
    {\rm (IV)} over $\bar{k}$
    if and only if $C$ has a unique minimal Weierstrass model and if the latter
    is semi-stable. 
  \item The curve $C$ has stable reduction of split type {\rm (V), (VI)}, or
    {\rm (VII)} if 
    and only if $C$ has more than one minimal Weierstrass model, and if
    the extremal minimal Weierstrass models $W_0, W_n$ 
    are semi-stable away from $p_0$ and $p_n^*$ respectively
    (Theorem~\ref{chain-mwm}).  
\item The curve $C$ has stable reduction of non-split type {\rm (V)} 
    (resp. {\rm (VII)}) over $K$ if and only if the following
    conditions are satisfied:  
\begin{enumerate}
\item $C$ has a unique minimal Weierstrass model $W$ with a quadratic point
   $p_0\in W_k$ of $\lambda(p_0)=3$ or $4$; 
\item taking an unramified extension $K_0/K$ with residue field $k(p_0)$,
$C_{K_0}$ has stable reduction of split type {\rm (V)} (resp. {\rm (VII)}).
    \end{enumerate}
\end{enumerate}
\end{theorem}

\begin{proof} (1) For the types (I), (II) and (III) (corresponding to geometrically
  integral stable curves of genus $2$), our statement is a just
  a special case of Proposition~\ref{prop:stab}. The type (IV) is proved in the
  same way, because the quotient $X/\qi=\PP^1_k$  in this case.
  
  (2) is a special case of Theorem~\ref{K1-K2} because
  stable reductions of split types (V), (VI) and (VII) are
  respectively equivalent to the split types
  $[I_0-I_0-m]$, $[I_0-I_r-m]$ with $r>0$ and
  $[I_q-I_r-m]$ with $q,r >0$ for $\cC_k$. Similarly,  
  (3) is a special case of Proposition~\ref{prop:non_split}. 
\end{proof}

\begin{algo} \label{algo:stab} Let $g(C)=2$. We summarize the previous
  results allowing us to detect whether $C$ has
  stable reduction over $K$, and then determines it when the answer is
  positive. 
  \begin{enumerate}[\rm (1)]
  \item Use Algorithm~\ref{algo:mwm} to find the minimal Weierstrass models of
    $C$. 
  \item Suppose that $C$ has more than one minimal Weierstrass model. 
   See Theorem~\ref{chain-mwm}. 
    \begin{enumerate}
    \item If $W_0$ and $W_n$ are semi-stable away from respectively $p_0$
      and $p_n^*$ ({\it i.e.}, $\delta(p)\le 2$ for all $p\in W_0, W_n$ different
      from $p_0$ and $p_n^*$), 
then $C$ has stable reduction of split type (V), (VI) or (VII). The irreducible
    components of the stable reduction are the normalizations of
    $(W_0)_k, (W_n)_k$ at respectively $p_0$ and $p_n^*$.
    \item Otherwise $C$ does not have stable reduction over $K$. 
        \end{enumerate}
\item Suppose that $C$ has a unique minimal Weierstrass model $W$.
  \begin{enumerate}
  \item If $W$ is semi-stable, the stable reduction of $C$ is $W_k$ and its 
type is (I), (II), (III) or (IV) over $\bar{k}$. 
\item Suppose that $W_k$ has a 
  quadratic point $p_0\in W_k$ with $\lambda(p_0)=3$ or $4$.
  Let $K_0/K$ be a quadratic unramified extension with
  residue field $k(p_0)$. 
If $C_{K_0}$ has split reduction type (V) or (VII), then $C$ has stable
reduction over $K$, of non-split type (V) or (VII). 
\item If none of the above condition is satisfied,
  $C$ does not have stable reduction over $K$.  
  \end{enumerate}
\end{enumerate}  
\end{algo}

\subsection{Volume form} \label{subsect:vol} 
Let $A=\Jac(C)$ be the Jacobian of $C$ and
let $\cA$ be its N\'eron model over $\cO_K$ with zero section
$e: \Spec \cO_K\to \cA$. A \emph{volume form} on
$\cA$ is a basis of $\det \omega_{\cA/\cO_K}$ 
where $\omega_{\cA/\cO_K}:=e^*\Omega^1_{\cA/\cO_K}$. 

Let $g=2$.  Let $y^2+Q(x)y=P(x)$ be an equation of $C$ over $K$. 
Let $\eta = dx/(2y+Q(x)) \wedge xdx/(2y+Q(x))\in \det H^0(C, \omega_{C/K})$
and let $\Delta$ be the discriminant of this equation. 
It is well known that the section
$\Delta^2 \eta^{\otimes 20}$ of $(\det H^0(C, \omega_{C/K}))^{\otimes 20}$ is
independent of the equation. This implies that if $W$ is a minimal Weierstrass
model of $C$, and
$\nu(\Delta_{C})=\nu(\Delta_W)$ is the minimal discriminant of $C$, then
\begin{equation}
  \label{eq:vol_w}
\det H^0(W, \omega_{W/\cO_K})=
\pi^{\frac{\nu(\Delta) - \nu(\Delta_{C})}{10}} 
\left(\frac{dx}{2y+G(x)} \wedge  \frac{xdx}{2y+G(x)}\right) \cO_K
\end{equation}
Recall for the next proposition that
$\varepsilon(W)=0$ or $1$ depending on whether $W_k$ is reduced or not. 

\begin{proposition} \label{prop:volume} Let $g=2$. Let $m_C$ be the number
  of minimal Weierstrass models of $C$ and let $W$ be one of them, 
defined by an equation $y^2+Q(x)y=P(x)$. 
\begin{enumerate}[\rm (1)] 
\item If $m_C\ge 2$, then  
  \begin{equation}
    \label{eq:vol}
    \det \omega_{\cA/\cO_K} =\pi^{m_C-1} \frac{dx}{2y+Q(x)}\wedge \frac{xdx}{2y+Q(x)}\cO_K 
  \end{equation}
as submodules of $\det H^0(C, \Omega^{1}_{C/K})$. 
\item Suppose that $m_C=1$, $\delta(p)\le 3$ 
  for all $p\in W_k$, and that there is no quadratic point $p_0\in W_k$ with
  $\lambda(p_0)=3+\varepsilon(W)$. Then Equality~\eqref{eq:vol} holds. 
\item Suppose that $m_C=1$, $\delta(p)\le 3$ for all $p\in W_k$, and
  there exists $p_0\in W_k$ with $\lambda(p_0)=3+\varepsilon(W)$. Let
  $m'\ge 1$ be the number of minimal Weierstrass models of $C_{K'}$ where
  $K'$ is an unramified extension of $K$ with residue field $k(p_0)$, then
Then 
  \[ \det \omega_{\cA/\cO_K} =\pi^{m'-1+\varepsilon(W)} \frac{dx}{2y+Q(x)}\wedge \frac{xdx}{2y+Q(x)}\cO_K 
  \]
as submodules of $\det H^0(C, \Omega^{1}_{C/K})$.
\end{enumerate}
\end{proposition}

\begin{proof} (1) Let $W_0, W_n$ be the extremal minimal Weierstrass models of
  $C$. Then $n=2m_C-2$.  
  By Lemma~\ref{lem:can-mini}, $\cC^{\can}\simeq W_0\vee W_n$. 
We have by Proposition~\ref{prop:volume_can} a canonical isomorphism 
\[
  \det \omega_{\cA/\cO_K} \simeq \det H^0(\cC^{\can}, \omega_{\cC^{\can}/\cO_K}).
\]  
We can now apply Proposition~\ref{prop:omega}.
This also follows from
results of \cite{Lcd} but in a more indirect way.

(2) First suppose $W_k$ is reduced. If $\delta(p)\le 2$ for all $p\in W_k$, then  
$W$ is stable (Theorem~\ref{thm:stableg2}(1)), hence $\cC^{\can}=W$. 
Suppose that $W_k$ has a point with $\delta=3$. 
By Proposition~\ref{prop:only_one},
$W$ is the unique minimal Weierstrass model of $C$ over any unramified
extension of $K$. Proposition~\ref{prop:K1K20} says that $\cC^{\can}=W$
and we conclude by Proposition~\ref{prop:volume_can}.

Suppose now $W_k$ non-reduced. According to Proposition~\ref{prop:mini},
$W$ is minimal over any unramified extension of $K$. As the desired equality 
is compatible with unramified extensions of $K$, and such extensions
do not change the multiplicities $\delta$, we can suppose that $k$ is
algebraically closed. Then $\cC^{\can}=W$ by Proposition~\ref{prop:e1d3}(4) 
and we conclude as in the previous case.

(3) If $W_k$ is reduced, we just apply (1)-(2) to $W_{\cO_{K'}}$. Suppose that 
$W_k$ is non-reduced. Let $W'$ be a minimal Weierstrass model
of $C_{K'}$ over $\cO_{K'}$. Then $\nu(\Delta_W)=\nu(\Delta_{W'})+10$
(Proposition~\ref{prop:mini}(3.b)). Now apply  Equality \eqref{eq:vol_w} 
and (1) to $W'$ which is not unique by Proposition~\ref{prop:mini}(3.a).  
\end{proof}
   
\subsection{Euler factors} \label{subsect:EF} In this subsection,
$k=\FF_q$ is a finite field.  Let $A$ be an abelian variety over $K$ with
N\'eron model $\cA$ over $\cO_K$. Let $\ell$ be a prime number different
from $\chara(k)$ and let $V_\ell(\cA_k^0)$
be the Tate module of $\cA_k^0$ (the identity component of $\cA_k$). 
The Frobenius $x\mapsto x^q$ of $\bar{k}$
on $\cA_k^0$ induces a $k$-linear 
automorphism  $\mathrm{Frob}$ on $V_\ell(\cA_k^0)$. Denote by
\[
  P({\cA_k^0}, T)=\det (1- T. \mathrm{Frob}^{-1}|_{V_\ell(\cA_k^0)^{\vee}})\in
  \mathbb Q_\ell[T].   
\] 
Then the \emph{Euler factor} of $A$ is the function
$P({\cA_k^0}, q^{-s})^{-1}$ over $\mathbb C$. See also \cite{BW}, \S 2.  
We have
\[ \deg P({\cA_k^0}, T)=\dim_{{\mathbb Q}_\ell} V_\ell(\cA_k^0)=
2a_C+t_C\]
where $a_C$ and $t_C$ are respectively the
abelian and toric ranks of $\cA_k^0$. We have $a_C+t_C\le g$.
In particular, when $g=2$, then $ \deg P({\cA_k^0}, T) \ge 3$ only when
$C$ has stable reduction over $K$. 

\begin{theorem}\label{thm:g2-Euler} Suppose that $g(C)=2$ and that $C$
  has more than one  minimal Weierstrass model. Let $W_0, W_n$ be the
  extremal ones (Definition~\ref{extremal-MWM}). Let $\Gamma_0, \Gamma_n$
  be respectively the normalizations of 
  $(W_0)_k$ and $(W_n)_k$ at $p_0$ and $p_n^*$. Then we have
\[
P({\cA^0_k}, T)=P({\Pic^0_{\Gamma_0/k}}, T)P({\Pic^0_{\Gamma_n/k}}, T). 
\]
\end{theorem}

\begin{proof} As $\cC^{\can}_k=(W_0\vee W_n)_k$ is reduced, by
  \cite{BLR}, Theorem 9.7/1, $\Pic^0_{\cC^{\can}/\cO_K}$ is representable and 
the canonical  morphism
  \[ \Pic^0_{\cC^{\can}/\cO_K}\to \cA^0\]
is an isomorphism. As $\cC^{\can}_k$ is the union of
  $\Gamma_0, \Gamma_n$ intersecting transversely at a rational point, we have
\[\Pic^0_{\cC^{\can}_k/k}\simeq \Pic^0_{\Gamma_0/k}\times
  \Pic^0_{\Gamma_n/k},\]
therefore
  $V_\ell( \Pic^0_{\cC^{\can}_k/k}) \simeq V_\ell(\Pic^0_{\Gamma_0/k}) \oplus 
V_\ell(  \Pic^0_{\Gamma_n/k})$ and the theorem is proved. 
\end{proof}

\begin{remark} As $\Gamma_0$ is geometrically integral of arithmetic genus $1$, 
  $\Pic^0_{\Gamma_0/k}$ is isomorphic to $\Gamma_0$ if $\Gamma_0$ is smooth; 
  is a torus of dimension $1$ if there is a node on $\Gamma_0$,
  and is trivial if there is a cusp.  
\end{remark}

\begin{remark} In \cite{MS}, the case of the reduction type $[I_0-I_0-m]$
(split or not)  is solved if $\chara(k)\ne 2$, using the theory of cluster pictures. 
\end{remark}

\begin{remark}\label{rmk:Euler_f} Suppose that $C$ has a unique minimal Weierstrass
  model $W$. 
\begin{enumerate}[\rm (1)] 
\item If $W_k$ is reduced and satisfies the conditions of
Proposition~\ref{prop:volume}(2), then $\cC^{\can}=W$ and
  as in the theorem above we have $\Pic^0_{\cC^{\can}/\cO_K}\simeq \cA^0$, hence 
  \[\cA^0_k =\Pic^0_{W_k/k}  \ \text{ and } \ P({\cA^0_k}, T)=P({\Pic^0_{W_k/k}}, T).\] 
\item Suppose that $W_k$ is non-reduced and satisfies the conditions of
  Proposition~\ref{prop:volume}(2). Then it follows from
  Proposition~\ref{prop:e1d3} 
  that $\cA^0_{\bar{k}}\simeq \Pic^0_{\cC_{\bar{k}/\bar{k}}}$ is unipotent and
$P({\cA^0_k}, T)=1$. 
\item Suppose that $W$ is as Proposition~\ref{prop:volume}(3).
  Let $K'/K$ be an unramified extension with residue field $k'=k(p_0)$,
  where $p_0\in W_k$ is the quadratic point with $\lambda(p_0)=3+\varepsilon$.
  If $W_{\cO_{K'}}$ is the unique minimal Weierstrass model of $C_{K'}$, then
  as in (1), we have $P(\mathcal A^0_k, T)=P(\Pic^0_{W_k/k}, T)$. 
  
  Suppose now that $C_{K'}$ has more than one minimal Weierstrass model.
  Let $W', W''$ be the extremal   minimal Weierstrass models 
  of $C_{K'}$ and let $\Gamma'$ be the
    normalization of $W'_{k'}$ at $p'_0\in W'\wedge W''$. Then  
  \[ P(\cA^0_k,T)=P(\Pic^{0}_{\Gamma'/k'}, T^2).\] 
  Indeed, we have 
  $\cA^0_{k'}=\Pic^0_{\cC^{\can}_{k'}/k'}\simeq \Pic^{0}_{\Gamma'/k'}\times
  {}^{\sigma}(\Pic^{0}_{\Gamma'/k'})$ as seen in the proof of
  Theorem~\ref{thm:g2-Euler}.   So 
  \[\cA^{0}_k\simeq      \mathrm{Res}_{k'/k} (\Pic^0_{\Gamma'/k'}). \] 
  To conclude, notice that for any smooth commutative algebraic group $G'$ over
  a finite extension $k'/k$, we have  $P(\Res_{k'/k}G', T)=P(G, T^{[k':k]})$  
(see the proof of \cite{MS}, Proposition 3.9(2)  or
  that of \cite{Mil-AAV}, Proposition 3). 
  \end{enumerate}
\end{remark}

\subsection{Tamagawa numbers} \label{subsect:Tn} 
The \emph{Tamagawa number of $\Jac(C)$} is the order of the group of
rational points $\Phi_C(k)$ of the group of components $\Phi_C$ of $\cA$.
Suppose that $g(C)=2$.
For any Kodaira symbol $\cK$ over $k$, we
denote
\[\Phi({\cK}):= \Phi_E(k)\]
where $E$ is an elliptic  
curve over $K$ whose minimal regular model has type $\cK$. It depends only
on $\cK$ and is given in Tate's algorithm, see also \cite{LB}, Remark 10.2.24. 
Propositions~\ref{prop:tkk}, \ref{prop:tkkc} and \ref{prop:m-1} below will be proved in a latter work under more general hypothesis. 

\begin{proposition} \label{prop:tkk} Suppose that $g(C)=2$ and
  $\cC_k$ has split type $[\cK_1-\cK_2-m]$, $m\ge 0$.
Then 
      \[ \Phi_C(k)\simeq \Phi({\cK_1})\times \Phi({\cK_2}).\] 
\end{proposition}

If $\cC_k$ has non-split type $[\cK-\cK-m]$, $m\ge 0$,
then there is a unique quadratic point $p_0$ in the unique minimal
Weierstrass model  
of $C$ with $\delta(p_0)=3$, and $\cC_{k(p_0)}$ has split type $[\cK-\cK-m]$
over $k(p_0)$ (see Propositions~\ref{prop:non_split}, \ref{prop:K1K20}). 
 
\begin{proposition} \label{prop:tkkc} Keep the above hypothesis and
  notation. Let 
  $[\cK-\cK-m]$ be the (split) type of $\cC_{k(p_0)}$ over $k(p_0)$. Then
  $\Phi_{C}(k) \simeq \Phi({\cK})$. 
 \end{proposition}

\begin{proposition} \label{prop:m-1} Suppose $k$ finite.
  \begin{enumerate}[\rm (1)] 
\item If $\cC_k$ has split type   $[\cK_1-\cK_2-(-1)]$, then 
\[ |\Phi_C(k)|=|\Phi({\cK_1})||\Phi({\cK_2})|.\]
\item If $\cC_{k}$  has non-split type
  $\ne  [{\tI}^*_0-{\tI}^*_0-(-1)], [{\tI}^*_{\nu-\ell-s}]$ over $\bar{k}$, 
  then $\cC_{k(p_0)}$ has split type
  $[\cK-\cK-(-1)]$  (Proposition~\ref{prop:e1d3}(4)) and 
  $|\Phi_C(k)=|\Phi(\cK)|$.
  \end{enumerate} 
\end{proposition}

Suppose that $\cC_{\bar{k}}$
has type  $[{\tI}^*_0-{\tI}^*_0-(-1)]$. Then
$\Phi_C(\bar{k})\simeq (\Z/2\Z)^4$. The group $\Phi_C(k)$ 
is determined by its order. 
By Proposition~\ref{prop:e1d3}(1), $\delta(p)\le 1$ for all $p\in W_k$. Such
a singularity is solved by adding a $\PP^1_{k(p)}$ of multiplicity $1$
(Lemma~\ref{lem:l2d1}), of self-intersection $-2[k(p):k]$. 
Let $T_C$ be the tuple of the $[k(p): k]$, $\delta(p)=1$,
in decrea\-sing order.  

\begin{proposition}\label{prop:tn_tc} Suppose $k$ finite and $\cC_k$ has type
  $[{\tI}^*_0-{\tI}^*_0-(-1)]$ over $\bar{k}$. Let $T_C$ be defined as above.
  Then $\Phi_C(k)$ is given in the following tables 
  \begin{enumerate}[\rm (1)]
  \item
{\renewcommand{\arraystretch}{1.5}
\[
\begin{array}{ |c|c|c|c|c|c|c| } 
 \hline
 T_C  & (3,3) & (3,2,1) &  (3,1,1,1) & (2,2,1,1) & (2,1,1,1,1) & (1,1,1,1,1,1) \\
\hline 
 \Phi_C(k) & (0) & \Z/2\Z & (\Z/2\Z)^2 & (\Z/2\Z)^2 & (\Z/2\Z)^3& (\Z/2\Z)^4\\ 
 \hline
\end{array}
\]
 }
\item 
{\renewcommand{\arraystretch}{1.5}
\[
\begin{array}{ |c|c|c|c|c|c| } 
 \hline
 T_C  & (6) & (5,1) & (4,2) & (4,1,1) & (2,2,2) \\
\hline 
 \Phi_C(k) & (0) & (0) & \Z/2\Z & \Z/2\Z & (\Z/2\Z)^2 \\ 
 \hline
\end{array}
\]
 }
  \end{enumerate}
\end{proposition}

\begin{proof} (1) The curve $\cC_k$ has split type $[\cK_1-\cK_2-(-1)]$ with
  $\cK_i\in \{ \tI^*_0, \tI^*_{0,2}, \tI^*_{0, 3} \}$ (see notation in 
  \cite{LB}, \S 10.2.1.)  For example the tuple $(2,2,1,1)$
  corresponds to $[\tI_{0,2}^*-\tI_{0,2}^*-(-1)]$. 
  These types are handled by Proposition~\ref{prop:m-1}(1).

  (2) Use \cite{BL}, Theorem 1.17 (the correcting term $qd\Z/d'\Z=\{ 0\}$)
  and Remark 1.13. 
\end{proof}

Let us terminate with the types $[{\tI}^*_{\nu-\ell-s}]$, $\nu, \ell, s\ge 1$. 
  By Proposition~\ref{prop:e1d3}(2.c),
  $\sum_{p\in W_k, \delta(p)=2} 2[k(p):k]=6$. Let $T_C$ be the tuple of the
    degrees  $[k(p):k]$ in this sum. 

\begin{proposition} 
  If $k$ is finite and $\cC_{\bar{k}}$ has type
  $[{\tI}^*_{\nu-\ell-s}]$ with $\nu, \ell,s>0$ as above. Then
  $|\Phi_C(k)|$ is given by 
{\renewcommand{\arraystretch}{1.5}
\[
\begin{array}{ |c|c|c|c| } 
 \hline
 T_C        & (3) & (2,1) & (1,1,1) \\
\hline 
 |\Phi_C(k)| &  1   & 4   & 16\\ 
 \hline
\end{array}
\]
 }
\end{proposition}

\begin{proof} Let 
  $\Gamma_1\simeq \PP^1_k$ be the strict transform of $(W_k)_{\mathrm{red}}$
  in $\cC$.  With the notation of \cite{BL}, Theorem 1.17, we have 
  $\Phi_C(k)=\ker\alpha/\Img \beta$, and we apply {\it op. cit.,} Remark 1.16
  with the entry $m^*_{11}$ of the adjoint matrix $M^*$. 
\end{proof} 

\subsection{Conductor of  \texorpdfstring{$\mathrm{Jac}(C)$}{Jac(C)}}
\label{subsect:cond} 
We compute $f_C$, the conductor of $\Jac(C)$. Recall that the number
$N$ of geometric components of $\cC_k$ in the following
proposition can be computed using
extended Tate's algorithm (Proposition~\ref{prop:exT} and Remark~\ref{rmk:e1d3}).

\begin{proposition} \label{prop:cond} Let $g=2$.
Let $\nu(\Delta_C)$ be the minimal discriminant of $C$ and let 
$N$ be the number of irreducible components of $\cC_{\bar{k}}$. 
   \begin{enumerate}[\rm (1)] 
   \item Suppose that $\cC_{\bar{k}}$ is of type
   $[\cK_1-\cK_2-m]$, $m\ge 0$. Then 
 \[
\nu(\Delta_C)=f_C+N-1+ 11m.  
 \]
\item Suppose that $C$ satisfies the conditions of Proposition~\ref{prop:e1d3}.
  Then 
 \[
\nu(\Delta_C)=f_C+N-1.  
 \]
\end{enumerate}
\end{proposition}

\begin{proof} Let $c(\cC)$ be defined by Equality (1) in \cite{Lcd}, page 51. Then
 combining Proposition 1 with Th\'eor\`eme 2 in {\it op. cit.}, we get 
   \[
\nu(\Delta_C)=f_C+N-1 + 11c(\cC).  
 \]

In Case (1), we have 
 $c(\cC)=m$ by \cite{Lcd}, Proposition 4 if $m=0$ and Proposition 6 if $m\ge 1$. 
In Case (2), $\cC^{\can}/\qi=W/\qi$ is smooth, so $c(\cC)=0$ by {\it loc. cit.},
Proposition 4.       
\end{proof}

\subsection{An example with \texorpdfstring{$X_0(22)$}{X0(22)}}\label{exp:22} Consider the modular curve  $C:=X_0(22)$ defined by
  the equation
      \begin{equation}\label{eq:mod} 
      y^2=x^6 + 12x^5 + 56x^4 + 148x^3 + 224x^2 + 192x + 64.   
      \end{equation}
      The command  PARI/gp {\tt hyperellminimalmodel(P)}, where
      $P$ is the right-hand side polynomial in \eqref{eq:mod}, 
      tells us  that 
  \begin{equation} \label{eq:min}  
    y_1^2+y_1=16x_1^6+ 48x_1^5 +
    56x_1^4 + 37x_1^3 +  14x_1^2 + 3x_1, 
  \end{equation}
  with $y=16y_1+8$, $x=4x_1$,  defines a minimal Weierstrass model
  $W_0$ over $\Z$,  with discriminant  $\Delta=2^{12}11^4$.
  Thus the curve 
  $C$  has good reduction at all primes $\ne 2, 11$ and Equation~\eqref{eq:mod}
  is minimal at all primes $\ne 2$. Let us determine
  some local factors of $A$ at the primes of bad reduction. 
  \medskip

  $\bullet$  {\bf Euler factor et Tamagawa number at $p=11$}. We  can use Equation~\eqref{eq:mod}.
  As   
\[ P(x) \equiv (x+3)^2(x+5)^2(x+9)^2 \mod 11,\]   
 $(W_0)_{\FF_{11}}$ is union of $2$ copies 
of $\PP^1_{\FF_{11}}$ intersecting transversely at $3$ rational points.
Thus 
    \[ L_{11}(\Jac(C), s)=(1-11^{-s})^{-2}. \]
 As the irreducible components of $\cC_{\FF_{11}}$ 
 are geometrically integral (they are isomorphic to $\PP^1_{\FF_{11}}$),
 the group of connected components 
of $\cA_{\FF_{11}}$ is a constant group.     
    {\tt genus2red(P)} tells us that the Tamagawa number of $A$ at
       $11$ is  $c_{11}=5$.  
    \medskip

    $\bullet$  {\bf Euler factor and Tamagawa number at $p=2$.}
  The fiber  $(W_0)_{\FF_2}$ has a singular point $p_0$ at $\infty$. Its
    normalization $\Gamma_0$ at $p_0$ is the elliptic curve $E$ over $\FF_2$ defined by the
    equation 
\[
z^2+z=u^3+u 
\]
(reduce Equation~\eqref{eq:min} mod $2$). Put $z_1=y_1/x_1^3$ and $t_1=1/x_1$. Then an  equation of $W_0$
is 
\[
z_1^3+t_1^3z_1=16+48t_1+56t_1^2+37t_1^3+14t_1^4+3t_1^5
\]
with $t_1(p_0)=0$. Following algorithm~\ref{algo:mwm}, we find that
with $t_1=2^2t_2$, $z_1=2^3z_2+4$, we have  
\[
z_2^2+(8t_2^3+1)z_2=3t_2+14t_2^2+33t_2^3+14t_2^4+48t_2^5 
\] 
which defines a new minimal Weierstrass model $W_2$ and there is
no others. The above equation mod $2$
is that of the normalization of $(W_2)_{\FF_2}$ at $p_2^*$. We find again
the elliptic curve $E$. So there are exactly two minimal Weierstrass
models $W_0, W_2$ over $\Z_{2\Z}$ (and also over $\Z$). The
reduction $\cC_{\FF_2}$ is union of two copies of $E$ 
intersecting transversely at a rational point. 
So the Tamagawa number $c_2=1$, and the Euler factor is
\[ L_2(A, s)=P(E, 2^{-s})^{-2}, \quad \text{with \ } P(E, T)=1+2T+2T^2.\]

$\bullet$ {\bf The conductor of $A$} is $11^2$ because  $A$ has
good reduction away from $11$ and multiplicative reduction at $11$. 
\medskip

$\bullet$ {\bf A volume form for $\cA$}  is 
\[
\omega= 2\ \frac{dx}{y} \wedge \frac{xdx}{y}.  
\]
Indeed let $\omega_0=(dx/2y) \wedge (xdx/2y)$ 
and  $\omega_1=2(dx_1/(2y_1+1) \wedge (x_1dx_1/(2y_1+1))$. 
Then $\omega_0$ is a basis over $\Z[1/2]$ because $W_0$ is stable
away from $2$, and $\omega_1$ a basis over $\Z_{2\Z}$ by
Proposition~\ref{prop:volume}.  
    Therefore $\omega=2^{3}\omega_0=\omega_1$  is a 
    basis over $\Z$.
\end{section}


\begin{thebibliography}{12}

\bibitem{Art} Michael Artin: {\it Lipman's proof of resolution of singularities for surfaces}, in Arithmetic Geometry, ed. Cornell and Silverman (1986), Springer. 

\bibitem{BLR} Siegfried Bosch, Werner L\"utkebohmert and Michel Raynaud, 
{\it N\'eron models},  Ergebnisse der Math., {\bf 3}. Folge, Bd. 21,
1990, Springer-Verlag.

\bibitem{BL} Siegfried Bosch, Qing Liu:  {\it Rational points of the group
  of components of a Néron model}, Manuscripta Math., {\bf 98} (1999), 275--293. 
  
\bibitem{BW} Irene Bouw, Stefan Wewers: 
  {\it Computing L-functions and semistable reduction of superelliptic curves},
  Glasgow Math. J. {\bf 59} (2017), 77--108.
  
\bibitem{DM} Pierre Deligne, David Mumford: {\it The irreducibility of the
 space of curves of given genus}, 
Publ. Math. IHES, {\bf 36} (1969), 75--109. 


\bibitem{DDMM}   Tim Dokchitser, Vladimir Dokchitser, C\'eline Maistret,
  Adam Morgan:   {\it Arithmetic of hyperelliptic curves over local fields}, 
 Math. Ann. {\bf 385} (2023), 1213--1322. 

  
\bibitem{FM} Jean Fresnel, Michel Matignon: {\it  Sur les espaces analytiques
    quasi-compacts de dimension 1 sur un corps valu\'e complet ultram\'etrique}, 
  Ann. Mat. Pura Appl. {\bf 145} (1986), 159--210. 

\bibitem{EGA_3} Alexander Grothendieck (r\'edig\'e avec la collaboration de
  Jean Dieudonn\'e) : {\'El\'ements de G\'eom\'etrie
    Alg\'ebrique},  III. \'Etude cohomologique des faisceaux coh\'erents,
  Premi\`ere partie, Publ. Math. IH\'ES, {\bf 11} (1961), 5--167. 
  
  
\bibitem{Lip} Joseph Lipman: {\it Rational singularities with applications
    to algebraic surfaces and unique factorization}, 
Publ. Math. IHES, {\bf 36} (1969), 195--279. 

\bibitem{Lcd} Qing Liu: {\it   Conducteur et discriminant minimal de courbes
    de genre 2}, Compositio Math., {\bf 94} (1994), 51--79. 
  
\bibitem{LC} Qing Liu: {\it Mod\`eles minimaux
  des courbes  de genre deux}, 
    J. f\"ur die reine und angew. Math., {\bf 453} (1994), 137--164. 
    
\bibitem{LTR} Qing Liu: {\it Mod\`eles entiers de courbes hyperelliptiques sur un anneau de valuation discr\`ete}, Trans. Amer. Math. Soc., {\bf 348} (1996),
  4577--4610.

\bibitem{LLR} Qing Liu, Dino Lorenzini, Michel Raynaud : {\it N\'eron models,
    Lie algebras, and reduction of curves of genus one}, 
    Invent. Math., {\bf 157} (2004), 455--518. 
  
  \bibitem{LB} Qing Liu: {\it Algebraic geometry and arithmetic curves},
  GTM , Oxford University Press, new edition (2006)  

\bibitem{LRN} Qing Liu: {\it Computing minimal Weierstrass equations
    of hyperelliptic curves}, Res. Number Theory {\bf 9}, 76 (2023).
  {\url{https://doi.org/10.1007/s40993-023-00483-5}}. 
    
\bibitem{Ld} Qing Liu: {\it  Desingularization of double covers of  regular surfaces}, preprint (2025) {\url{https://arxiv.org/abs/2504.16808}} 

\bibitem{Dino} Dino Lorenzini: {\it On the group of components of a Neron model},
J. Reine Angew. Math. {\bf 445} (1993), 109--160. 
  
\bibitem{MS} C\'eline Maistret, Andrew V. Sutherland: 
  {\it Computing Euler factors of genus 2 curves at odd primes of
    almost good reduction}, Res. Number Theory (2025) 11:37. 

  
\bibitem{Mus} Simone Muselli, {\it Models and integral differentials of
    hyperelliptic curves},  Glasgow Mathematical Journal, {\bf 66}
  (2024), 382--439. 
  
 \bibitem{Mil-AAV} James S. Milne: {\it On the Arithmetic of Abelian Varieties},
  Invent. Math.,  {\bf  7} (1972), 177--190.
  
\bibitem{MuSt} J. Steffen M\"uller, Michael Stoll: {\it 
Canonical heights on genus two Jacobians}, 
  Algebra and Number Theory, {\bf 10} (2016), 2153--2234.
  
\bibitem{NU} Yukihiko Namikawa and Kenji Ueno: {\it The complete
  classification of fibers in pencils of curves of genus two}, Manuscr. 
Math. {\bf 9} (1973), 143--186.

\bibitem{OS} Andrew Obus, Padmavathi Srinivasan, {\it
  Conductor-discriminant inequality for hyperelliptic curves in odd residue characteristic}, Int. Math. Res. Notices, {\bf 2024},  no. 9  (2024) 7343--7359.

  
\bibitem{Ogg} Andrew P. Ogg : {\it On pencils of curves of genus two},
  Topology {\bf 5} (1966), 355--362. 
  
 
\bibitem{pari} The PARI Group,  PARI/GP,  Univ. Bordeaux, {\url{http://pari.math.u-bordeaux.fr/}}

\bibitem{Pr} Rachel Pries: {\it Construction of covers with formal and rigid geometry}, in Courbes semi-stables et groupe fondamental en g\'eom\'etrie
  alg\'ebrique (Luminy, 1998), Birkh\"auser, Basel, 2000, pp. 157–167.  
  
\bibitem{vdP} Marius van der Put: {\it The class group of a one-dimensional
    affinoid space}. Ann. Inst. Fourier {\bf 30} (1980), 155--164. 

\bibitem{Sa} Mohamed Sa\"idi: {\it Wild ramification and a vanishing cycles formula},
J. Algebra {\bf 273} (2004), 108--128.

\bibitem{Silv} Joseph Silverman:
  {Advanced topics in the arithmetic of elliptic curves},
  Grad. Texts in Math., {\bf 151}, Springer-Verlag, New York, 1994.
  525 pp.
  
\bibitem{Tate} John Tate,  {\it Algorithm for determining the type of
    singular fiber in an elliptic pencil}, in Modular Functions of One
  Variable. IV, pages 33--52, Lecture Notes in Math., {\bf 476},
  Springer-Verlag (1975).  

\bibitem{GW} Gayn Winters, {\it On the existence of certain families
  of curves}, Amer. J.  Math. {\bf 96} (1974), pp. 215--228. 
  
\end{thebibliography}
\end{document}